\documentclass[11pt]{article}

\usepackage{latexsym}
\usepackage{amssymb}
\usepackage{amsthm}
\usepackage{amscd}
\usepackage{amsmath}
\usepackage{mathrsfs}
\usepackage{graphicx}
\usepackage{hyperref}
\usepackage{shuffle}

\usepackage[all]{xy}
\input xy \xyoption{frame}
\xyoption{dvips}

\usepackage[colorinlistoftodos]{todonotes}
\usepackage{fullpage}
\usetikzlibrary{chains,scopes,decorations.markings}

\theoremstyle{definition}
\newtheorem* {theorem*}{Theorem}
\newtheorem* {conjecture*}{Conjecture}
\newtheorem{theorem}{Theorem}[section]
\newtheorem{thmdef}[theorem]{Theorem-Definition}

\theoremstyle{definition}

\newtheorem* {example*}{Example}

\newtheorem{lemma}[theorem]{Lemma}
\theoremstyle{definition}
\newtheorem{definition}[theorem]{Definition}
\theoremstyle{definition}

\newtheorem{conjecture}[theorem]{Conjecture}
\newtheorem{proposition}[theorem]{Proposition}
\newtheorem{corollary}[theorem]{Corollary}

\newtheorem{remark}[theorem]{Remark}
\theoremstyle{definition}
\newtheorem {example}[theorem]{Example}
\theoremstyle{definition}

\theoremstyle{definition}

\theoremstyle{definition}

\theoremstyle{definition}

\def\({\left(}
\def\){\right)}

\newcommand{\CC}{\mathbb{C}}

\newcommand{\cO}{\mathcal{O}}
\newcommand{\cR}{\mathcal{R}}

\newcommand{\cI}{\mathcal{I}}

\newcommand{\cD}{\mathcal{D}}

\def\cX{\mathcal{X}}
\def\cW{\mathcal{W}}

\def\NN{\mathbb{N}}

\def\CC{\mathbb{C}}

\def\ZZ{\mathbb{Z}}

\def\GL{\mathrm{GL}}

\def\sh{\mathrm{sh}}

\def\barr{\begin{array}}
\def\earr{\end{array}}
\def\ba{\begin{aligned}}
\def\ea{\end{aligned}}
\def\be{\begin{equation}}
\def\ee{\end{equation}}

\def\Cyc{\mathrm{Cyc}}
\def\Fix{\mathrm{Fix}}
\def\Pair{\mathrm{Pair}}
\def\pair{\mathrm{pair}}

\def\qquand{\qquad\text{and}\qquad}
\def\quand{\quad\text{and}\quad}

\def\quord{\quad\text{or}\quad}

\def\inv{\operatorname{inv}}
\def\Inv{\operatorname{Inv}}

\def\I{\mathcal{I}}

\def\DesR{\operatorname{Des}_R}
\def\DesL{\operatorname{Des}_L}

\def\hs{\hspace{0.5mm}}

\def\ben{\begin{enumerate}}
\def\een{\end{enumerate}}

\def\cE{\mathcal E}

\def\hs{\hspace{0.5mm}}

\def\fpf{{\mathsf {FPF}}}

\def\Des{\operatorname{Des}}
\def\NDes{\operatorname{NDes}}
\def\NNeg{\operatorname{NNeg}}
\def\NFix{\operatorname{NFix}}

\def\ellhat{\hat\ell}

\def\a{\textbf{a}}
\def\b{\textbf{b}}

\newcommand{\xRightarrow}[2][]{\ext@arrow 0359\Rightarrowfill@{#1}{#2}}

\newcommand{\SO}{\operatorname{SO}}
\renewcommand{\O}{\operatorname{O}}
\newcommand{\Sp}{\operatorname{Sp}}

\newcommand{\rank}{\operatorname{rank}}

\newcommand{\cA}{\mathcal{A}}

\def\cAfpf{\cA_\fpf}
\def\cRfpf{\hat\cR_\fpf}
\def\iR{\hat\cR}

\newcommand{\arc}[2]{ \ar @/^#1pc/ @{-} [#2] }
\def\arcstop{\endxy\ }
\def\arcstart{\ \xy<0cm,-.06cm>\xymatrix@R=.1cm@C=.10cm }
\newcommand{\arcstartc}[1]{\ \xy<0cm,-.15cm>\xymatrix@R=.1cm@C=#1cm}

\def\ellhat{\hat\ell}

\def\neg{\operatorname{neg}}
\def\Neg{\operatorname{Neg}}
\def\sort{\operatorname{sort}}
\def\offset{\operatorname{offset}}

\def\NCSM{\textsf{NCSP}}

\def\XA{\cW_{{\tt embed}}}
\def\HA{\cA_{{\tt hecke}}}
\def\bar{\overline}

\def\scov{<\hspace{-1.6mm}\vartriangleleft_B}
\def\bcov{<\hspace{-1.6mm}\blacktriangleleft_B}

\def\X{\mathcal{X}}
\def\Y{\mathcal{Y}}
\def\Z{\mathcal{Z}}
\def\wB{\bar1\hs\bar  2\hs\bar 3\cdots  \bar n}

\def\O{\mathrm{O}}

\def\W{W^{\mathsf{BC}}}
\def\DGroup{W^{\mathsf{D}}}

\def\HA{\cA_{\mathsf{Hecke}}}
\def\cAB{\cA_{\mathsf{Brion}}}

\def\BCDphi{\iota^{\mathsf{BC}}}

\def\fpf{\mathsf{FPF}}
\def\BC{\mathsf{BC}}
\def\cAk{\cA_k}
\def\cAkfpf{\cA_k^{\fpf}}
\def\yfpf{y^{\fpf}}
\def\gfpf{g^{\fpf}}
\def\shfpf{\sh^{\fpf}}
\def\signs{\mathsf{signs}}

\numberwithin{equation}{section}
\allowdisplaybreaks[1]
\UseCrayolaColors

\makeatletter
\renewcommand{\@makefnmark}{\mbox{\textsuperscript{}}}
\makeatother

\begin{document}
\title{Atoms for signed permutations}
\author{
    Zachary Hamaker \\
    Department of Mathematics \\
    University of Florida \\
    { \tt zachary.hamaker@gmail.com}
\and
    Eric Marberg\\
    Department of Mathematics \\
    HKUST \\
    {\tt eric.marberg@gmail.com}
}

\date{}

\maketitle

\begin{abstract}
There is a natural analogue of weak Bruhat order on the involutions in any Coxeter group.
The saturated chains of intervals in this order correspond to reduced words for a certain set of group elements called atoms.
Brion gives a general formula for the cohomology class of a $K$-orbit closure in an arbitrary flag variety, where $K$ is a symmetric subgroup of a complex algebraic group.
In type A, the terms in this formula are indexed by 
atoms for permutations.
We study the combinatorics of atoms for involutions in the group of signed permutations.
In particular, we give a compact description of the atom set for any signed involution and endow it with the structure of a graded poset.
Our main result, as an application, is to identify explicitly the terms in Brion's cohomology formula in types B and C.
These descriptions apply to all $K$-orbits in these types and are the first of their kind outside of type A.
\end{abstract}


\section{Introduction}

This article is about the properties of certain elements, which we call \emph{atoms},
associated to involutions in Coxeter groups. Here is a quick definition.
Let $(W,S)$ be a Coxeter system with length function $\ell : W \to \NN =\{0,1,2,\dots\}$.
There is a unique associative product $\circ:W \times W \to W$ such that $s \circ s = s$ for $s \in S$ and $u \circ v = uv$ for $u,v \in W$ with $\ell(u) + \ell(v) = \ell(uv)$ \cite[\S7.1]{Humphreys}.
This is sometimes called the \emph{Demazure product}, while $(W,\circ)$ is the \emph{$0$-Hecke monoid} of $(W,S)$.

Write $\I(W)= \{ w \in W : w^{-1} = w\}$  for the set of involutions in $W$.
The operation $w \mapsto w^{-1} \circ w$ is a surjective map $W\to \I(W)$, and we 
let $\HA(z)=\left\{ w \in W : w^{-1} \circ w = z\right\}$ for each $z \in \I(W)$.
This set is always nonempty and we define $\cA(z)$ to be its subset of minimal length elements.
Following \cite{HMP2}, we refer to the elements of $\HA(z)$ as the \emph{Hecke atoms} of $z \in \I(W)$, while the elements of $\cA(z)$
are the \emph{atoms} of $z$.
When $W$ is finite, the atoms of $z \in \I(W)$ are precisely the minimal 
length elements $w \in W $ with $wz \leq w$ in (strong) Bruhat order \cite[Theorem 4.12]{HMP2}.


\subsection{Motivation}\label{intro2-sect}

One reason why
atoms are interesting objects is that their reduced expressions correspond to saturated chains in the \emph{involution weak (Bruhat) order} introduced by Richardson and Springer \cite{RichSpring}.
This is the partial order on $\I(W)$ with $x \leq_\I y$ if there exists $w \in W$ such that $w^{-1} \circ x \circ w = y$.

The involution weak order is relevant to the study of spherical varieties.
Let $G$ be a complex algebraic group with Borel subgroup $B$, and let $\theta$ be  a self-inverse automorphism of $G$ that preserves $B$.
 The \emph{symmetric subgroup} $K = \{g \in G: \theta(g) = g\}$ acts with finitely many orbits on the flag variety $G/B$.
The opposite Borel subgroup also acts on $G/B$ with finitely many orbits, whose closures are the Schubert varieties $X_w$ indexed by the elements of the Weyl group $W$.

Fix a $K$-orbit $\cO$ in $G/B$.
Brion gives a general formula describing the cohomology class of the closure of $\cO$ in terms of Schubert varieties in \cite{Brion2001}.
We can express this formula as
\begin{equation}
\label{eq:Brion}
[\overline{\cO}] = \sum_{w \in \cAB(\cO)} 2^{d_\cO(w)}[X_w] \in H^*(G/B)	
\end{equation}
where the sum is over a certain set $\cAB(\cO) \subset W$ and each $d_\cO(w)$ is a nonnegative integer.
The precise definitions of $\cAB(\cO)$ and $d_\cO(w)$ involve a partial order on the $K$-orbits in $G/B$
that lifts the involution weak order.
We review this material
 in Section~\ref{geom-sect}.
 
For example, if $G = \GL(n,\CC)$ and $\theta(g) = (g^{-1})^T$, then we have $K = \O(n,\CC)$ and $W = S_n$.
In this case, the $K$-orbit closures $\cO_z$ are indexed by $z \in \I(S_n)$, the integer $d_{\cO_z}(w)$ is the number of $2$-cycles in $z$ for all $w$,
and the set $\cAB(\cO_z)$ is precisely the set of atoms $\cA(z)$.
One goal of this paper is to clarify the extent to which this parallel between\
$\cAB(\cO)$ and $\cA(z)$ holds outside type A. We focus on the classical types B and C, when $G = \Sp(2n,\CC)$ or $\SO(n,\CC)$.

Other motivations for studying atoms come from their enumerative properties.
Several authors have noted remarkable features of atoms for involutions in the symmetric group.
For example, the inverse Hecke atoms $\HA(z)^{-1} = \{ w^{-1} : w \in \HA(z)\}$ are precisely the equivalence classes in $S_n$ 
under the so-called \emph{Chinese relation} from~\cite{CEHKN,DuchampKrob}. 
There are also explicit combinatorial descriptions of $\cA(z)$ for $z \in \I(S_n)$ \cite{CJ,CJW},
which are used in  \cite{HMP2} to show that  $\cA(z)$ is naturally a bounded, graded poset (and conjecturally, a lattice).
Another goal of the present work is to generalize these descriptions to atoms for signed permutations.
Some of our theorems in this direction can be viewed as strengthening
results of Hu and Zhang,
who prove an analogue of the Matsumoto-Tits theorem for the reduced expressions of atoms for signed permutations in  \cite{HuZhang1,HuZhang2}.

\subsection{Outline of results}
\label{introResults}

Fix a positive integer $n$ and let $[n]= \{1,2,\dots,n\}$.
Let $[\pm n] = [n] \sqcup -[n]$
and write $\W_n$ for the group of bijections $w : [\pm n] \to [\pm n]$ satisfying $w(-i) = -w(i)$ for each $i \in [n]$.
We refer to the elements of $\W_n$ as signed permutations.
Define 
\[t_0 = (-1,1)  \in \W_n
\quand
 t_i = (-i,-i-1)(i,i+1) \in \W_n\text{ for $i \in [n-1]$.}\]
With respect to these generators, 
$\W_n$ is the rank $n$ Coxeter group of types B and C. 
%
Throughout, we use the term \emph{word} to mean a finite sequence of integers.
 The \emph{one-line representation} of a permutation $w$ in $S_n$ or $\W_n$
 is the word $w_1w_2\cdots w_n$ where $w_i=w(i)$.
We usually write $\bar{m}$  in place of $-m$ so that, for example, 
the 8 elements of $\W_2$ 
are
$ 12,$ $\bar 1 2,$  $1 \bar 2$, $\bar 1\hs \bar 2$, $2 1$,  $\bar 2 1$, $ 2 \bar 1$, and $\bar 2\hs \bar1$. 

Define $\vartriangleleft_A$ to be the relation on $n$-letter words with
\be\label{<A-eq} \cdots cab\cdots \vartriangleleft_A \cdots bca \cdots\ee
whenever $a<b<c$ and the corresponding ellipses ``$\cdots$'' mask identical subsequences.
We apply $\vartriangleleft_A$ to permutations via their one-line representations,
and  
refer to the transitive closure $<_A$ of $\vartriangleleft_A$ as the \emph{atomic order} of type A.

A poset $(P,<)$ is \emph{graded} if there exists a function $\rank : P \to\NN$ 
with $\rank(y) = \rank(x) + 1$ whenever $x<y$ is a covering relation.
A poset is \emph{bounded} if it has a unique minimum and a unique maximum.
Results in \cite[\S6.1]{HMP2} show that if $z \in \I(S_n)$ then $\cA(z)^{-1}=\{ w^{-1} : w\in \cA(z)\}$ is a bounded, graded poset under $<_A$. 
%
%
%
The sets $\cA(z)^{-1}$ were first studied in \cite{CJ,CJW}, where they
are denoted $\cW(z)$.
%
As will be reviewed in Section~\ref{nest-sect},  there is an explicit construction for the minimal and maximal elements in these posets, as well as 
a simple algorithm to recover $z$ from the one-line representation of any $w \in \cA(z)^{-1}$.
Such properties play a key technical role in \cite{HMP3,HMP4,HMP5}.

This paper introduces signed analogues of the partial order $<_A$.
To sketch our main results, we write $\vartriangleleft_B$ for the  relation on $n$-letter words with 
\be\label{<B-eq} \bar b \bar a \cdots \vartriangleleft_B a \bar b \cdots 
\qquand
\bar c \bar b \bar a  \cdots \vartriangleleft_B \bar c  a  \bar b \cdots
\ee
whenever $0<a<b<c$ and the corresponding ellipses mask identical subwords.
Unlike $\vartriangleleft_A$, this relation only changes the letters at the start of a word.
The \emph{(weak) atomic order} $<_{B}$ of type B is the transitive closure of \emph{both} $\vartriangleleft_A$ and $\vartriangleleft_B$.
We apply these relations on words to elements of $\W_n$ via their one-line representations.
%

Fix $z \in \I(\W_n)$. We will show that  $\cA(z)^{-1} = \{ w \in \W_n : w^{-1} \in \cA(z)\}$ is preserved by
the relations $\vartriangleleft_A$ and $\vartriangleleft_B$,
in the sense that if $v,w \in \W_n$ are such that $v\vartriangleleft_A w$ (respectively, $v \vartriangleleft_B w$),
then $v \in \cA(z)^{-1}$ if and only if $w \in \cA(z)^{-1}$; see Lemma~\ref{<B-lem}.
%
%
%
Let $\Neg(z) = \{ i \in [n] : z(i) = -i\}$. The following combines Theorems~\ref{nc-thm} and \ref{minmax-thm}:

\begin{theorem}\label{intro-thm2}
If $m = |\Neg(z)|$, then $(\cA(z)^{-1},<_A)$ has $\binom{m}{\lfloor m/2\rfloor}$ connected components.
These components are in bijection with the perfect matchings in  $\Neg(z) \sqcup -\Neg(z)$ 
that are noncrossing and symmetric with respect to negation.
Moreover, each component in $(\cA(z)^{-1},<_A)$ is isomorphic to $(\cA(\zeta)^{-1},<_A)$ for some $\zeta \in \I(S_n)$.
\end{theorem}

The second author and Pawlowski \cite{MP} use 
this classification to prove a
formula conjectured in \cite{HMP1} for the number of maximal chains in $(\I(\W_n),<_\I)$.

The following holds by Corollaries~\ref{hz-cor1}, \ref{cat-cor}, and \ref{rank-cor}. Let $k= \left\lceil \tfrac{1}{2} |\Neg(z)|\right\rceil$.

\begin{theorem}\label{intro-thm3}
The poset $(\cA(z)^{-1},<_B)$ is graded and connected  with $\frac{1}{k+1}\binom{2k}{k}$ maximal elements.
\end{theorem}

Our proof of these results   will establish more than what is stated here.
In particular, we will completely describe the minimal and maximal elements in $(\cA(z)^{-1},<_A)$ for $z \in \I(\W_n)$,
and give an explicit construction of the bijection mentioned in Theorem~\ref{intro-thm2}.
The assertion that $(\cA(z)^{-1},<_B)$ is connected  for all $z \in \I(\W_n)$ is equivalent to a result of Hu and Zhang \cite[Theorem 4.8]{HuZhang2}.
Our methods 
provide an alternate proof of Hu and Zhang's theorem.

\begin{example}\label{intro-ex}
The Hasse diagram of $(\cA(z)^{-1},<_B)$
for $z =\overline{1}\hs \overline{2} \hs\overline{3}\hs \overline{4} \in \I(\W_4)$ is
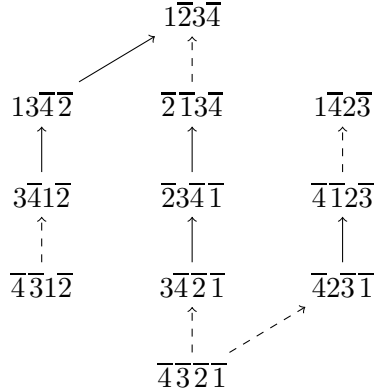
\begin{center}
\begin{tikzpicture}[scale=0.9]
\node (0) at (2,7.2) {$1\bar 2 3 \bar 4$};  
\node (1a) at (0,6) {$13\bar4\hs\bar2$};  
\node (1b) at (2,6) {$\bar2\hs\bar13\bar4$};  
\node (1c) at (4,6) {$1\bar42\bar3$};  
\node (2a) at (0,4.8) {$3\bar41\bar2$};  
\node (2b) at (2,4.8) {$\bar23\bar4\hs\bar1$};  
\node (2c) at (4,4.8) {$\bar4\hs\bar12\bar3$};  
\node (3a) at (0,3.6) {$\bar4\hs\bar31\bar2$};  
\node (3b) at (2,3.6) {$3\bar4\hs\bar2\hs\bar1$};  
\node (3c) at (4,3.6) {$\bar42\bar3\hs\bar1$};  
\node (4a) at (2,2.4) {$\bar4\hs\bar3\hs\bar2\hs\bar1$};  
\draw  [->,dashed] (4a) -- (3b);
\draw  [->,dashed] (4a) -- (3c);
\draw  [->,dashed] (3a) -- (2a);
\draw  [->] (3b) -- (2b);
\draw  [->] (3c) -- (2c);
\draw  [->] (2a) -- (1a);
\draw  [->] (2b) -- (1b);
\draw  [->,dashed] (2c) -- (1c);
\draw  [->] (1a) -- (0); 
\draw  [->,dashed] (1b) -- (0); 
\end{tikzpicture}
\end{center}
The solid arrows indicate the relations $\vartriangleleft_A$
while the dashed arrows indicate $\vartriangleleft_B$.
The six minimal elements relative to $<_A$ are in bijection with the six perfect noncrossing symmetric matchings in 
the set $[\pm 4]$ via the following correspondence, which is described in general in Section~\ref{shapes-sect}:
\[
\ba
\\[-12pt]
\bar4\hs\bar3\hs\bar2\hs\bar1\quad&\leftrightarrow\quad \arcstart
{
*{\bullet}  \arc{1.2}{rrrrrrr}    & *{\bullet}   \arc{0.9}{rrrrr}    & *{\bullet} \arc{0.6}{rrr}     & *{\bullet}   \arc{0.3}{r}     & *{\bullet} & *{\bullet}    & *{\bullet}   & *{\bullet}    
} 
\arcstop
\\[-12pt]\\
3\bar4\hs\bar2\hs\bar1 \quad&\leftrightarrow\quad \arcstart
{
*{\bullet}  \arc{0.4}{r}    & *{\bullet}      & *{\bullet}   \arc{0.8}{rrr}  & *{\bullet}    \arc{0.4}{r}    & *{\bullet} & *{\bullet}   & *{\bullet} \arc{0.4}{r}   & *{\bullet}    
} 
\arcstop
\\[-12pt]\\
1\bar42\bar3\quad&\leftrightarrow\quad \arcstart
{
*{\bullet}  \arc{0.8}{rrr}    & *{\bullet}   \arc{0.4}{r}    & *{\bullet}    & *{\bullet}     & *{\bullet} \arc{0.8}{rrr}  & *{\bullet} \arc{0.4}{r}   & *{\bullet}   & *{\bullet}    
} 
\arcstop
\ea
\qquad\quad
\ba
\\[-12pt]
\bar4\hs\bar31\bar2\quad&\leftrightarrow\quad \arcstart
{
*{\bullet}  \arc{1.2}{rrrrrrr}    & *{\bullet}   \arc{0.9}{rrrrr}    & *{\bullet}   \arc{0.4}{r}  & *{\bullet}       & *{\bullet}  \arc{0.4}{r} & *{\bullet}    & *{\bullet}   & *{\bullet}    
} 
\arcstop
\\[-12pt]\\
 \bar42\bar3\hs\bar1\quad&\leftrightarrow\quad\arcstart
{
*{\bullet}  \arc{1.2}{rrrrrrr}    & *{\bullet}   \arc{0.4}{r}    & *{\bullet}    & *{\bullet}   \arc{0.4}{r}     & *{\bullet} & *{\bullet} \arc{0.4}{r}   & *{\bullet}   & *{\bullet}    
} 
\arcstop
\\[-12pt]\\
3\bar41\bar2\quad&\leftrightarrow\quad\arcstart
{
*{\bullet}  \arc{0.4}{r}    & *{\bullet}      & *{\bullet}   \arc{0.4}{r}  & *{\bullet}     & *{\bullet}  \arc{0.4}{r}  & *{\bullet}   & *{\bullet} \arc{0.4}{r}   & *{\bullet}    
} 
\arcstop
\ea
\]
In this case, the matching of an atom should be viewed as describing which fixed points combine to form 2-cycles in $z$, a perspective that is made precise in Section~\ref{relshape-sect}.
\end{example}

%

As noted above, our results also have applications to Brion's cohomology formula \eqref{eq:Brion}. Let $p,q \in \NN$ with $p+q=n$.
There are three families of symmetric subgroups $K\subset G$ for which the Weyl group is $W=\W_n$:
\ben
\item[(BI)]  $G=\SO(2n+1,\CC)\supset K =S(\O(2p,\CC)\times \O(2q+1,\CC))$.
\item[(CI)] $G=\Sp(2n,\CC)\supset K= \GL(n,\CC)$.
\item[(CII)] $G=\Sp(2n,\CC)\supset K= \Sp(2p,\CC)\times \Sp(2q,\CC)$.
\een
For precise versions of  the following informal statement, see Theorems~\ref{bt1}, \ref{bt11}, and \ref{bt2}.

\begin{theorem}
\label{intro-Brion}
Suppose $(G,K)$ has type (BI) with $p-q\in\{0,1\}$ or type (CI).
Then for each $K$-orbit $\cO$ in $G/B$, the set $\cAB(\cO)^{-1}$ 
is the  union of
an explicit subset of connected components in $(\cA(z)^{-1},<_A)$ for some $z \in \I(\W_n)$.
There are similarly explicit descriptions of $\cAB(\cO)$ 
in type (BI) with $p-q\notin \{0,1\}$ and type (CII), involving minor generalizations
of the sets $\cA(z)$.
\end{theorem}

The extra notation needed to handle types  (BI) with $p-q\notin \{0,1\}$ and (CII)
is given in Section~\ref{relshape-sect}.
The orbit closures  described in Theorem~\ref{intro-Brion} are indexed by combinatorial objects called \emph{clans}, defined in Section~\ref{clans-sect}.
Theorem~\ref{intro-Brion} gives the first efficient way of computing Brion's formula when $G = \SO(n,\CC)$ or $G = \Sp(2n,\CC)$.
It also provides the first non-recursive description of the terms in Brion's formula when the $K$-orbits in $G/B$ are indexed by clans.

Here is an outline of the paper.
Section~\ref{embed-sect} gives some preliminaries.
We spend Sections~\ref{nest-sect}--\ref{geom-sect}
proving Theorems~\ref{intro-thm2} and \ref{intro-thm3} and various related results.
The proof of Theorem~\ref{intro-Brion} occupies
Sections \ref{relshape-sect} and \ref{geom-sect}.
We briefly discuss the sets $\HA(z)$ in Section~\ref{hecke-sect}.
In Section~\ref{atomic-sect},
we enumerate the elements $z \in \I(W_n)$ with $|\cA(z)| = 1$.
Appendix~\ref{not-sect} contains an index of symbols.

\subsection*{Acknowledgements}

This work was partially supported by HKUST grant IGN16SC11
and Hong Kong RGC Grant ECS 26305218.
We thank
Alex Hultman and Mikael Hansson
for several helpful comments, and Brendan Pawlowski for useful conversations. 
We also thank the anonymous referees for useful suggestions.

\section{Preliminaries}
\label{embed-sect}

We write $\ZZ$ for the integers, $\NN$ for the nonnegative integers, and set $[n] = \{1,2,\dots,n\}$ for $n \in \NN$.
Let $(W,S)$ be a Coxeter system
with length function $\ell : W\to \NN$ and Demazure product $\circ$, as described at the start of the introduction.
  Again let $ \I(W)= \{ w \in W : w^2=1\}$.

  \subsection{Demazure products}\label{dem-subsect}
  
  We include a few remarks about how to compute with $\circ$.
  If $w \in W$ and $s \in S$
  then $w\circ s$ is  either $w$ or $ws$,
while $s\circ w$ is either $w$ or $sw$.
If $w \in W$ and $w=s_1s_2\cdots s_l$ is a reduced word then $w=s_1\circ s_2 \circ \dots \circ s_l$.
If
$s \in S$ and $z \in \I(W)$ have $\ell(szs) = \ell(z)$ then $szs=z$ \cite[Lemma 3.4]{H2}, so 
\be\label{szs-eq}
 s\circ z \circ s = \begin{cases} szs &\text{if $zs\neq sz$ and $\ell(zs) > \ell(z)$}\\
 zs &\text{if $zs=sz$ and $\ell(zs) > \ell(z)$} \\
 z &\text{if $\ell(zs) < \ell(z)$}.
 \end{cases}
 \ee
 This identity is useful for understanding 
the covering relations in the involution weak order  $<_\I$, which all have the form $z<_\I s \circ z \circ s$ for $z \in \I(W)$ and $s \in S$ with $\ell(zs)>\ell(z)$.

 %

It follows by an inductive argument that the set $\HA(z) = \{ w \in W : w^{-1} \circ w = z\}$ 
is nonempty for all $z \in \I(W)$.
The set $\cA(z)$ of minimal length elements of $\HA(z)$ is therefore also nonempty for all $z \in \I(W)$.
Denote the left and right descent sets of $w \in W$ by
$\DesL(w) = \{ s \in S : \ell(sw) < \ell(w)\}$ and $\DesR(w) = \{ s \in S : \ell(ws) < \ell(w)\}.$
A finite Coxeter group has a unique longest element $w_0$ with $\DesL(w_0)=\DesR(w_0)=S$.
In $S_n$ one has $w_0= n\cdots 321$, while in $\W_n$ one has
 $w_0= \wB$.
Write $<_L$ and $<_R$ for the left and right weak orders on $W$, which are the partial orders with covering relations $sw \lessdot_L w$ 
and $wt \lessdot_R w$ for
all $w\in W$, $s \in \DesL(w)$, and $t \in \DesR(w)$. 
 
\begin{proposition}\label{weak-prop}
 Suppose $z \in \I(W)$ and $w \in \cA(z)$.
 Then  $\DesR(w) \subset \DesR(z)$.
 Moreover, if $v \in W$ has $v<_Rw$ then  $v \in \cA(y)$ for some $y \in \I(W)$.
 \end{proposition}
 
 \begin{proof}
We have $z =vw$ for some $v \in W$ with $\ell(z) = \ell(v) + \ell(w)$.
  \end{proof}

The \emph{involution length function} 
 $\ellhat :\I(W) \to \NN$ is the map that assigns to 
 $z \in \I(W)$ the common value of $\ell(w)$ for $w \in \cA(z)$.
The \emph{absolute length function} 
 $\ell' : W \to \NN$
 is the map that assigns to  $w \in W$ the minimum number of factors $l$
needed to express $w=t_1t_2\cdots t_l$ as a product of {reflections} $t_i \in T:= \{ wsw^{-1} : (w,s) \in W\times S\}$.
It holds that $\ellhat(z) = \tfrac{1}{2}(\ell(z) + \ell'(z))$ for $z \in \I(W)$ \cite{H1}.
Thus 
\be\label{prime-eq}
\ell'(s\circ z \circ s) =\begin{cases} \ell'(z) + 1 & \text{if $zs=sz$ and $\ell(zs) > \ell(z)$} \\
\ell'(z) &\text{otherwise}
\end{cases}
\quad\text{for }s \in S,\ z \in \I(W).
\ee



\subsection{Signed permutations}\label{sogmed-sect}

Continue to let $s_i = (i,i+1) \in S_n$ for $i \in [n-1]$,
so that
 $S_n=\langle s_1,s_2,\dots,s_{n-1}\rangle $ is the rank $n-1$  Coxeter group of type A.
Write $\Inv(w)$ for
 the set  
of \emph{inversions} of $w \in S_n$ in $[n]$, that is, pairs $(i,j) \in [n]\times [n]$ with $i<j$ and $w(i) > w(j)$, and define $\inv(w) = |\Inv(w)|$.
It is well-known that if  $w \in S_n$ then $\ell(w) = \inv(w)$ and  $s_i \in \DesR(w)$ if and only if $w(i) > w(i+1)$.

The reflections in $S_n$ are the transpositions $(i,j)$ for $i\neq j$ in $[n]$.
A \emph{cycle} of  $w \in S_n$ is an orbit in $[n]$ under the action of the cyclic group $\langle w\rangle$.
 If $w \in S_n$ has $k$ cycles in $[n]$ then $\ell'(w)=n-k$, and 
if $z \in \I(S_n)$ then $\ell'(z)$ is the number of \emph{2-cycles} of $z$, that is, cycles of size two.

If  $w \in S_n$ and $i \in[n-1]$ then the Demazure product for $S_n$ satisfies $w\circ s_i = ws_i$ if and only if $w(i) < w(i+1)$.
The operation $z\mapsto s_i \circ z \circ s_i$ for an involution $z \in \I(S_n)$ has the following interpretation.
The cycles of $z$ all have size at most two, and we visualize $z$ as a matching in $[n]$.
If $i$ and $i+1$ are isolated points, then $s_i\circ z\circ s_i=zs_i$ is formed by adding the new edge $\{i,i+1\}$.
Otherwise, if $z(i) < z(i+1)$, then we form $s_i\circ z \circ s_i=s_izs_i$ by interchanging vertices $i$ and $i+1$. 

\begin{example}
The atoms of $321 = s_2\circ s_1\circ s_1 \circ s_2 = s_1\circ s_2\circ s_2 \circ s_1 \in \I(S_3)$ are $231 = s_1s_2$ and $312 = s_2s_1$, while the Hecke atoms are
these elements plus $321=s_1s_2s_1=s_2s_1s_2$.
\end{example}

As in Section~\ref{introResults}, let
$t_0 = (-1,1)$
 and
$
 t_i = (-i,-i-1)(i,i+1)$ for $i \in [n-1]$.
With respect to these generators, 
$\W_n = \langle t_0,t_1,\dots,t_{n-1}\rangle$ is the rank $n$ Coxeter group of types B and C.
 If $\sigma =\sigma_1\sigma_2\cdots \sigma_n \in \W_n$ then  $\sigma t_0 = \bar{\sigma_1}\sigma_2\cdots \sigma_n$
 and
 $\sigma t_i = \sigma_1\cdots \sigma_{i+1}\sigma_i \cdots \sigma_n$.
One can show that  $t_0$ appears exactly $\ell_0(\sigma):= |\{ i \in [n]: \sigma(i) < 0\}|$ times in any reduced word for $\sigma \in \W_n$,
and that
  $\ell(\sigma) = \frac{1}{2} (\inv(\sigma) +\ell_0(\sigma))$ where
 $\inv(\sigma)$ denotes the number of inversions of $\sigma$ in the set $[\pm n]$.
 If one defines $\sigma_0 =0$ 
 then $t_i \in \DesR(\sigma)$ 
 if and only if $\sigma_i > \sigma_{i+1}$ for all $i \in \{0\} \sqcup [n-1]$

The reflections in $\W_n$ are the elements
$ s_{ii} = (i,-i),$ $ s_{ij}=(i,-j)(j,-i), $ and $ t_{ij}=(i,j)(-i,-j)$
for $i,j \in [n]$ with $i\neq j$.
Let $\ell'_0(\sigma) $ denote the number of cycles of $\sigma \in \W_n$ in $[\pm n] = [n] \sqcup -[n]$
that are preserved by the negation map.
One can show that $\ell'(\sigma) = n - \frac{1}{2}(k - \ell'_0(\sigma))$ where $k$ is the number of cycles of $\sigma$ in $[\pm n]$.
Let $z \in \I(\W_n)$. Define $\Pair(z) = \{ (a,b) \in [\pm n]\times [n] : |a| < z(a) = b\}$ and $\pair(z) = |\Pair(z)|$,
and let $\neg(z) = |\Neg(z)|$ where $\Neg(z) = \{i \in [n] : z(i) = -i\}$.
Then
\be\label{ell'} 
\ell'(z) = \neg(z) + \pair(z)\qquad\text{for }z \in \I(\W_n).
\ee
 Observe that if $z \in \I(\W_n)$ then $\ell'_0(z) = \neg(z)$.

Let $z \in \I(\W_n)$ and consider the symmetric matching on $[\pm n]$ whose edges are the cycles of $z$. 
The operation $z \mapsto t_i \circ z \circ t_i$  may be described in terms of this matching as follows.
 If $- 1$ and $1$ are isolated points, then $t_0\circ z \circ t_0$ is formed by adding the edge $\{- 1, 1\}$,
 while if $z(-1) < z(1)$ then $t_0\circ z \circ t_0$ is formed by interchanging vertices $- 1$ and $1$.
 Assume $i \in [n-1]$ and $z(i) <z(i+1)$. There are three possibilities.
 We obtain $t_i \circ z \circ t_i$ by adding the edges $\{i,i+1\}$ and $\{-i,-i-1\}$ when $i$ and $i+1$ are isolated points,
 by interchanging the vertices $i$ and $i+1$ when $z(i+1)=-i$,
  or else by interchanging vertices $i$ and $i+1$ and then also $-i$ and $-i-1$.
 
\begin{example}
The permutations $\bar{2}\hs\bar{1}=t_0t_1t_0$ and $1\bar 2=t_1t_0t_1$ belong to $\cA(\bar{1}\hs\bar{2})$,
while $\bar{3}\hs\bar{2}\hs\bar{1}=t_0t_1t_2t_0t_1t_0$ and $2\bar 3\hs \bar 1 = t_0t_1t_2t_1t_0t_1$ 
belong to $\cA(\bar{1}\hs\bar{2}\hs\bar{3})$, and $\bar2 \hs \bar3\hs \bar 1 = t_0t_1t_2t_0t_1t_0t_1 \in \HA(\bar1 \hs \bar2 \hs \bar3)$.
\end{example}


\subsection{Embedding $\W_n \hookrightarrow S_{2n}$}\label{embed-subsect}

There is a useful embedding of $\W_n$ in $S_{2n}$.
Define
$\Psi_n : \W_n \to S_{2n}$
 by 
 $\Psi_n(\sigma) = \psi \circ \sigma \circ \psi^{-1}$ for $\sigma \in \W_n$,
where
$\psi$ is the order-preserving bijection $[\pm n] \to [2n]$.
The map $\Psi_n $ is an injective group homomorphism.
Moreover, $\Psi_n$ is the unique monoid homomorphism $(\W_n,\circ) \to (S_{2n},\circ)$
 under which 
$ t_0 \mapsto s_n$ and $t_i \mapsto s_{n+i}s_{n-i}$ for $i \in [n-1]$.
As a consequence, we have
\be
\label{ell-prop}
\ell\(\Psi_n(\sigma)\)  = 2\ell(\sigma) - \ell_0(\sigma)
\qquand
\ell'\(\Psi_n(\sigma)\)  = 2\ell'(\sigma) - \ell'_0(\sigma)
\qquad\text{for $\sigma \in \W_n$.}\ee
As a group homomorphism, $\Psi_n$ restricts
 to a map
 $\I(\W_n) \to \I(S_{2n})$, so we have 
\begin{equation}
\label{ellhat-cor}
 \ellhat\(\Psi_n(z)\) = 2 \ellhat(z) - \tfrac{1}{2}\(\ell_0(z) + \neg(z)\)
 \qquad\text{for $z \in \I(\W_n)$}.
\end{equation}
 
\begin{lemma}\label{equiv-lem}
Suppose $\sigma \in \W_n$ and $z \in \I(\W_n)$.
Then 
$\sigma \in \cA(z)$
if and only if
it holds that $\Psi_n(\sigma) \in \HA\(\Psi_n(z)\)$ and  $\ell\(\Psi_n(\sigma)\) -  \ell\(\Psi_n(z)\)  = \frac{1}{2}\(\ell_0(z)+\neg(z)\) -  \ell_0(\sigma)$.
\end{lemma}

\begin{proof}
Since $\Psi_n$ is injective,  $\sigma^{-1}\circ \sigma =z$
if and only if $\Psi_n(\sigma)^{-1}\circ \Psi_n(\sigma)= \Psi_n(z)$.
By \eqref{ell-prop} and \eqref{ellhat-cor}, we have
 $\ell(\sigma) = \ellhat(z)$ if and only if  the given length condition holds.
\end{proof}

\begin{remark}\label{gen-rmk}
There is also an embedding 
$S_n \hookrightarrow \W_n$ with $s_i \mapsto t_i$ for $i \in [n-1]$, which defines  
a homomorphism of monoids $(S_n,\circ) \to (\W_n,\circ)$, but this will be less useful in our applications.
\end{remark}

\section{Nested descents}
\label{nest-sect}

A \emph{word} is a finite sequence of integers.
A \emph{(one-line) descent} of a word $w=w_1w_2\cdots w_n$ is a pair $(w_i,w_{i+1})$ with $w_i > w_{i+1}$.
Let $\Des(w)$ be the set of descents of $w$.
A \emph{subword} of $w$ is any (not necessarily consecutive) subsequence.
Define $\sort_L(w)$ (respectively, $\sort_R(w)$) to be the subword of $w$ formed by omitting
$w_{i+1}$ (respectively, $w_{i}$) whenever $(w_i,w_{i+1}) \in \Des(w)$.
Adapt these definitions to elements of $S_n$ or $\W_n$ by identifying permutations with their one-line representations.
For example,
if $w=2134765 \in S_7$ then $\Des(w) = \{ (2,1), (7,6), (6,5)\}$, so
$ \sort_L(w) = 2347$ and $ \sort_R(w) = 1345$.
On the other hand, if $\sigma = \bar 2 1 \bar 3 4 \bar 7 6 \bar 5 \in \W_7$ 
then $\Des(\sigma) =\{(1,-3),(4,-7),(6,-5)\}$, so
$\sort_L(\sigma) = \bar {2}146$ and
$\sort_R(\sigma) = \bar2\hs\bar3\hs\bar7\hs\bar5.$

Suppose $X$ is a set of $n$ integers $x_1<x_2<\dots<x_n$.
Let $S_X$ denote the group of permutations of $X$, viewed as a Coxeter group
relative to the generators $(x_i,x_{i+1})$ for $i \in[n-1]$.
The \emph{one-line representation} of $\sigma \in S_X$ is the word $\sigma(x_1)\sigma(x_2)\dots \sigma(x_n)$.
As a Coxeter group, $S_X$ has a Demazure product $\circ$, which gives us a set of atoms $\cA(z)$
for each involution $z \in \I(S_X)$. 
For example, if $X = \{1,3,5\}$ then $531 \in \I(S_{X})$ are $\cA(531) = \{ 513,351\}$.

If $w=w_1w_2\cdots w_n$ is a word, then we write $[[w]]$ for the subword 
formed by omitting each repeated letter after its first appearance, going left to right.
For $z \in \I(S_X)$, define  \[\Cyc_A(z) := \{ (a,b) \in X\times X : a \leq b = z(a)\}.\]
Suppose we have
$\Cyc_A(z) = \{(a_1,b_1),(a_2,b_2),\dots,(a_l,b_l)\}=\{(c_1,d_1),(c_2,d_2),\dots,(c_l,d_l)\}$
where 
$a_1<a_2<\dots<a_l $ and $ d_1<d_2<\dots<d_l$.
We define the permutations $0_A(z),1_A(z) \in S_X$  by
\be\label{01A-eq}
0_A(z) := [[ b_1a_1b_2a_2\cdots b_la_l]] 
\qquand
 1_A(z) := [[d_1c_1d_2c_2\cdots d_lc_l]].
 \ee
 Alternatively,   $0_A(z) $ and $1_A(z)$ are the unique elements of $S_X$ 
for which  $\sort_R( 0_A(z) )$ and $\sort_L(1_A(z))$ are increasing and
  $\Des(0_A(z))=\Des(1_A(z))= \{ (b,a) \in X\times X : a < b = z(a)\}$.
Thus if $z=(1,2)(4,7)(5,6)  \in \I(S_7)$ then $0_A(z) = 2137465$ and $1_A(z) = 2136574$,
while if $z = (1,2)(4,7)(5,6) \in \I\(S_{\{1,2,4,5,6,7\}}\)$ then $0_A(z) =  217465$ and $1_A(z) =216574$.

Let $<_A$ be the transitive closure of $\vartriangleleft_A$ from \eqref{<A-eq}.
Write $\sim_A$ for the symmetric closure of $<_A$.

\begin{theorem}[See \cite{HMP2}] \label{0-z-thm}
Suppose $X\subset \ZZ$ is a finite set and $z \in \I(S_X)$. Then 
\[\cA(z)^{-1} = \left\{ w \in S_X : 0_A(z) \leq _A w \right\} = \left\{ w \in S_X : w \leq_A 1_A(z)\right\}.\]
\end{theorem}

\begin{proof}
It suffices to assume $X=[n]$; the result is then \cite[Theorem 6.10 and Proposition 6.14]{HMP2}.
\end{proof}

\begin{remark}
To generate the inverse atoms of $z \in \I(S_n)$,
read off $0_A(z)$ and $1_A(z)$ from the cycle structure of $z$, then find all elements spanned from these
by the covering
relation $\vartriangleleft_A$. 
\end{remark}

We can apply $\vartriangleleft_A$, $<_A$, and $\sim_A$ to signed permutations in one-line notation.

\begin{lemma}\label{<A-lem}
If $w \in \W_n$, $z \in \I(\W_n)$, $v \in \cA(z)^{-1}$, and 
$v \sim_A w$, then 
$w \in \cA(z)^{-1}$.
\end{lemma}

\begin{proof}
If $v,w \in \W_n$
 then $v \vartriangleleft_A w$ iff $v = ut_{i+1} t_{i} $ and $w = ut_{i} t_{i+1} $ for $(i,u)\in[n-1]\times W$
with  $\ell(v) = \ell(w) = \ell(u)+2$.
The lemma follows as $(t_{i+1}t_{i})^{-1} \circ (t_{i+1}t_{i}) = (t_{i}t_{i+1})^{-1} \circ (t_{i}t_{i+1}) = t_it_{i+1}t_i$.
\end{proof}

We will say that a word $w_1w_2\cdots w_n$ has a \emph{consecutive 321-pattern} 
if for some $i \in [n-2]$ it holds that $w_iw_{i+1}w_{i+2} = cba$ where $a<b<c$. Define \emph{consecutive $312$-} and \emph{$231$-patterns} similarly.
A  permutation in $S_n$ or $ \W_n$ has a consecutive 321-pattern if its one-line representation does.

\begin{proposition}\label{321-prop}
If $w  \in \cA(z)^{-1}$ for $z$ in $\I(S_n)$ or $\I(\W_n)$, then $w$ has no consecutive 321-patterns.
\end{proposition}

\begin{proof}
If $w \in \W_n$ has $w(i) > w(i+1) > w(i+2)$, then we can write $w = v t_it_{i+1}t_i = v t_{i+1}t_i t_{i+1}$ where $v \in \W_n$ has $\ell(w) = \ell(v) + 3$,
in which case $\ell(wt_i) < \ell(w)$ and $w^{-1}\circ w = (wt_i)^{-1} \circ (wt_i)$, so $w^{-1}$ is not an atom of any involution.
The same conclusion holds when $w \in S_n$ by Remark~\ref{gen-rmk}.
\end{proof}

\begin{lemma}\label{minmax-lem}
Assume a word $w$ has no consecutive 321-patterns.
Then $w$ is minimal (respectively, maximal) relative to $<_A$ if and only if
$\sort_R(w)$ (respectively, $\sort_L(w)$) is increasing.
\end{lemma}

\begin{proof}
The word $\sort_R(w)$ (respectively, $\sort_L(w)$) fails to be increasing precisely when $w$ has 
a consecutive 321- or 312-pattern (respectively, 321- or 231-pattern).
\end{proof}

\begin{corollary}
\label{E-prop}
Suppose $\cE$ is the $\sim_A$-equivalence class of a word with distinct letters that 
is minimal under $<_A$ and has no consecutive 321-patterns.
If $X$ is the set of letters in this word then
 $\cE = \cA(z)^{-1}$ for some $z \in \I(S_X)$.
\end{corollary}

\begin{proof}
Suppose $w=w_1w_2\cdots w_n$ is a word with distinct letters that is minimal under $<_A$ and has no consecutive
321-patterns. If $X = \{w_1,w_2,\dots,w_n\}$ then
there is a unique involution   $z \in \I(S_X)$ with $w=0_A(z)$, and 
the $\sim_A$-equivalence class of $w$ is $\cA(z)^{-1}$ by Theorem~\ref{0-z-thm}.
\end{proof}

Suppose $X$ is a set with $[\pm n] = X \sqcup -X$.
The one-line representation of each  $\sigma \in S_X$
is  also the one-line representation of an element of $\W_n$.
Define $\XA(z)\subset \W_n$ for $z \in \I(S_X)$ as the 
 image of $\cA(z)^{-1}$ under this inclusion $S_X \hookrightarrow \W_n$.

\begin{corollary}\label{cV-cor}
Suppose $z \in \I(\W_n)$ and $\cE$ is an equivalence class in $\cA(z)^{-1}$ under $\sim_A$.
Then $\cE = \XA(\zeta)$
where $X$ is some subset with $[\pm n] = X \sqcup -X$ and $\zeta$ is some involution in $S_X$.
Consequently, $\cE$ has a unique minimal element and a unique maximal element under $<_A$.
\end{corollary}

\begin{proof}
Choose an element  $w\in\cE$ that is minimal under $<_A$.
By Proposition~\ref{321-prop}, $w$ has no consecutive 321-patterns, so the result follows by Corollary~\ref{E-prop}.
\end{proof}

\begin{example}\label{cV-ex}
The Hasse diagram of $(\cA(z)^{-1},<_A)$
for $z = \bar 1\hs\bar 2\hs\bar4\hs\bar3 \in \I(\W_4)$ is 
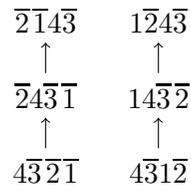
\begin{center}
\begin{tikzpicture}[scale=0.9]
\node (a) at (0,0) {$\bar2\hs\bar14\bar 3$};  \node (b) at (1.5,0) {$1\bar 2 4 \bar 3$};
\node (c) at (0,-1) {$\bar{2}4\bar{3}\hs\bar1$};  \node (d) at (1.5,-1) {$14\bar3\hs\bar2$};
\node (e) at (0,-2) {$4\bar3\hs\bar2\hs\bar1$};  \node (f) at (1.5,-2) {$4\bar{3}1\bar{2}$};
\draw  [->] (e) -- (c);
\draw  [->] (c) -- (a);
\draw  [->] (f) -- (d);
\draw  [->] (d) -- (b);
\end{tikzpicture}\end{center}
and  $\cA(z)^{-1}= \XA(\zeta_1)\sqcup \XA(\zeta_2)$ for $\zeta_1 = 4\bar213 = (\bar 3,4)(\bar2)(\bar1)$
and $\zeta_2 = 41\bar2\hs\bar3=(\bar 3,4)(\bar2,1)$.
\end{example}

We introduce the following terminology to associate a certain directed graph to any word.
Define the \emph{children} of a word $w$ to be the  subwords formed by removing a single descent.
Inductively define the \emph{descendants} of $w$ to consist of $w$  along with the descendants
of each of its children. Now construct the \emph{nested descent graph} of $w$ as the directed graph  
on the set of descendants of $w$ with a directed edge $u\to v$ whenever $v$ is a child of $u$.
Label each edge in this graph by the unique descent $(b,a)$ that is removed from the source to get the target.
As usual, we
adapt this definition to a permutation $w$ in $S_n$ or $\W_n$ by identifying $w$ with its one-line representation.

\begin{example}
The nested descent graph of $w=54321$ is shown below:
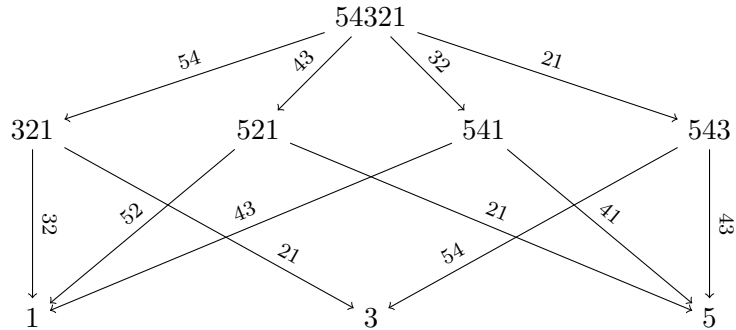
\begin{center}
\begin{tikzpicture}[xscale=0.8,yscale=0.8]
\node (00) at (4.5,0) {$54321$};
\node (10) at (0,-1.5) {$321$}; \node (11) at (3,-1.5) {$521$}; \node (12) at (6,-1.5) {$541$}; \node (13) at (9,-1.5) {$543$};
\node (20) at (0,-4) {$1$}; \node (21) at (4.5,-4) {$3$}; \node (22) at (9,-4) {$5$};
\draw  [->] (00) -- (10) node[midway,sloped,above] {{\scriptsize$54$}};
\draw  [->] (00) -- (11) node[midway,sloped,above] {{\scriptsize$43$}};
\draw  [->] (00) -- (12) node[midway,sloped,above] {{\scriptsize$32$}};
\draw  [->] (00) -- (13) node[midway,sloped,above] {{\scriptsize$21$}};
\draw  [->] (10) -- (20) node[midway,sloped,above] {{\scriptsize$32$}};
\draw  [->] (10) -- (21) node[near end,sloped,above] {{\scriptsize$21$}};
\draw  [->] (11) -- (20) node[midway,sloped,above] {{\scriptsize$52$}};
\draw  [->] (11) -- (22) node[midway,sloped,above] {{\scriptsize$21$}};
\draw  [->] (12) -- (20) node[midway,sloped,above] {{\scriptsize$43$}};
\draw  [->] (12) -- (22) node[midway,sloped,above] {{\scriptsize$41$}};
\draw  [->] (13) -- (22) node[midway,sloped,above] {{\scriptsize$43$}};
\draw  [->] (13) -- (21) node[near end,sloped,above] {{\scriptsize$54$}};
\end{tikzpicture}\end{center}
\end{example}

By construction, a word $w$ is the unique global source in its nested descent graph.

\begin{thmdef}\label{recdes-thmdef}
Suppose $z \in \I(\W_n)$ and $w\in  \cA(z)^{-1}$. 
 The nested descent graph of $w$ then has a unique global sink.
Choose a path from $w$ to the global sink and suppose  $(b_1,a_1),(b_2,a_2),\dots,(b_l,a_l)$ is the corresponding sequence of edge labels.
The set 
\be\label{ndes-eq}
\NDes(w) := \{  (b_1,a_1),(b_2,a_2),\dots,(b_l,a_l)\}
\ee
is then independent of the choice of path. Moreover, if $X = \{ w(1), w(2), \dots,w(n)\}$ then
it holds that
$w \in \XA(\zeta)$ for $\zeta = (a_1,b_1)(a_2,b_2)\cdots (a_l,b_l) \in \I(S_X)$.
\end{thmdef}

\begin{proof}
By Corollary~\ref{cV-cor}, there exists a set $X$ with $[\pm n ] = X\sqcup -X$ and an involution $\zeta \in \I(S_X)$ such that $w \in \XA(\zeta)$.
Let $\sigma \in \cA(\zeta)^{-1}$ be the preimage of $w$ under the map $\cA(\zeta)^{-1} \to \XA(\zeta)$.
To prove this result, it suffices to show that
(a) the nested descent graph of $\sigma$ has a unique global sink $\xi$, 
(b) the set of edge labels in the nested descent graph of $\sigma$ is the same for all paths from $\sigma$ to $\xi$, and
(c) the set $\NDes(\zeta)$ of edge labels described in (b) is precisely $\{ (b,a)  \in X \times X : a<b=\zeta(a)\}$.
These claims are a special case of \cite[Theorem 7.3]{M2}.
Alternatively, one can check (a), (b), and (c) directly for $\sigma = 0_A(\zeta)$, and then deduce by induction 
that the desired properties hold in general.
\end{proof}

Let $z \in \I(\W_n)$ and $w \in \cA(z)^{-1}$.
We call $\NDes(w)$ the set of \emph{nested descents} of $w$.
Write $\xi(w)$ for the unique global sink in the nested descent graph of $w \in \cA(z)^{-1}$.
Define $\NFix(w)$ as the set of letters in $\xi(w)$ that are positive, and
define $\NNeg(w)$ as the set of absolute values of the letters in $\xi(w)$ that are negative. 
We call elements of these sets \emph{nested negated points} and \emph{nested fixed points} of $w$.
If $w=w_1w_2\cdots w_n$
and $\NDes(w) = \{ (b_1,a_1),(b_2,a_2),\dots,(b_l,a_l)\}$,
then \[\{w_1,w_2,\dots,w_n\} = \{ a_1,a_2,\dots,a_l\} \sqcup \{b_1,b_2,\dots,b_l\}\sqcup \NFix(w) \sqcup -\NNeg(w).\]
\begin{example}\label{ndes-ex}
The nested descent graph of $w = \bar 1 67\bar2 348\bar 9 5 \in \W_9$ is shown below:
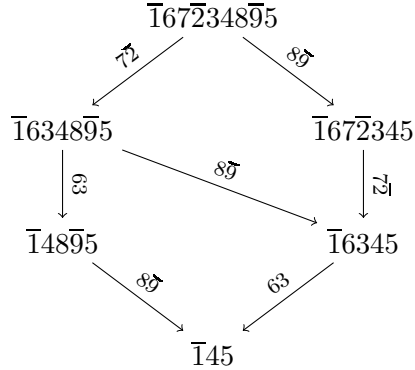
\begin{center}
\begin{tikzpicture}[xscale=0.8,yscale=0.8]
\node (00) at (0,0) {$\bar 1 67\bar 2 348\bar 9 5$};
\node (10) at (-2,-1.5) {$\bar 1 6 348\bar 9 5$};
\node (11) at (2,-1.5) {$\bar 1 67\bar2 345$};
\node (20) at (-2,-3) {$\bar 1 48\bar 9 5$};
\node (21) at (2,-3) {$\bar 1 6 345$};
\node (30) at (0,-4.5) {$\bar 1 45$};
\draw  [->] (00) -- (10) node[midway,sloped,above] {{\scriptsize$7\bar2$}};
\draw  [->] (00) -- (11) node[midway,sloped,above] {{\scriptsize$8\bar 9$}};
\draw  [->] (10) -- (20) node[midway,sloped,above] {{\scriptsize$6 3$}};
\draw  [->] (10) -- (21) node[midway,sloped,above] {{\scriptsize$8 \bar 9$}};
\draw  [->] (11) -- (21) node[midway,sloped,above] {{\scriptsize$7\bar2$}};
\draw  [->] (20) -- (30) node[midway,sloped,above] {{\scriptsize$8 \bar 9$}};
\draw  [->] (21) -- (30) node[midway,sloped,above] {{\scriptsize$63$}};
\end{tikzpicture}\end{center}
We have $w \in \cA(z)^{-1}$ for $z = (1,\bar 1)(2,\bar7)(\bar 2,  7)(3, 6)(\bar 3, \bar6)(8,\bar 8)(9,\bar 9) \in \I(\W_9)$.
As predicted by Theorem-Definition~\ref{recdes-thmdef}, the nested descent graph of $w$ has a unique global sink $\xi(w)=\bar 1 45$,
and all paths from the source to the sink 
have edge labels $(8,-9), (7,-2),(6,3)$ in some order. 
We therefore have $\NDes(w) = \{ (8,-9), (7,-2),(6,3)\}$, $\NNeg(w) = \{ 1\}$, and $\NFix(w)=\{4,5\}$.
\end{example}

\begin{corollary}\label{ndes-cor2} 
Let $z \in \I(\W_n)$ and suppose $v, w \in \cA(z)^{-1}$.
\begin{enumerate}
\item[(a)] No word in the nested descent graph of $w$ has a consecutive 321-pattern.
\item[(b)] If  $v\sim_A w$ then $\NDes(v) = \NDes(w)$,
 $\NFix(v) = \NFix(w)$, and $\NNeg(v) = \NNeg(w)$.
\end{enumerate}

\end{corollary}

\begin{proof}
Part (a) is necessary for $\NDes(w)$ to be well-defined.
Part (b) is an immediate consequence of Corollary~\ref{cV-cor} and  Theorem-Definition~\ref{recdes-thmdef}.
\end{proof}

Can, Joyce, and Wyser's results in type A \cite{CJW} have this consequence for signed involutions.

\begin{lemma}\label{cjw-lem}
Let $z\in \I(\W_n)$ and $w \in \cA(z)^{-1}$. 
Suppose $e,e' \in \NFix(w) \sqcup -\NNeg(w)$ and $(a,b),(a',b') \in \NDes(w)$.
The following properties then hold: (1) If $e<e'$ then $ee'$ is a subword of $w=w_1w_2\cdots w_n$. (2) If $e<a<b$ (respectively, $a<b<e$) then $eba$ (respectively, $bae$) is a subword of $w$.
(3) Finally, if $a<a'$ and $b<b'$ then $bab'a'$ is a subword of $w$.
\end{lemma}

\begin{proof}
Because $w \in \XA(\zeta)$ for the involution $\zeta \in \I(S_{\{w_1,w_2,\dots,w_n\}})$ 
whose nontrivial cycles are the pairs in $\NDes(w)$,
these properties are equivalent to \cite[Theorem 2.5]{CJW}.
It is also an instructive exercise to derive the lemma by considering the nested descent graph of $w$.
For each hypothesis, one can check that the desired conclusion fails only if an extraneous descent appears in $\NDes(w)$
or if we can relate $w$ via $\sim_A$ to an element $v \in \W_n$ with a consecutive 321-pattern.
\end{proof}
 
 \begin{lemma}\label{tech-lem1}
Let $z \in \I(\W_n)$, $w \in \cA(z)$, and $i \in [n]$. 
Suppose $w(1) =i$ and $w(2)=-(i+1)$. 
Then $z(1) = -1$ and $z(2)=-2$.
\end{lemma}

\begin{proof}
Let $u =w t_1t_0t_1$
and $y = u^{-1}\circ u \in \I(\W_n)$. Then $\ell(w) = \ell(u) + 3$, so $u \in \cA(y)$ and $t_1\notin\DesR(y)$.
Since  $yt_1 =  u^{-1}\circ t_i \circ u  = t_1 y$,
we have either
$y(1) =1$ and $y(2)=2$, or $y(1) = -2$ and $y(2) = -1$. The second case is impossible since 
$y < t_1 \circ y \circ t_1  = yt_1 < t_0\circ (yt_1) \circ t_0$,
so it follows
that 
$z = t_1 \circ t_0 \circ t_1 \circ y \circ t_1 \circ t_0 \circ t_1=
 yt_0t_1t_0t_1$
and consequently that $z(1) = -1$ and $z(2)=-2$.
\end{proof}

A word $w_1w_2\cdots w_n$ has a \emph{consecutive $\bar1\hs\bar2$-pattern} if for some $i \in [n-1]$
 it holds that $0>w_i>w_{i+1}$.
 A  permutation in $S_X$ or $ \W_n$ has a consecutive $\bar{1}\hs\bar{2}$-pattern if its one-line representation does.
 
 \begin{proposition}\label{12bar-prop}
 If $w \in \cA(z)^{-1}$ for some $z \in \I(\W_n)$,
then neither $w$ nor any other word in $w$'s nested descent graph has a consecutive $\bar1\hs\bar2$-pattern.
\end{proposition}


\begin{proof}
Suppose that $w \in \W_n$ has
$w^{-1}(1) =-i$ and $w^{-1}(2)=-(i+1)$ where $i \in [n-1]$,
and assume $w \in \cA(z)^{-1}$ for some $z \in \I(\W_n)$. We produce a contradiction.
Since $t_0w<_Lw$, we have $t_0w \in \cA(y)^{-1}$ for some $y \in \I(\W_n)$ by Proposition~\ref{weak-prop}.
As $(t_0w)^{-1}(1) = i$ and $(t_0w)^{-1}(2) = -(i+1)$,  it follows
from Lemma~\ref{tech-lem1} that $y(1) = -1$.
But this means that $y=(t_0w) \circ (t_0w)^{-1} = w \circ w^{-1} = z$,
which is impossible since $\ell(t_0w) < \ell(w)$ and $w^{-1} \in \cA(z)$.

Next suppose $w \in \W_n$ is an arbitrary permutation with a consecutive $\bar{1}\hs\bar{2}$-pattern,
so that for some $a<b$ in $[n]$ it holds that $w^{-1}(a) = -i$ and $w^{-1}(b) = -(i+1)$.
It is an exercise to construct an element $v \in \W_n$ with $v^{-1}<_R w^{-1}$ and $v^{-1}(1) = -i$ and $v^{-1}(2) =-(i+1)$.
By the previous paragraph,  $v^{-1}$ is not an atom for any involution,
so by Proposition~\ref{weak-prop} neither is $w^{-1}$.
This shows that no inverse atom of a signed involution has a  consecutive $\bar{1}\hs\bar{2}$-pattern.
By
 Corollary~\ref{ndes-cor2}(a), the same holds for all words in the nested descent graph of an inverse atom.
\end{proof}

\section{Partial orders}\label{order-sect}

Again
define $<_{B}$ to be the transitive closure of $\vartriangleleft_A$ and $\vartriangleleft_B$ from \eqref{<B-eq}.
Let $\sim_B$ be the symmetric closure of $<_B$.
We apply $<_B$,  $\vartriangleleft_B$ and  $\sim_B$ to elements of $\W_n$ via their one-line representations.
%

\begin{lemma}\label{<B-lem}
If $w \in \W_n$, $z \in \I(\W_n)$, $v \in \cA(z)^{-1}$, and $v \sim_B w$, then $w \in \cA(z)^{-1}$.
\end{lemma}

\begin{proof}
Suppose $v,w \in \W_n$ are such that $v\vartriangleleft_B w$.
If $v_1v_2=\bar b  \bar a$ and $w_1w_2 = a \bar b$ where $0<a<b$,
then there exists $\sigma \in \W_n$
such that 
$v =  \sigma\circ \bar{2}\hs\bar{1}$, $w=\sigma\circ 1 \bar 2$,
and $\ell(v) =\ell(w) = \ell(\sigma) + 3$.
If $v_1v_2v_3 = \bar c \bar b  \bar a$ and $ w_1w_2w_3 = \bar c a \bar b$ where $0<a<b<c$,
then there exists $\sigma \in \W_n$
such that $
v =  \sigma\circ \bar{3}\hs\bar 2\hs\bar 1$, $w =\sigma\circ \bar3 1\bar 2$,
and $\ell(v)=\ell(w) =\ell(\sigma) + 6$.
The lemma follows by checking 
that $\bar{2}\hs\bar{1}=t_0t_1t_0$ and $1\bar 2=t_1t_0t_1$  are inverse atoms of $\bar{1}\hs\bar{2}$,
while $\bar{3}\hs\bar{2}\hs\bar{1}=t_0t_1t_0t_2t_1t_0$ and $\bar 3 1\bar 2 = t_1t_0t_1t_2t_1t_0$ 
 are inverse atoms of $\bar{1}\hs\bar{2}\hs\bar{3}$. 
\end{proof}

Define $\vartriangleleft^+_B$ as the ``extended'' relation on $n$-letter words with $v\vartriangleleft^+_B w$
if for some $i\in[n-1]$ 
\be\label{black}
 v_1<v_2<\dots<v_i = w_{i+1}<v_{i+1} = -w_i < 0
\qquand
v_j=w_j \text{ if }j \notin \{i,i+1\}
.
\ee
Thus
$ \bar{z}\cdots \bar{c}\bar{b}\bar{a}\cdots  \vartriangleleft^+_B \bar{z}\cdots \bar{c}a\bar{b}\cdots $
if $0<a<b<c<\dots<z$. We have ${\vartriangleleft_B} \Rightarrow {\vartriangleleft^+_B}$. Conversely:

\begin{lemma}\label{jump-lem}
If $v$ and $w$ are words with $n$ letters
and $v\vartriangleleft^+_B w$, then $v\sim_B w$.
\end{lemma}

\begin{proof}
Assume $v$, $w$, and $i$ are as in \eqref{black}.
Form $v'$  from $v$
by replacing  $v_{j-1}v_j$ by $\bar{v_j}v_{j-1}$ for each even index $j<i$.
If $i$ is odd (respectively, even), then define $v''$ by removing the subword $v_iv_{i+1}$ 
(respectively, $v_{i-1}v_iv_{i+1}$)
from $v'$ and placing it at the start of the word.
Define $w'$ from $w$ and $w''$ from $w'$ analogously. 
By induction on $i$, we can assume that $v\sim_B v'$ and $w'\sim_B w$.
It is an exercise to check that $v' < _A v''$ and $w' <_A w''$,
and it holds by definition that $v'' \vartriangleleft_B w''$,
so $v\sim_B w$.
\end{proof}

Next define $\scov$ as the relation on $n$-letter words with $v \scov w$
if for some $i \in [n-1]$ and some positive numbers $a,b$ it holds that $v_j =w_j$ for $j \notin \{i,i+1\}$ while
\be\label{black2}v_iv_{i+1} = \bar{b}\bar{a},\qquad  w_iw_{i+1}=a\bar b, \qquand 0<a<b=\min\{ |v_1|,|v_2|,\dots,|v_i|\}.\ee
When $v_1<v_2<\dots<v_i$ these conditions are equivalent to \eqref{black}, 
so  
$v \vartriangleleft^+_B w\ \Rightarrow\ v \scov  w$.


\begin{lemma}\label{jump2-lem}
Let $z \in \I(\W_n)$ $v \in \W_n$, and $w \in \cA(z)^{-1}$. If $v\scov w$ then $w \sim_B v \in \cA(z)^{-1}$.
\end{lemma}

The converse does not hold: when $v\scov w $ and $v\in \cA(z)^{-1}$ it may occur that $w\notin \cA(z)^{-1}$.

\begin{proof}
Assume $v$, $ w$, and $i$ are as in \eqref{black2} and $w \in \cA(z)^{-1}$.
 Let $j$ be the maximal index in $\{1,2,\dots,i-1\}$ such that $w_j>0$.
If no such index $j$ exists then
the numbers $w_1,w_2,\dots,w_{i-1}$ are all negative, so it follows from Proposition~\ref{12bar-prop}
 that $v\vartriangleleft^+_B w$ whence $v \sim_B w$ by Lemma~\ref{jump-lem}.
 
 Suppose $j<i$.
 Proposition~\ref{321-prop} then implies that $j<i-1$.
 First consider the case when $j<i-2$, and let 
$e=v_j=w_j$, $c=v_{j+1}=w_{j+1}$, and $d=v_{j+2}=w_{j+2}$.
Since $c$ and $d$ are both negative,
 it follows by Proposition~\ref{12bar-prop} that $c<d<0<e$ so we have $ecd \vartriangleleft_A dec$.
 Form $v'$ and $w'$  from the one-line representations of $v$ and $w$
 by replacing the subword $ecd = v_jv_{j+1}v_{j+2}=w_jw_{j+1}w_{j+2}$ by $dec$.
 Then $v \vartriangleleft_A v' \scov w'$ and $w\vartriangleleft_A w' \in \cA(z)^{-1}$,
 so by induction 
 $v\sim_B v'\sim_B w'\sim_B w$.

Now consider the case when $j=i-2$, and let $d=v_{i-2}=w_{i-2}$ and $c=-v_{i-1}=-w_{i-1}$.
By construction both $c$ and $d$ are positive and greater than $b>a$, so we have $d\bar c a\bar b \vartriangleleft_A a d \bar c \bar b$.
It follows by Proposition~\ref{12bar-prop} that $b<c$, so we also have $ad\bar c \bar b \vartriangleleft_A a \bar b d \bar c$
and similarly $d \bar c \bar b \bar a \vartriangleleft_A \bar b d \bar c \hs\bar a \vartriangleleft_A \bar b \bar a d \bar c$. 
Let $v''$ and $w''$ be the signed permutations formed from the one-line representations of $v$ and $w$
by replacing the subword $d \bar c \bar b \bar a$ by $\bar b \bar a d \bar c$
and the subword $ d\bar c a\bar b$ by $a \bar b d \bar c$.
Then $v <_A v'' \scov w''$ and $w<_A w'' \in \cA(z)^{-1}$,
so again by induction $v \sim_B v '' \sim_B w'' \sim_B w$. 
\end{proof}


\begin{lemma}\label{ht-lem}
Suppose $v,w\in \W_n$
and $i \in [n-1]$ are as in \eqref{black2} so that $v\scov w$.
Assume $v,w \in \cA(z)^{-1}$
for some $z \in \I(\W_n)$,
and set $a = -v_{i+1}= w_i < b=-v_i = -w_{i+1}$.
Then
$\NDes(w) = \NDes(v) \sqcup \{(a,-b)\}$, $\NNeg(v) = \NNeg(w) \sqcup \{ a, b\}$, and $\NFix(v) = \NFix(w)$.
\end{lemma}

\begin{proof}
By Corollary~\ref{ndes-cor2}(a) and Proposition~\ref{12bar-prop}, we know that 
(1) none of the vertices in the nested descent graphs of $v$ or $w$
have 321- or $\bar1\hs\bar2$-patterns,
and we assume by hypothesis that (2) $0<a<b<\min\{ |v_1|,|v_2|,\dots,|v_{i-1}|\}=\min\{ |w_1|,|w_2|,\dots,|w_{i-1}|\}$.
It follows that we may choose a path from $w$ to the global sink in its nested descent graph
whose last edge is the unique one labeled by the descent $a\bar b$.
Replacing the subword $a\bar b$ by $\bar b \bar a$ in all but the last vertex in this path 
produces a path from the source to some vertex in the nested descent graph of $v$.
In view of (1) and (2), this vertex must be the global sink.
The lemma now follows from
Theorem-Definition~\ref{recdes-thmdef}.
\end{proof}

The \emph{(weak) atomic order} of type B is the transitive closure $<_{B}$ of $\vartriangleleft_A$ and $\vartriangleleft_B$.
Define the \emph{strong atomic order} $\ll_B$ of type B 
to be the transitive closure of the relations $\vartriangleleft_A$ and $\scov$.
 
\begin{corollary}\label{isorder-cor}
If $z \in \I(\W_n)$
then $<_B$  and $\ll_B$ restrict to partial orders on $\cA(z)^{-1}$.
\end{corollary}

\begin{proof}
Define $h(u) =  | \{ (a,-b) \in \NDes(u) : 0 < a < b \}|$ for $u \in \cA(z)^{-1}$.
It suffices to show that $\ll_B$ is antisymmetric.
This follows since 
if $v,w \in \cA(z)^{-1}$ have $v\vartriangleleft_A w$ or $v\vartriangleleft_B w$,
then either $w$ exceeds $v$ in reverse lexicographic order
while $h(w) = h(v)$,
or $h(w) = h(v)+1$ by Lemma~\ref{ht-lem}.
\end{proof}

\begin{figure}[h]
\begin{center}
\begin{tikzpicture}[scale=0.9]
\node (0) at (4,7.2) {$\bar 1 2 \bar 3 4 \bar 5$};  
\node (1a) at (2,6) {$\bar 1 2 4 \bar 5\hs \bar 3$};  
\node (1b) at (4,6) {$2 \bar 3\hs \bar 1 4 \bar 5$};  
\node (1c) at (8,6) {$\bar 3 1 \bar 2 4 \bar 5$};  
\node (1d) at (14,6) {$\bar 1 2 \bar 5 3 \bar 4$};  
\node (2a) at (2,4.8) {$\bar 1 4 \bar 5 2 \bar 3$};  
\node (2b) at (4,4.8) {$2 \bar 3 4 \bar5\hs \bar 1$};  
\node (2c) at (6,4.8) {$\bar 3 \hs\bar 2\hs \bar 1 4 \bar 5$};  
\node (2d) at (8,4.8) {$\bar 3 1 4 \bar 5 \hs\bar 2$};  
\node (2e) at (10,4.8) {$\bar 5 1 \bar 2 3 \bar 4$};  
\node (2f) at (14,4.8) {$2\bar 5\hs \bar 1 3 \bar 4$};  
\node (3a) at (0,3.6) {$\bar 5 1 \bar 4 2 \bar 3$};  
\node (3b) at (2,3.6) {$4 \bar 5\hs \bar 1 2 \bar 3$};  
\node (3c) at (4,3.6) {$24\bar 5 \hs\bar 3\hs \bar 1$};  
\node (3d) at (6,3.6) {$\bar 3\hs \bar 2 4 \bar 5\hs \bar 1$};  
\node (3e) at (8,3.6) {$\bar 3 4\bar 51\bar 2$};  
\node (3f) at (10,3.6) {$\bar 5 1 3\bar 4\hs \bar 2$};  
\node (3g) at (12,3.6) {$\bar 5\hs \bar 2\hs \bar 1 3 \bar 4$};  
\node (3h) at (14,3.6) {$2\bar 5 3\bar 4\hs \bar 1$};  
\node (4a) at (2,2.4) {$\bar 5\hs \bar 4\hs \bar 1 2 \bar 3$};  
\node (4b) at (4,2.4) {$4\bar 5 2 \bar 3\hs \bar 1$};  
\node (4c) at (6,2.4) {$\bar 3 4\bar 5\hs\bar 2\hs\bar 1$};  
\node (4d) at (8,2.4) {$4\bar 5\hs\bar 31\bar 2$};  
\node (4e) at (10,2.4) {$\bar 5 3\bar 41\bar 2$};  
\node (4f) at (12,2.4) {$\bar 5\hs \bar 2 3\bar 4\hs \bar 1$};  
\node (5a) at (4,1.2) {$\bar 5\hs\bar 4 2 \bar 3\hs \bar 1$};  
\node (5b) at (6,1.2) {$4\bar 5\hs \bar 3\hs \bar 2\hs \bar 1$};  
\node (5c) at (8,1.2) {$\bar 5\hs \bar 4\hs \bar 31\bar 2$};  
\node (5d) at (10,1.2) {$\bar 5 3\bar 4\hs \bar 2\hs \bar 1$};
\node (6) at (7,0) {$\bar 5\hs \bar 4\hs \bar 3\hs \bar 2\hs \bar 1$};  
\draw  [->,dotted] (6) -- (5a); 
\draw  [->,dashed] (6) -- (5b);
\draw  [->,dotted] (6) -- (5c); 
\draw  [->,dashed] (6) -- (5d); 
\draw  [->] (5a) -- (4a);
\draw  [->,dashed] (5a) -- (4b);  
\draw  [->] (5b) -- (4c);
\draw  [->,dotted] (5b) -- (4b); 
\draw  [->,dotted] (5b) -- (4d); 
\draw  [->,dashed] (5c) -- (4d);
\draw  [->,dashed] (5c) -- (4e);
\draw  [->,dotted] (5d) -- (4e); 
\draw  [->] (5d) -- (4f);
\draw  [->,dashed] (4a) -- (3a);
\draw  [->,dashed] (4a) -- (3b);
\draw  [->] (4b) -- (3b);
\draw  [->] (4b) -- (3c);
\draw  [->] (4c) -- (3d);
\draw  [->,dotted] (4c) -- (3e); 
\draw  [->] (4d) -- (3e);
\draw  [->] (4e) -- (3f);
\draw  [->] (4f) -- (3g);
\draw  [->,dashed] (4f) -- (3h);
\draw  [->] (3b) -- (2a);
\draw  [->] (3c) -- (2b);
\draw  [->,dashed] (3d) -- (2b);
\draw  [->] (3d) -- (2c);
\draw  [->] (3e) -- (2d);
\draw  [->] (3f) -- (2e);
\draw  [->,dashed] (3g) -- (2e);
\draw  [->,dashed] (3g) -- (2f);
\draw  [->] (3h) -- (2f);
\draw  [->] (2a) -- (1a);
\draw  [->] (2b) -- (1b);
\draw  [->,dashed] (2c) -- (1b);
\draw  [->,dashed] (2c) -- (1c);
\draw  [->] (2d) -- (1c);
\draw  [->] (2f) -- (1d);
\draw  [->] (1a) -- (0); 
\draw  [->] (1b) -- (0); 
\end{tikzpicture}\end{center}
\caption{The Hasse diagram of the poset $(\cA(z)^{-1},\ll_B)$ for $z = \bar{1}\hs\bar2\hs\bar3\hs\bar4\hs\bar{5} \in \I(\W_5)$.
The solid arrows correspond to the covering relations $\vartriangleleft_A$,
the dashed arrows correspond to $\vartriangleleft_B$,
and the dotted arrows correspond to the remaining relations $\scov$.}
\label{fig2}
\end{figure}

In the example shown in 
Figure~\ref{fig2}, the  orders $<_B$ and $\ll_B$  restricted to $\cA(z)^{-1}$ are  graded and connected,
and $(\cA(z)^{-1},\ll_B)$ has a unique minimal element.
We will show that these properties are general phenomena.
Suppose $z\in \I(\W_n)$. Let
$\Fix(z) = \{ i \in [n] : z(i) = i\}$ and recall that $\Pair(z) = \{ (a,b) \in [\pm n]\times [n] : |a| < z(a) = b\}$ and  $\Neg(z) = \{i \in [n] : z(i) = -i\}$.
Next define \be\label{cycb-eq}
\Cyc_B(z) := 
  \Pair(z)
  \sqcup
  \{ (-a,-a) : a \in \Neg(z) \} 
  \sqcup
 \{ (a,a) : a \in \Fix(z) \}.
 \ee
If $\Cyc_B(z) = \{(a_1,b_1),(a_2,b_2),\dots,(a_l,b_l)\}$ where $a_1<a_2<\dots<a_l$,
then we let 
\be\label{0B-eq}
0_B(z) := [[ b_1a_1b_2a_2\cdots b_la_l]].
\ee
Recall that $[[w]]$ denotes the subword of $w=w_1w_2\cdots w_n$ formed by omitting all repeated letters after their first appearance.
Thus $0_B(z)$ is a word with distinct letters by construction,
and it is straightforward to check that $0_B(z)$ is in fact the one-line representation of an element of $ \W_n$.
If $z = (1,\bar 1)(2,\bar7)(\bar 2,  7)(3, 6)(\bar 3, \bar6)(8,\bar 8)(9,\bar 9) \in \I(\W_9)$
as in Example~\ref{ndes-ex}, then
$ \Neg(z) =\{1,8,9\},
$
$
\Fix(z) =\{4,5\},
$
and
$ 
\Pair(z) =\{ (- 2,7), (3,6) \},$
so $ 0_B(z) = \bar 9\hs \bar 8 7\bar 2\hs \bar 1 63 4 5$.


\begin{lemma}\label{0isatom-lem}
If $z \in \I(\W_n)$ then
 $0_B(z) \in \cA(z)^{-1}$.
 \end{lemma}

\begin{proof}
Fix $z \in \I(\W_n)$ and recall the definition of   $\Psi_n : \W_n \to S_{2n}$ from Section~\ref{embed-subsect}.
Since $\frac{1}{2}(\ell_0(z) + \neg(z))  = \ell_0(0_B(z)^{-1})$ by definition,
 it suffices by Lemma~\ref{equiv-lem} to check that $\Psi_n(0_B(z))^{-1} \in \cA(\Psi_n(z))$.
This follows by applying \cite[Theorem 2.5]{CJW}, which is just Lemma~\ref{cjw-lem} restricted to $S_n\hookrightarrow \W_n$.
In detail, if $\upsilon \in S_n$,  $\zeta \in \I(S_n)$,
and $\Cyc_A(\zeta) = \{ (a,b) \in[n]\times [n] : a\leq b = \zeta(a)\}$,
then we have $\upsilon \in \cA(\zeta)^{-1}$ precisely when
(1) if $(a,b) \in \Cyc_A(\zeta)$ then $b$ is weakly left of $a$ in the one-line representation of $\upsilon$, and no number $e \in [n]$ with $a<e<b$ appears between $a$ and $b$,
and
(2) if $(a,b),(a',b') \in \Cyc_A(\zeta)$ are such that $a<a'$ and $b<b'$ then $ba'$ is a subword of $\upsilon$.
It is straightforward to check that (1) and (2) hold for $\zeta = \Psi_n(z)$ and $\upsilon = \Psi_n(0_B(z))$.
\end{proof}

Putting things together leads to a short proof of the following theorem.

\begin{theorem}
\label{hz-thm}
If $z \in \I(\W_n)$ then $0_B(z)$ is the unique minimum in $(\cA(z)^{-1},\ll_B)$.
\end{theorem}

The element $0_B(z)$ may fail to be the unique minimum in $\cA(z)^{-1}$ under $<_B$ 
as well as under the transitive closure of  $\vartriangleleft_A$ and $\vartriangleleft_B^+$
(which is between $<_B$ and $\ll_B$ in strength).

\begin{proof}
Fix $z \in \I(\W_n)$ and $w \in \cA(z)^{-1}$.
 Corollary~\ref{cV-cor} and Proposition~\ref{12bar-prop} imply that $v\leq _A w$ 
for a unique element of the form $v =
[[b_1a_1b_2a_2\cdots b_la_l]] \in \W_n$ where  $a_i=b_i$ or $a_i<b_i > 0$ for each $i \in [l]$ and  $a_1<a_2<\dots<a_l$.
We have $\NDes(w) = \NDes(v) = \Des(v) = \{ (b_i,a_i) : i\in[l],\ a_i < b_i\}$ by Corollary~\ref{ndes-cor2}(b).
As in the proof of Corollary~\ref{isorder-cor},
let $h(w) =  | \{ (a,-b) \in \NDes(w) : 0 < a < b \}|$.
If $h(w) = 0$,
then $v = 0_B(y)$ for some $y \in \I(\W_n)$, in which case we must have $y=z$ by Lemma~\ref{0isatom-lem}, so
$0_B(z) =v\leq_A w$.
Assume $h(w) > 0$ and let 
$j \in [l]$ be the smallest index such that $a_j < 0 < b_j < -a_j$. 
Then whenever $i<j$ and $a_i < b_i$, it must hold that $a_i < a_j < 0$ and $-a_j < b_j$, so we have $a_i < a_j < b_j < b_i$.
Therefore  $u  \scov  v$ for the signed permutation
$ u = [[b_1a_1\cdots  a_j \bar{b_j} \cdots b_{l}a_{l}]] \in \W_n.$
Since $h(u) +1=h(v)=h(w)$ by Lemma~\ref{ht-lem}, we may assume by induction that $0_B(z)=u$ or $0_B(z) \ll_B u$,
so $0_B(z) \ll_B w $ since $ u \scov v \leq_A w$.
\end{proof}

\begin{corollary}\label{hz-cor1}
If $z \in \I(\W_n)$ then $\cA(z)^{-1}$ is a single equivalence class under $\sim_B$.
 \end{corollary}
 
  \begin{proof}
 This is immediate from the preceding theorem and Lemma~\ref{jump2-lem}.
 \end{proof}
 
Let $\cR(w)$ denote the set of reduced words for $w \in \W_n$ and
define $\iR(z) = \bigsqcup_{w \in \cA(z)} \cR(w)$ when $z \in \I(\W_n)$.
It is well-known that $\cR(w)$ is spanned and preserved by the \emph{braid relations} 
$\cdots t_0t_1t_0t_1\cdots \sim \cdots t_1t_0t_1t_0 \cdots $ and 
$\cdots t_it_{i+1}t_i \cdots \sim \cdots t_{i+1}t_it_{i+1}\cdots$ for $i \in [n-1]$.
The previous corollary is equivalent to the following result of Hu and Zhang.

\begin{corollary}[{Hu and Zhang \cite[Theorem 4.8]{HuZhang2}}]\label{hz-cor2}
If $z \in \I(\W_n)$ then $\iR(z)$ is spanned and preserved by the usual set of braid relations for $\W_n$ plus
the initial relations $t_it_{i+1} \cdots \sim t_{i+1}t_i \cdots$ for $i \in [n-1]$ and $t_0t_1t_0 \cdots \sim t_1t_0t_1\cdots$  
and $t_0t_1t_2t_0t_1t_0 \cdots \sim t_0t_1t_2 t_1t_0 t_1 \cdots$.
\end{corollary}

\begin{proof}
To deduce this from Corollary~\ref{hz-cor1} or vice versa, it suffices
to check that $z\in \I(\W_n)$ has atoms $v,w \in \cA(z)$
with $v^{-1} \vartriangleleft_A w^{-1}$ or $v^{-1} \vartriangleleft_B w^{-1}$
if and only if $v$ and $w$ have reduced words connected by the given relations.
This holds by a simple calculation using Proposition~\ref{12bar-prop}.
\end{proof}

Hansson and Hultman, extending this result, have found a general description of the relations needed to span the sets $\iR(z)$
for (twisted) involutions $z$ in any Coxeter group~\cite{HH}.

\section{Noncrossing shapes}\label{shapes-sect}

The connected posets $(\cA(z)^{-1},<_B)$ and $(\cA(z)^{-1},\ll_B)$ for $z \in \I(\W_n)$
are no longer intervals, though examples suggest they may be meet semilattices.
In this section, we characterize these posets' extremal elements. 

Consider a subset $X\subset [\pm n]$ with $X=-X$.
A \emph{matching} in (the complete graph on) $X$ is a set $M$ of pairwise disjoint 2-element subsets of $X$.
A matching $M$ is \emph{symmetric} if $\{-i,-j\} \in M$ whenever $\{i,j\} \in M$;
\emph{perfect} if for each $i \in X$ there exists a unique $j \in X$ with $\{i,j\} \in M$;
and \emph{noncrossing} if no two subsets $\{i,k\},\{j,l\} \in M$ have $i<j<k<l$.
The 3 perfect noncrossing symmetric matchings in $[\pm 3]$ are
$ \{ \{1,\bar 1\}, \{2,\bar 2\}, \{3,\bar3\}\}$, $ \{\{1,2\},\{\bar1,\bar2\},\{3,\bar3\}\}$, and $\{\{1,\bar1\},\{2,3\},\{\bar2,\bar3\}\}$.
In general, there are $\binom{n}{\lfloor n/2\rfloor}$ symmetric noncrossing perfect matchings in $[\pm n]$,
and such matchings are in bijection
with many other combinatorially defined objects (see \cite[A001405]{OEIS}).

Let $z\in \I(\W_n)$ and recall the sets $\Neg(z)$, $\Fix(z)$, and $\Pair(z)$ introduced before Lemma~\ref{0isatom-lem}.
Define $\NCSM(z)$ as the set of noncrossing, symmetric, perfect matchings in $\Neg(z) \sqcup -\Neg(z)$.
For each matching $M \in \NCSM(z)$,
we define three related sets:
\be\label{neg-pair-cyc-b-eq}
\ba
\Neg(z,M) &:= \{ i \in \Neg(z) : \{i,-i\} \in M\},
\\
\Pair(z,M) &:= \Pair(z) \sqcup \{ (-b,a) : \{a,b\} \in M\text{ and }0<a<b \},
\\
 \Cyc_B(z,M) &:= 
  \Pair(z,M)
  \sqcup
  \{ (-a,-a) : a \in \Neg(z,M) \} 
  \sqcup
 \{ (a,a) : a \in \Fix(z) \}
.
\ea\ee
Suppose we have 
$\Cyc_B(z,M)= \{ (a_1,b_1),(a_2,b_2),\dots,(a_l,b_l)\} =  \{ (c_1,d_1), (c_2,d_2),\dots,(c_l,d_l)\}$
where $a_1 < a_2 < \dots <a_l$ and $d_1<d_2<\dots<d_l$.
We define $0_B(z,M)$ and $1_B(z,M)$ to be the words 
\be\label{01B-eq} 0_B(z,M) := [[b_1a_1b_2a_2\cdots b_la_l]]
\qquand
1_B(z,M) := [[d_1c_1d_2c_2\cdots d_lc_l]].
\ee
 We have $\Neg(z)  = \Neg(z,M_{\min})$, $\Pair(z) = \Pair(z,M_{\min})$, and $0_B(z) = 0_B(z,M_{\min})$ 
 for the matching $M_{\min} := \{ \{i,-i\} : i \in \Neg(z)\} $.
 
 \begin{example}\label{m-ex}
Let $z = (1,\bar 1)(2,\bar7)(\bar 2,  7)(3, 6)(\bar 3, \bar6)(8,\bar 8)(9,\bar 9) \in \I(\W_9)$
 as in Example~\ref{ndes-ex}, so that 
$\Neg(z) =\{1,8,9\}$.
The three elements of $\NCSM(z)$ are
\[
M^1 = \{  \{\bar 9, 9\}, \{\bar8,8\},\{\bar1,1\}\},
\quad
M^2 = \{ \{\bar 9, 9\}, \{ \bar 1, \bar 8\}, \{ 1, 8\}\},
\quad
M^3= \{ \{\bar 8,\bar 9\}, \{\bar 1, 1\}, \{ 8,9\}\}.
\]
We have $\Neg(z,M^1) = \Neg(z) = \{1,8,9\}$ and $\Pair(z,M^1) = \Pair(z) = \{ (\bar 2,7),(3,6)\}$, so 
\[ 0_B(z,M^1) = \bar 9\hs \bar 8 7\bar 2 \hs\bar 1 63 4 5
\qquand
1_B(z,M^1) = \bar 9 \hs\bar 8\hs \bar 1 45 63 7\bar 2.\]
Similarly $\Neg(z,M^2) = \{9\}$ and $\Pair(z,M^2) = \{(\bar 2, 7),(3,6),(1,\bar 8)\}$, so
\[
0_B(z,M^2) = \bar 9 1\bar 8 7\bar 2 6345
\qquand
1_B(z,M^2)= \bar 9 1\bar 8 4563 7\bar 2.
\]
Finally $\Neg(z,M^3) = \{1\}$ and $\Pair(z,M^3) = \{ (\bar 2,7),(3,6),(8,\bar 9)\}$, so
\[ 
0_B(z,M^3) = 8\bar 9 7 \bar 2\hs \bar 1 63 45
\qquand
1_B(z,M^3) = \bar 1 4 5 63 7 \bar 2 8 \bar 9.\]
 \end{example}

\begin{proposition}\label{01-prop}
Let $z \in \I(\W_n)$ and $M \in \NCSM(z)$. 
The words $0_B(z,M)$ and $1_B(z,M)$ may be interpreted as elements of $\W_n$
written  in one-line notation. Under $<_A$, the permutation $0_B(z,M)$ is minimal while $1_B(z,M)$ 
 is maximal, and it holds that $0_B(z,M) \leq_A 1_B(z,M)$.
 \end{proposition}
 
 \begin{proof}
 Let $X$ be the set of numbers $a$ and $b$
occurring in pairs $(a,b) \in \Cyc_B(z,M)$.
Check that $[\pm n] = X \sqcup -X$,  and conclude that
$0_B(z,M)$ and $1_B(z,M)$ belong to $\W_n$.
Define $\zeta \in \I(S_X)$ to be the involution with $a<b=z(a)$ for $a,b \in X$ 
if and only if $(a,b) \in \Pair(z,M)$.
Then we have $0_B(z,M)= 0_A(\zeta)$ and $1_B(z,M)=1_A(\zeta)$ as words, so the result follows
from Theorem~\ref{0-z-thm}.
 \end{proof}

The following corollary refers to the map $\Psi_n : \W_n \to S_{2n}$ from Section~\ref{embed-subsect}.

 \begin{corollary}\label{isatom-lem}
Let $z \in \I(\W_n)$. If $0_B(z) \leq_A w \in \W_n$ then
$\Psi_n( w)^{-1} \in \cA(\Psi_n(z))$.
\end{corollary}

\begin{proof}
We know that $0_B(z) \in \cA(z)^{-1}$ by Lemma~\ref{0isatom-lem}.
If $u,v \in \W_n$ and $u \vartriangleleft_A v$, then 
$\Psi_n(u) \vartriangleleft_A w \vartriangleright_A \Psi_n(v)$ for some $w \in S_{2n}$.
Hence if $u,v \in \W_n$ and $u\sim_A v$ then $\Psi_n(u) \sim_A \Psi_n(v)$,
so the corollary follows from Theorem~\ref{0-z-thm}.
\end{proof}

If $w \in \cA(z)^{-1}$ for $z \in \I(\W_n)$,
then we define $\NDes_B(w)$ to be the subset of $\NDes(w)$ given by removing all pairs of the form $(a,-b)$ where $0<a<b$,
and we define $\NNeg_B(w)$ to be the set given by adding to $\NNeg(w)$
both $a$ and $b$ for each pair $(a,-b) \in \NDes(w)$ with $0<a<b$.
For example, if $w = \bar 1 67\bar2 348\bar 9 5\in \cA(\bar1\hs\bar76453\bar2\hs\bar8\hs\bar9)^{-1}$
then $\NDes(w) = \{(8,-9),(7,-2),(6,3)\}$ and $\NNeg(w) = \{1\}$,
so
$\NDes_B(w) = \{ (7,-2),(6,3)\}$ and $\NNeg_B(w) =\{1,8,9\}$.

Given $w \in \cA(z)$, one can recover $z$
by finding a reduced word $w=t_{i_1}t_{i_2}\cdots t_{i_l}$
and then calculating $z = t_{i_l} \circ \cdots \circ t_{i_2} \circ t_{i_1} \circ  t_{i_1}\circ t_{i_2}\circ \cdots \circ t_{i_l}$.
This naive algorithm is very inefficient.
The following result shows that  $z$ is in fact determined by the nested descent set of $w^{-1}$.
 
\begin{lemma}\label{prop-thm}
Let $z \in \I(\W_n)$ and $w \in \cA(z)^{-1}$.
Then $\Fix(z) = \NFix(w)$, $\Neg(z) = \NNeg_B(w)$ and
$\Pair(z) = \{ (a,b) : (b,a) \in \NDes_B(w)\}$.
 \end{lemma}

\begin{proof}
It straightforward to check that each claim holds if $w=0_B(z)$ by inspection and when $w \sim_B 0_B(z)$ 
by Corollary~\ref{ndes-cor2}(b)
and Lemma~\ref{ht-lem}, so is true for all $w \in \cA(z)^{-1}$ by Corollary~\ref{hz-cor1}.
\end{proof}

Suppose $M$ is a symmetric matching.
 We call a pair $\{i,j\} \in M$ a \emph{trivial block} if $i+j=0$.
 
\begin{definition}
Suppose $z \in \I(\W_n)$.
Define the \emph{shape} of $w \in \cA(z)^{-1}$
to be the symmetric perfect matching $\sh(w)$ 
whose nontrivial blocks include
$\{a,b\}$ and $\{-a,-b\}$ for each $(a,-b) \in \NDes(w)$ with $0<a<b$, and
whose trivial blocks are
the subsets $\{e,-e\}$ for each $e \in\NNeg(w)$.
\end{definition}

The shape $\sh(w)$ is a matching in the set $\Neg(z)\sqcup -\Neg(z) = \{ i \in [\pm n] : z(i) = -i\}$.
For example, if
 $w = \bar 1 67\bar2 348\bar 9 5 \in \cA(\bar1\hs\bar76453\bar2\hs\bar8\hs\bar9)^{-1}$
then  $\sh(w) = \{ \{ \bar 8, \bar 9\}, \{\bar 1, 1\}, \{8,9\}\}$.
In the next result, let  $\pair(M)$ be half the number of nontrivial blocks in a symmetric matching $M$.

\begin{proposition}\label{pair-prop}
If $z \in \I(\W_n)$ and $w \in \cA(z)^{-1}$ then 
\[\pair(\sh(w)) = |\{ a \in [n] : a \leq -z(a)\}|  - \ell_0(w).\]
\end{proposition}


\begin{proof}
Fix $z \in \I(\W_n)$. It is clear from \eqref{0B-eq} that if $w=0_B(z)$ 
then 
\[ \ell_0(w) = |\{ i \in [n] : w(i) < 0\}| = |\{ a \in [n] : a \leq -z(a)\}|  
\] while $\pair(\sh(w)) =0$.
It follows from Lemma~\ref{ht-lem} that if $v,w \in \cA(z)^{-1}$ and $v\scov w$ as in \eqref{black2}
then $\pair(\sh(w)) = \pair(\sh(v)) + 1$ and $\ell_0(w) = \ell_0(v) - 1$.
On other hand, if $v \vartriangleleft_A w$ then  $\pair(\sh(w)) = \pair(\sh(v))$ and $\ell_0(w) = \ell_0(v)$ by
Corollary~\ref{ndes-cor2}(b).
Since $0_B(z)$ is the unique minimum in $(\cA(z)^{-1},\ll_B)$ by Theorem~\ref{hz-thm}, 
the result follows.
\end{proof}
The following shows that $w\mapsto \sh(w)$ is a well-defined
map $ \cA(z)^{-1} \to \NCSM(z)$.

\begin{theorem}\label{nc-thm}
If $z \in \I(\W_n)$ and $w \in \cA(z)^{-1}$ then $\sh(w) \in\NCSM(z)$; in other words,
 the shape of $w$, which is symmetric and perfect by definition, is also noncrossing.
\end{theorem}

\begin{proof}
Let $z \in \I(\W_n)$.
The value of $\sh(\cdot)$ is constant on $\sim_A$-equivalence classes by Corollary~\ref{ndes-cor2}(b),
and
$\sh(0_B(z))   = \{ \{i,-i\} : i \in \Neg(z)\} \in \NCSM(z)$.
Suppose $v,w \in \cA(z)^{-1}$ and $i\in [n-1]$ are as in \eqref{black} so that $v\vartriangleleft^+_B w$
and $v_1=w_1<v_2=w_2 < \dots< v_{i-1} = w_{i-1} < v_i < 0$.
Set 
$a = -v_{i+1}= w_i$ and $b=-v_i = -w_{i+1}$ so that $0<a<b$.
Since $\sim_B$ is the transitive, symmetric closure of $\vartriangleleft_A$ and $\vartriangleleft^+_B$ and since $0_B(z) \in \cA(z)^{-1}$,
it suffices by Corollary~\ref{hz-cor1},  
to show that $\sh(v)$ is noncrossing if and only if $\sh(w)$ is noncrossing.

We have
$\sh(v) \setminus \sh(w) = \{ \{a,-a\},\{b,-b\}\}$ and $\sh(w)\setminus \sh(v) = \{ \{a,b\}, \{-a,-b\}\}$
by Lemma~\ref{ht-lem}.
If $\sh(w)$ is noncrossing,
then the only way $\sh(v)$ can fail to be noncrossing is if 
there exists a nontrivial block $\{c,d\} \in \sh(v) \cap \sh(w)$
with $0<c<a<b<d$.
But this would imply that both $(a,-b)$ and $(c,-d)$ were elements of $\NDes(w)$,
contradicting Lemma~\ref{cjw-lem}(3)
since $c\bar{d}a\bar{b}$ is not a subword of $w$.

Conversely, if $\sh(v)$ is noncrossing,
then $\sh(w)$ can fail to be noncrossing only is if there 
exists a trivial block $\{e,-e\} \in \sh(v) \cap \sh(w)$ with $a<e<b$.
But then we would have $\{a,b,e\}\subset \NNeg(v)$,
so 
Lemma~\ref{cjw-lem}(1) would imply that
$\bar{b}\bar{e}\hs\bar{a}$ is a subword of $v$, which is impossible as $\bar b$ and $\bar a$ are consecutive in $v$.
Thus $\sh(v)$ is noncrossing if and only if $\sh(w)$ is also.
\end{proof}

Let $z \in \I(\W_n)$ and $M \in \NCSM(z)$.
If $0_B(z,M)$ and $1_B(z,M)$ are contained in $\cA(z)^{-1}$, 
then they
have shape $\sh(0_B(z,M)) = \sh(1_B(z,M)) = M$,
as do all elements
 $w \in \cA(z)^{-1}$
with $0_B(z,M) \leq_A w \leq_A 1_B(z,M)$ by Corollary~\ref{ndes-cor2}(b).
The previous theorem shows that only noncrossing shapes are possible
for inverse atoms; the following confirms that all such shapes occur.
The map $\sh: \cA(z)^{-1} \to \NCSM(z)$ therefore 
 provides the bijection mentioned in Theorem~\ref{intro-thm2}.
 
 \begin{theorem}\label{minmax-thm}
 Let $z \in \I(\W_n)$. If $M \in \NCSM(z)$ then $0_B(z,M)$ and $1_B(z,M)$ are minimal and maximal elements of $(\cA(z)^{-1},<_A)$, respectively.
 Moreover, all
  minimal (respectively, maximal) elements in $(\cA(z)^{-1},<_A)$ have the form $0_B(z,M)$ (respectively, $1_B(z,M)$) for some $M \in \NCSM(z)$.
 \end{theorem}

 \begin{proof}
 Suppose $w \in \cA(z)^{-1}$ is minimal under $<_A$ and $M = \sh(w) \in \NCSM(z)$.
As $\sort_R(w)$ is then increasing by Lemma~\ref{minmax-lem},
it follows that $\NDes(w) = \Des(w)$. 
From this observation and Lemma~\ref{prop-thm}, it is an exercise to deduce that $w$ must be equal to $0_B(z,M)$.
If $w \in \cA(z)^{-1}$ is maximal under $<_A$, then it follows similarly that $w =1_B(z,\sh(w))$.

Choose an arbitrary matching $M \in \NCSM(z)$.
It remains  to show that $0_B(z,M)$ and $1_B(z,M)$ in fact belong to $\cA(z)^{-1}$.
 From Lemma~\ref{<A-lem}, Corollary~\ref{cV-cor}, and the previous paragraph, it is enough
to construct a single element $w \in \cA(z)^{-1}$ with $\sh(w) = M$.
We prove this by induction on the number of nontrivial blocks in $M$.
If $M$ has no nontrivial blocks then $0_B(z,M) = 0_B(z)$ has shape $M$ and belongs to $\cA(z)^{-1}$ by Corollary~\ref{isatom-lem}.
Otherwise, we can find a nontrivial block $\{a,b\} \in M$
with $0<a<b$ such that no $\{a',b' \} \in M$ has $0<a'<a<b<b'$.
Replacing $\{a,b\}$ and $\{-a,-b\}$ in $M$ 
by $\{a,-a\}$ and $\{b,-b\}$ yields another noncrossing matching $M' \in \NCSM(z)$
with strictly fewer nontrivial blocks.
Let $v=1_B(z,M')$. By induction, we may assume that $v \in \cA(z)^{-1}$.
Since $M$ is noncrossing, we must have
$\{a,b\} = \{a,a+1,\dots,b\} \cap \Neg(z,M')$,
so 
$v_1<v_2<\dots<v_i=-b < v_{i+1} = -a<0$ for some $i \in [n-1]$.
Replacing the subword $v_iv_{i+1}=\bar{b}\bar{a}$
in the one-line representation of $v$ by
 $a\bar b$
 gives  a signed permutation $w$
with $v \sim_B w \in \cA(z)^{-1}$ by Lemma~\ref{jump-lem}, and  it follows by Lemma~\ref{ht-lem} that $\sh(w) = M$.
\end{proof}

Given $z \in \I(\W_n)$ and $M \in \NCSM(z)$,  let $\cA(z , M) :=\{ w \in \cA(z) : \sh(w^{-1})=M\}$.

\begin{corollary}
Suppose $z \in \I(\W_n)$. If $M \in \NCSM(z)$ then
\[
\cA(z,M) = \{ w \in \W_n : 0_B(z,M) \leq_A w^{-1}\}
= \{ w \in \W_n :  w^{-1} \leq_A 1_B(z,M)\}.\]
Moreover, we have $\cA(z) = \bigsqcup_{M \in \NCSM(z)} \cA(z,M)$.
\end{corollary}

\begin{corollary}
If $z \in \I(\W_n)$ then $(\cA(z)^{-1},<_A)$ is connected if and only if $\neg(z) \leq 1$.
\end{corollary}

\begin{proof}
There are $|\NCSM(z)|$ components in $(\cA(z)^{-1},<_A)$, which is 1 if and only if $\neg(z) \leq 1$.
\end{proof}

Let $\NCSM^k(z)$ for $z \in \I(\W_n)$ be the set of  matchings in $ \NCSM(z)$ with at most $k$ trivial blocks.
We have $\NCSM^0(z) = \NCSM^1(z)$ if $\neg(z)$ is even and $\NCSM^0(z) = \varnothing$ if $\neg(z)$ is odd.

\begin{corollary}\label{max-cor}
Let $z \in \I(\W_n)$. The permutations $1_B(z,M)$ for $M \in \NCSM^1(z)$ are the maximal elements in $\cA(z)^{-1}$
under both atomic orders $<_B$ and $\ll_B$.
Moreover, $\cA(z)^{-1}$ is the union of the lower intervals in $(\W_n,\ll_B)$ bounded above by these elements.
\end{corollary}

\begin{proof}
Each maximal element in $\cA(z)^{-1}$ under either atomic order 
is necessarily of the form $1_B(z,M)$ for some $M \in \NCSM(z)$ by Theorem~\ref{minmax-thm}.
If $M \in \NCSM(z)$ has $k$ trivial blocks, then we can write $1_B(z,M)=\bar{a_k}\hs \cdots \bar{a_2}\hs\bar{a_1}b_1b_2\cdots b_{n-k}$
where $0<a_1<a_2<\dots<a_k$, $0<b_1$, and $b_1b_2\cdots b_{n-k}$ contains no consecutive negative numbers.
Evidently $1_B(z,M)$ is maximal under $<_B$ (and also $\ll_B$) if and only if $k<2$.
The last assertion in the corollary holds by Lemma~\ref{jump2-lem}.
\end{proof}


\begin{corollary}\label{cat-cor}
If $z \in \I(\W_n)$ and $m = \lceil \frac{1}{2}\neg(z)\rceil$, then
the number of  elements in $\cA(z)^{-1}$ that are maximal under $<_B$ (equivalently, $\ll_B$)
is the $m$th Catalan number $C_m = \frac{1}{m+1} \binom{2m}{n}$.
\end{corollary}

\begin{proof}
There is a bijection from $\NCSM^1(z)$ to the set of noncrossing perfect matchings in
 $[2m]$, whose enumeration by $C_m$ is well-known:
 remove all blocks without positive elements from $M \in \NCSM^1(z)$ and standardize the numbers in the remaining blocks 
to be $1,2,\dots, 2m$.
\end{proof}

\section{Rank functions}\label{rank-sect}

In this section we show that the atomic orders $<_B$ and $\ll_B$ are graded.
Fix $z \in \I(\W_n)$ and $w \in \cA(z)^{-1}$. 
Define $\offset_A(w)$ to be the number of pairs
$((b_1,a_1),(b_2,a_2)) \in \NDes(w)\times \NDes(w)$ with $a_1<a_2< b_2<b_1$.
Let  
$L = \{ b : (b,a) \in \NDes(w)\text{ for some }a\}
$ and
$ R =\{w_1,w_2,\dots,w_n\} \setminus L$
and define $\rank_A : \cA(z)^{-1} \to \ZZ$ to be the function with
\[ \rank_A(w) := \inv(w|_{ R}) - \inv(w|_L) + \offset_A(w) \in \ZZ,\]
where $w|_{ R}$ and $w|_L$ are the words formed  from $w_1w_2\cdots w_n$ by omitting
all entries not in $ R$ and $L$, respectively, and 
$\inv(v) = |\{ (i,j) \in [k]\times [k]: i<j\text{ and }v_i>v_j\}|$ for a word $v=v_1v_2\cdots v_k$.
If $z=\bar{5}\hs\bar{4}\hs\bar{3}\hs\bar{2}\hs\bar{1}$ and $w = \bar{3} 4 \bar5 1 \bar 2 \in \cA(z)^{-1}$, for example,
then $\NDes(w) = \{(1,\bar 2),(4,\bar 5)\}$, $L = \{1,4\}$,  $R = \{ \bar 2,\bar 3,\bar 5\}$,
 $w|_R = \bar{3}\hs\bar{5}\hs\bar{2}$, $w|_L = 41$, and $\inv(w|_R) = \inv(w|_L) =\offset_A(w) =\rank_A(w) = 1$.
As a consequence of Corollary~\ref{cV-cor}, the following result is equivalent to \cite[Lemma 6.13]{HMP2}.
 
\begin{proposition}\label{arank-prop}
Let $z \in \I(\W_n)$.   If $v,w \in \cA(z)^{-1}$ and $v \vartriangleleft_A w$, then $\rank_A(w) = \rank_A(v) + 1$,
and an element  $w \in \cA(z)^{-1}$ is minimal relative to $<_A$ if and only if $\rank_A(w) = 0$.
\end{proposition}

\begin{proof}
The first claim is immediate from the
 way we define $\vartriangleleft_A$ and $\NDes(w)$.
By Lemma~\ref{minmax-lem}, $w \in \cA(z)^{-1}$ is minimal relative to $<_A$ if and only if $\inv(w|_R)  = 0$
 and $\inv(w|_L)=\offset_A(w)$.
 \end{proof}
 
Still with $z \in \I(\W_n)$ and $w \in \cA(z)^{-1}$,
define $\offset_B(w)$ to be the number of pairs of nested descents
$((b_1,a_1),(b_2,a_2)) \in \NDes(w)\times \NDes(w)$ satisfying $ a_1 \leq a_2 <-b_1 < 0 < b_1 \leq b_2 $.
Set 
\[ 
\rank_B(w) := \rank_A(w) + \offset_B(w) \in \NN.
\]
The value of $\offset_B(w)$  is the sum of three quantities:
 the number of descents $(a,-b) \in \NDes(w)$ with $a<b$, 
 the number of pairs $(a_1,-b_1),(a_2,-b_2) \in \NDes(w)$
 with $a_1 <a_2 < b_2 < b_1$, 
 and
 the number of pairs $(a_1,-b_1),(b_2,-a_2) \in \NDes(w)$
 with $a_1<a_2<\min\{b_1,b_2\}$.
For example, if  $w = 1\bar 5 2 \bar 3 6\bar 4$ 
then $\NDes(w) = \{(1,\bar 5), (2,\bar 3),(6,\bar 4)\}$,
$w|_L = 126$, $w|_R = \bar{5}\hs\bar{3}\hs\bar{4}$, $\inv(w|_L) = 0$, $\inv(w|_R) = \offset_A(w) = 1$, $\rank_A(w) = 2$,
 $\offset_B(w) =  4$, and $\rank_B(w) = 6$.

The function $\offset_B(\cdot)$ is constant on $\sim_A$-equivalence classes by Corollary~\ref{ndes-cor2}(b),
so we  have $\rank_B(w) = \rank_B(v) + 1$ for $v,w \in \cA(z)^{-1}$ with  $v\vartriangleleft_A w$ by Proposition~\ref{arank-prop}. In addition:

\begin{proposition}\label{brank-prop}
Let $z \in \I(\W_n)$. 
If $v,w \in \cA(z)^{-1}$ and $v \scov w$, then $\rank_B(w) = \rank_B(v) + 1$,
and an element  $w \in \cA(z)^{-1}$ is minimal relative to $\ll_B$ if and only if $\rank_B(w) = 0$.
\end{proposition}

\begin{proof}
Fix $v,w \in \cA(z)^{-1}$. 
First assume  $v\vartriangleleft^+_B w$ and let 
 $i \in [n-1]$ be as in \eqref{black}.
 Let $a = -v_{i+1}= w_i$ and $b=-v_i = -w_{i+1}$
so that $0<a<b$ and $\NDes(w) = \NDes(v) \sqcup \{(a,-b)\}$.
Then $v_j=w_j \in \NNeg(v) \cap \NNeg(w)$ for $1 \leq j < i$.
From Theorem-Definition~\ref{recdes-thmdef} and Lemma~\ref{cjw-lem}, we deduce that 
the difference $A=\inv(v|_R) - \inv(w|_R)$ is the number of pairs $(y,x) \in \NDes(v)$ with $x<-a$,
the difference $B=\inv(w|_L) - \inv(v|_L)$ is the number of pairs $(y,x) \in \NDes(v)$ with $y<a$,
and
the difference $C=\offset_A(w) - \offset_A(v)$ 
is the number of pairs $(y,x) \in \NDes(v)$ with either $x<-b<a<y$ or $-b<x<y<a$.
On the other hand, 
the difference $D=\offset_B(w) - \offset_B(v) - 1$
is the number of pairs $(y,x) \in \NDes(v)$ with $-b < x < -a <0< a<y$.
To prove that $\rank_B(w) = \rank_B(v) + 1$,
it suffices to show that $A+B=C+D$.
This is straightforward on noting that 
$\NDes(w)$ contains no elements  $(y,x)$ with $x<-b<y<a$ by Lemma~\ref{cjw-lem}(3),
or with $-b < x< -a < 0 < y < a$ since $\sh(w)$ is noncrossing.

Next suppose that $v\scov w$ and let $i \in [n-1]$ be as in \eqref{black2}.
Since we have $|v_{i+1}| < |v_i| = \max\{ |v_1|, |v_2|,\dots,|v_i|\}$, and since inverse atoms do not have consecutive 321- or $\bar1\hs\bar2$-patterns,
it follows by Lemma~\ref{cjw-lem} 
that there are chains of elements $v=v^0 \vartriangleleft_A v^1 \vartriangleleft_A \cdots \vartriangleleft_A v^k $
and $w=w^0 \vartriangleleft_A w^1 \vartriangleleft_A \cdots \vartriangleleft_A w^k $ 
with $v^k \vartriangleleft^+_B w^k$.
By Proposition~\ref{arank-prop} and the previous paragraph, we deduce that $\rank_B(w) = \rank_B(w^k) - k = \rank_B(v^k) + 1-k = \rank_B(v)+1$.

The last assertion follows from Proposition~\ref{arank-prop} since $\offset_B(w)=0$ only if $\sh(w)$ is trivial.
\end{proof}


\begin{corollary}\label{rank-cor}
If $z \in \I(\W_n)$ then $(\cA(z)^{-1},<_A)$, $(\cA(z)^{-1},<_B)$, and $(\cA(z)^{-1},\ll_B)$ are graded.
\end{corollary}

A notable property of  $(\cA(z)^{-1},<_B)$ and $(\cA(z)^{-1},\ll_B)$, apparent in Example~\ref{intro-ex} and Figure~\ref{fig2},
is that these  connected, graded posets have unique elements of maximal rank.
To prove that this holds in general, we introduce a third variation of the covering relation $\vartriangleleft_B$.

Define $\bcov$ to be the relation on $n$-letters words that has $v\bcov w$ if for some indices $1 \leq i < j < n$ 
and some positive numbers $a,b,c$ it holds that 
$v_k = w_k$ for $k \notin \{ i,j,j+1\}$ while
\be\label{bcov-eq}
v_iv_jv_{j+1} = \bar c a \bar b,\quad w_iw_jw_{j+1} = \bar a b \bar c,\quand
\max\{ |w_1|,|w_2|,\dots,|w_{j-1}|\} = a <b<c.\ee
Equivalently, we have $\cdots \bar c \cdots a\bar b \cdots \bcov \cdots \bar a \cdots b\bar c \cdots$
whenever the corresponding ellipses mask identical subsequences 
and it holds that $0<a<b<c$ and all hidden
letters to the left of $a$ in the first word (equivalently, to the left of $b$ in the second word) have absolute value less than $a$.
As usual, we apply this relation to signed permutations via their one-line representations.

\begin{proposition}\label{bl-prop}
Let $z \in \I(\W_n)$. Suppose  $v,w \in \W_n$ are such that $v\bcov w$. Then $v \in \cA(z)^{-1}$ if and only if $w \in \cA(z)^{-1}$,
and if this holds then $\rank_B(w) = \rank_B(v) +1$.
\end{proposition}

\begin{proof}
Let $1\leq i < j <n$ be such that \eqref{bcov-eq} holds. We prove the result by induction on $i+j$.
Our argument relies on two base cases. When $i = j-1 =1$, 
 the result follows from Lemma~\ref{<B-lem} and Proposition~\ref{brank-prop}  since  
$
\bar c \bar b \bar a  \vartriangleleft_B \bar c a \bar b
$
and
$
\bar c \bar b \bar a  \vartriangleleft_B b\bar c\hs \bar a \vartriangleleft_A \bar a b \bar c
$ 
for any $0<a<b<c$.
When $i=j-1=2$ the lemma  follows similarly
from the fact that if $0<a<b<c<d$ then
$\bar d  \bar c \bar b \bar a \vartriangleleft_B \bar d  b \bar c\hs \bar a \vartriangleleft_A \bar d \bar a b \bar c \vartriangleleft_B a \bar d b \bar c$
and
$\bar d  \bar c \bar b \bar a  \vartriangleleft_B c \bar d  \hs \bar b \bar a \vartriangleleft_A \bar b c \bar d  \bar a\vartriangleleft_A \bar b \bar a c \bar d \vartriangleleft_B a \bar b c \bar d$.

For the inductive step, let $0<a = v_j = -w_i < b = -v_{j+1} = w_j < c=-v_i = -w_{j+1}$. 
First suppose $i < j-1$. Define $v'$ from $v$ by replacing the subword $v_{j-1}v_jv_{j+1}$ by $v_jv_{j+1}v_{j-1}$, and form $w'$ from $w$ similarly.
Since 
  $v_{j-1}=w_{j-1}$ is less than $a$ in absolute value, it follows that $v'\vartriangleleft_A v $ and $w'\vartriangleleft_A w$ and $v' \bcov w'$.
By induction, the proposition holds with $v$ and $w$ replaced by $v'$ and $w'$, so by Lemma~\ref{<B-lem} and Proposition~\ref{brank-prop}, the result also holds for $v$ and $w$.

Suppose alternatively that $2 < i=j-1$. 
We may assume that at least one of $v$ or $w$ belongs to $\cA(z)^{-1}$. 
Since inverse atoms do not have consecutive 321- or $\bar1\hs\bar2$-patterns and since
all numbers in the subwords $v_1v_2\cdots v_{i-1} = w_1w_2\cdots w_{i-1}$
have absolute value less than $a$, it must hold that 
$v_i < v_{i-2} < v_{i-1}$ and $w_i < w_{i-2} <w_{i-1}$.
Define $v'$ from $v$ by replacing the subword $v_{i-2}v_{i-1}v_i$ by $v_{i-1}v_iv_{i-2}$,
and form $w'$ from $w$ similarly.
We once again have $v'\vartriangleleft_A v $ and $w'\vartriangleleft_A w$ and $v' \bcov w'$, and 
may deduce 
that the proposition holds by induction.

If $i=j-1 \in \{1,2\}$, finally, then we are in one of the base cases already considered.
\end{proof}

Fix $z \in \I(\W_n)$ with $m=\neg(z)$, and suppose $\{ i \in [\pm n] : z(i) = -i\} = \{a_1<a_2<\dots<a_{2m}\}$.
 Let $M_{\max} :=\{ \{a_1,a_2\}, \{a_3,a_4\},\dots,\{a_{2m-1},a_{2m}\}\}\in \NCSM(z)$
and define 
\be
\label{1B-eq}
1_B(z) := 1_B(z,M_{\max}).
\ee
For example, if $z =  \bar{1}\hs\bar2\hs\bar3\hs\bar4$
then $1_B(z) = 1\bar 2 3\bar 4$
while if $z =  \bar{1}\hs\bar2\hs\bar3\hs\bar4\hs\bar{5}$
then $1_B(z) = \bar1 2 \bar3 4\bar 5$.
Theorem~\ref{minmax-thm} implies that $1_B(z) \in \cA(z)^{-1}$.
We have $0_B(z) = 0_B(z,M_{\min})$ for $M_{\min} = \{ \{ i,-i\} : i \in \Neg(z)\}$.
Define $\lll_B$ to be the transitive closure of the three relations $\vartriangleleft_A$, $\scov$, and $\bcov$.
For lack of a better term, we refer to $\lll_B$ as the \emph{very strong atomic order} of type B.

\begin{proposition}
Restricted to $\cA(z)^{-1}$ for any $z \in \I(\W_n)$, the relation $\lll_B$ is a bounded, graded partial order,
whose unique minimum is $0_B(z)$ and whose unique maximum is $1_B(z)$.
\end{proposition}

This proposition is reminiscent of Stembridge's results about the top and bottom classes of a permutation \cite[Propositions 4.1 and 4.2]{StemRed}.

\begin{proof}
Let $z \in \I(\W_n)$. The claim that $\lll_B$ is a graded partial order on $\cA(z)^{-1}$ with $0_B(z)$
as its unique minimum is immediate from Propositions~\ref{brank-prop} and \ref{bl-prop}.
Let $M \in \NCSM^1(z)$. By Corollary~\ref{max-cor}, it is enough to show that $1_B(z,M)$ 
is not maximal under $\lll_B$  if $M \neq M_{\max}$.

To this end, assume $M\neq M_{\max}$ and write $\{ i \in [\pm n] : z(i) = -i\} = \{a_1<a_2<\dots<a_{2m}\}$ as above.
Since $M$ is  noncrossing and symmetric with at most one trivial block,
there must exist a pair of nesting blocks $\{a_i, a_{k+1}\},\{a_{j},a_{k}\} \in M$ with $i < j < k$ and $m < j$.
If possible, choose these blocks such that $i = 2m-k$ so that $\{a_i,a_{k+1}\}$ is trivial; this is always possible if $M$ has a trivial block
distinct from $\{a_m,a_{m+1}\}$.
Let $a=a_j$, $b=a_k$, and $c=a_{k+1}$ so that $0<a<b<c$.
Since $\{a_i,a_{k+1}\}$ is then the only block $\{x,y\}\in M$ with $x<a<b< y$,
the one-line representation of $1_B(z,M)$  has the form $ \cdots \bar c \cdots a \bar b\cdots $
and all letters to the left of $a$ in this word have absolute value less than $a$.
Hence $1_B(z,M)$ is not maximal with respect to $\lll_B$, as needed.
\end{proof}

\begin{figure}[h]
\begin{center}
\begin{tikzpicture}[scale=0.9]
\node (0) at (2,7.2) {$1\bar 2 3 \bar 4$};  
\node (1a) at (0,6) {$13\bar4\hs\bar2$};  
\node (1b) at (2,6) {$\bar2\hs\bar13\bar4$};  
\node (1c) at (4,6) {$1\bar42\bar3$};  
\node (2a) at (0,4.8) {$3\bar41\bar2$};  
\node (2b) at (2,4.8) {$\bar23\bar4\hs\bar1$};  
\node (2c) at (4,4.8) {$\bar4\hs\bar12\bar3$};  
\node (3a) at (0,3.6) {$\bar4\hs\bar31\bar2$};  
\node (3b) at (2,3.6) {$3\bar4\hs\bar2\hs\bar1$};  
\node (3c) at (4,3.6) {$\bar42\bar3\hs\bar1$};  
\node (4a) at (2,2.4) {$\bar4\hs\bar3\hs\bar2\hs\bar1$};  
\draw  [->,dashed] (4a) -- (3a);
\draw  [->,dashed] (4a) -- (3b);
\draw  [->,dashed] (4a) -- (3c);
\draw  [->,dashed] (3a) -- (2a);
\draw  [->,dashed] (3b) -- (2a);
\draw  [->] (3b) -- (2b);
\draw  [->] (3c) -- (2c);
\draw  [->,dotted] (3c) -- (2b);
\draw  [->] (2a) -- (1a);
\draw  [->] (2b) -- (1b);
\draw  [->,dashed] (2c) -- (1c);
\draw  [->,dotted] (2c) -- (1b);
\draw  [->] (1a) -- (0); 
\draw  [->,dashed] (1b) -- (0); 
\draw  [->,dotted] (1c) -- (0);
\end{tikzpicture}
\end{center}
\caption{The Hasse diagram of $(\cA(z)^{-1}, \lll_B)$
for $z =\overline{1}\hs \overline{2} \hs\overline{3}\hs \overline{4} \in \I(\W_4)$; compare with Example~\ref{intro-ex}.
The solid, dashed, and dotted arrows correspond to $\vartriangleleft_A$, $\scov$, and $\bcov$, respectively.}
\end{figure}

\begin{corollary}
If $z \in \I(\W_n)$ then $1_B(z)$ is the unique element at which $\rank_B : \cA(z)^{-1} \to \NN$ attains its maximum value.
\end{corollary}

\section{Relative shapes}\label{relshape-sect}

There is a natural generalization of the shape of 
an inverse atom which will be needed in Section~\ref{geom-sect}.
For $y,z \in \cI(\W_n)$ define the set of \emph{relative atoms} $\cA(y,z)$
to consist of the minimal length elements $w \in \W_n$ with $z = w^{-1}\circ y \circ w$.
This set may be empty although $\cA(z) = \cA(1,z)\neq \varnothing$.
Recall that the longest element $w_0 = \wB \in\W_n$  is central.
We write $-w := w_0 w  =w w_0$.

\begin{lemma}
\label{terminal_atoms}	
If $y,z \in \I(\W_n)$ then   $\cA(y,z)^{-1} = \cA(-z,-y)$.
\end{lemma}

\begin{proof}
This holds since $w_0$ is central and $t_i \in \DesR(w)$ if and only if $t_i \notin \DesR(-w)$.
\end{proof}

Suppose $y,z \in \cI(\W_n)$ and $w \in \cA(y,z)$. 
Let $\a = (a_1,a_2,\dots,a_l) \in \cR(w)$, which we identify with the set of nonnegative integer
sequences such that 
$w=t_{a_1}t_{a_2}\cdots t_{a_l}$ is a reduced word. 
Fix $i \in [l]$ and let $v = t_{a_1}t_{a_2}\cdots t_{a_{i-1}}$.
If $a_i$ and $a_{i+1}$ are positive fixed points of $v^{-1}\circ y \circ v$ then we set
 \[\sh_i(\a;y) := \left\{\left\{ v(a_i),v(a_{i}+1)\right\}, \left\{ -v(a_i),-v(a_{i}+1)\right\}\right\}.\]
If $a_i=0$ and  $ 1$ is a fixed point of $v^{-1}\circ y \circ v$ then we  
 define 
 \[\sh_i(\a;y) := \left\{\left\{ v(-1),v(1)\right\}\right\}.\]
 Otherwise, let $\sh_i(\a;y) :=\varnothing$.
The union $\sh_1(\a;y)\sqcup \sh_2(\a;y) \sqcup \cdots \sqcup \sh_l(\a;y)$ is disjoint and gives a perfect symmetric matching on 
 a subset $X\subseteq -\Fix(y)\sqcup \Fix(y)$. We set
 \[ \sh_\infty(\a; y) := \{ \{-j,j\} : j \in  \Fix(y) \setminus X\}
 \]
 and define $ \sh(\a;y) := \sh_1(\a;y)\sqcup \sh_2(\a;y) \sqcup \cdots \sqcup \sh_l(\a;y) \sqcup \sh_\infty(a;y)$.
 
\begin{example}
Suppose $y = 1234\bar5 $ and $z =  \bar 1 \hs \bar 2 \hs\bar 3 4\bar 5$.
Then $w':=\bar 3 \hs \bar 2 \hs \bar 1 45$ and $w'':= \bar 3 1 \bar 2 45$ are both in $\cA(y,z )$.
The expressions $w'= t_0 t_1 t_2 t_0 t_1 t_0$ and $w'' =  t_1 t_0 t_1 t_2 t_1 t_0$
are reduced and
\[\ba
\sh(\a;y) &= \{\{-1,1 \}, \{-2,2\},\{-3,3\},\{-4,4\}\}&\text{for $\a=(0,1,2,0,1,0)$}, \\
\sh(\b;y) &= \{\{-2,-1 \}, \{1,2\},\{-3,3\},\{-4,4\}\}&\text{for $\b=(1,0,1,2,1,0)$.}
\ea
\]
\end{example}

Recall the definition of $0_B(z)$ from \eqref{0B-eq}.
One has $0_B(z)=1$ if and only if $z=1$.

 \begin{theorem}\label{rel-prop1}
 Suppose $y,z \in \I(\W_n)$, $w \in \cA(y,z)^{-1}$, and $\a \in \cR(w)$.
 Then $w\cdot 0_B(-z) \in \cA(-y)$ and $\sh(\a;y) = \sh(w\cdot 0_B(-z))\in \NCSM(-y)$.
Thus, if $z= w_0:=\wB$  then $\sh(\a;y)=\sh(w)$.
 \end{theorem}
 
 \begin{proof}
First assume $z = w_0$. Then $w^{-1} \in \cA(-y)$ by Lemma~\ref{terminal_atoms}
and we want to show that $\sh(\a;y) = \sh(w)$ for any $\a = (a_1,a_2,\dots,a_l) \in \cR(w)$.
 Let $i = a_1$, $s = t_i$, $\tilde \a = (a_2,\dots,a_l)$, and $\tilde y := s\circ y\circ s$.
 We may assume by induction that $\sh(\tilde\a;\tilde y)=\sh(sw)$. There are three cases.
 
Suppose $i=0$ and $ 1\in \Fix(y)$. Then $1 \in \Fix(-\tilde y)$, which is equal to $\NFix(sw)$ by Lemma~\ref{prop-thm},
  and $\pm 1$ are not in any of the blocks of $\sh(sw)$ so $\sh(\a;y) = \{ (-1,1)\} \sqcup \sh(sw)$.
   It is easy to see from the definition of the nested descent graph in Section~\ref{nest-sect} that in this case $\NFix(w) = \NFix(sw)\setminus\{1\}$, $\NNeg(w) = \NNeg(sw)\sqcup\{1\}$, and $\NDes(w) = \NDes(sw)$, 
  so we also have $\sh(w) =  \{ (-1,1)\} \sqcup \sh(sw)$. Thus $\sh(\a;y) = \sh(w)$ as claimed.
  
  Next assume $i>0$ and $\{i,i+1\} \subset \Fix(y)$. Then 
  $(-i,i+1) \in \Pair(-\tilde y)$ so $(i+1,-i) \in \NDes(sw)$ by Lemma~\ref{prop-thm}.
  Since the numbers $\pm i$ and $\pm(i+1)$ are not in any of the blocks of $\sh(sw)$,
  we also have $\sh(\a;y) = \{ \{-i-1,-i\}, \{i,i+1\}\} \sqcup \sh(sw)$.
  But now it is clear from the definition of the nested descent graph
  that $\NFix(w) = \NFix(sw)$, $\NNeg(w) = \NNeg(sw)$, and $\NDes(w) = (\NDes(sw)\setminus \{(i+1,-i)\}) \sqcup \{(i,-i-1)\}$, so $\sh(w)=\{ \{-i-1,-i\}, \{i,i+1\}\} \sqcup \sh(sw)=\sh(\a;y)$.
  
If neither of these cases occur, then either $i=0$ but $ 1\notin\Fix(y)$,
or $i>0$ but $\{i,i+1\}\not\subset\Fix(y)$.
Here, it is easy to see that $\sh(\a;y) = \{ \{s(b), s(b')\} : \{b,b'\} \in \sh(\tilde\a;\tilde y)\}$
and, again using Lemma~\ref{prop-thm}, that $\sh(w) =  \{ \{s(b), s(b')\} : \{b,b'\} \in \sh(sw)\}$,
so $\sh(\a;y) = \sh(w)$ in all cases.

Now suppose $z \in \I(\W_n)$ is arbitrary. Then
$w \cdot 0_B(-z) \in \cA(-y)^{-1}$ since 
$w \in \cA(-z,-y)^{-1}$ and $0_B(-z)\in \cA(-z)^{-1}=\cA(z,w_0)$, and
since for any $g \in \cA(-z)$ we have
\[
\cA(-z,-y) = \left\{  g v : v \in \cA(-y) \text{ and }
\ell( g  v) = \ell(v) - \ell(g)
\right\}\]
by \eqref{szs-eq}.
Let $\b \in \cR(0_B(-z))$.
Since $\sh(0_B(-z)) = \sh(\b;z) = \{ \{-i,i\} : i \in \Fix(z)\}$
by the claim proved above, 
it is clear   
that if $\a = (a_1,a_2,\dots,a_l) \in \cR(w)$
then $\sh(\a\b;y)$ is formed from $\sh_1(\a;y)\sqcup\sh_2(\a;y)\sqcup \cdots \sqcup \sh_l(\a;y)$
by adding a sequence of trivial blocks $\{-i,i\}$ with $i \in \Fix(y)$.
But this is precisely the definition of $\sh(\a;y)$,
so $\sh(\a;y) = \sh(\a\b;y) = \sh(w\cdot 0_B(-z))$.
%
 \end{proof}
 
 The following slightly technical property will be needed in Section~\ref{sect-Brion-B}.
 
 \begin{lemma}\label{tbbb-lem}
  Let $y,z \in \I(\W_n)$, $w \in \cA(y,z)$, and $\a  \in \cR(w)$.
Suppose $|\sh_i(\a;y) | = |\sh_j(\a;y)| = 1$ for $i<j<\infty$. Then the unique pairs $\{-p,p\} \in \sh_i(\a;y) $ and $\{-q,q\} \in \sh_j(\a;y) $ have $|p|<|q|$, and
it also holds that $|q|<|r|$ for all $\{-r,r\} \in \sh_\infty(\a;y)$.
 \end{lemma}
 
  \begin{proof}
 We may assume $z = \wB$ so that $w \in \cA(-y)^{-1}$.
   Write $\a = (a_1,a_2,\dots,a_l)$ and suppose 
  $\sh_i(\a;y) =\{\{-p,p\}\}$ and $ \sh_j(\a;y)=\{\{-q,q\} \}$ for $0<q<p$.
  We argue that this leads to a contradiction.
By definition, we must have $a_i=a_j=0$ along with
  $t_{a_1}t_{a_2}\cdots t_{a_i}(1) = -p$ and   $t_{a_1}t_{a_2}\cdots t_{a_j}(1) = -q$.
  Since the chain $1 < t_{a_1} < t_{a_1}t_{a_2} < \dots< w$ is strictly
  increasing, it follows that  $w$ must have the form
  $w = \cdots \bar q \cdots \bar p \cdots $. But this is impossible if $0<q<p$ by Proposition~\ref{12bar-prop}
  since both $\bar p$ and $\bar q$ must be present in the global sink $\xi(w)$ for the nested descent graph of $w$.
 \end{proof}
 

%
%
%

Fix an integer $ 0\leq k \leq n$ and define $y^{\BC}_{n,k} \in \I(\W_n)$ and $g^{\BC}_{n,k} \in \I(\W_n)$ by
\be\label{ybc-eq} y^\BC_{n,k}:= (-1,1)(-2,2)\cdots(-k,k)
\quand g^{\BC}_{n,k} :=(-1,k)(-2,k-1)\cdots(-k,1).\ee Then
$y^{\BC}_{n,0} =g^{\BC}_{n,0}= 1$ and $y^{\BC}_{n,n} = \wB$ while $g^{\BC}_{n,k} = 0_B(y^{\BC}_{n,k}) \in \cA(y^{\BC}_{n,k}) \cap \cA(y^{\BC}_{n,k})^{-1}$.

\begin{lemma}\label{gk-lem}
Let $w \in \W_n$.
Then $\ell(w \cdot g^\BC_{k,n}) = \ell(w) - \ell(g^\BC_{k,n})$ if and only if $w(i)<0$ for all $i \in [k]$.
\end{lemma} 

\begin{proof}
This holds by induction since $g^\BC_{k,n} = g^{\BC}_{k-1,n} t_{k-1}\cdots t_2t_1t_0$ and $\ell(g^\BC_{k,n}) - \ell(g^\BC_{k-1,n})=k$.
\end{proof}

Fix $z \in \I(\W_n)$. Let $\cAk(z) := \cA(y^\BC_{n,k},z)$.
If $w \in \cAk(z)^{-1}$ then 
$w \cdot g^{\BC}_{n,k}  \in \cA(z)^{-1}
$
so we may define 
$\sh_k(w) := \sh( w \cdot g^{\BC}_{n,k}) .$ 
Let 
$
\cAk(z,M) :=  \{ w \in \cAk(z) : \sh_k(w^{-1}) = M\}$ for $M \in \NCSM(z)$
and write $\NCSM_k(z)$ for the set of matchings $M \in \NCSM(z)$ with at least $k$ trivial blocks.

\begin{lemma}\label{cak-cor}
If $z \in \I(\W_n)$ and  $M \in \NCSM(z)$ then
\[\cAk(z,M) 
= \left\{ g^\BC_{k,n}\cdot w : w \in  \cA(z,M) \text{ with }w^{-1}(1) <w^{-1}(2) < \dots < w^{-1}(k)<0\right\}.\]
Moreover, the set $\cAk(z,M)$ is nonempty if and only if $M\in \NCSM_k(z)$.
\end{lemma}

\begin{proof}
Let $y =y^{\BC}_{k,n}$ and $ g= g^\BC_{k,n}$. 
Then $w \in \cAk(z)$ if and only if $gw \in \cA(z)$ and $\ell(gw) = \ell(g) + \ell(w)$.
This holds precisely when $w \in \cA(z)$ and $w^{-1}(i) < 0$ for all $i \in [k]$ by Lemma~\ref{gk-lem}.
The latter condition is equivalent $w^{-1}(1) <w^{-1}(2) < \dots < w^{-1}(k)<0$
for $w\in \cA(z)$ by
Proposition~\ref{12bar-prop}.

Let $ g= g^\BC_{k,n}$.
If $w \in \cAk(z,M)^{-1}$ then $wg \in \cA(z)^{-1}$ and
 $wg(1) < wg(2) <\dots < wg(k) < 0$, so the set 
 $\NNeg(gw)$ has size at least $k$
and $\sh(gw) = \sh_k(w)=M$ has at least $k$ trivial blocks.
Conversely, it is easy to check that if $M \in \NCSM_k(z)$ then 
$g \cdot 1_B(z,M)^{-1} \in \cAk(z,M)$.
\end{proof}


%

\begin{corollary}
If $z \in \I(\W_n)$ then $\cAk(z) = \bigsqcup_{M \in \NCSM_k(z)} \cAk(z,M)$.
\end{corollary}


One can efficiently generate the sets $\cAk(z,M)$ in the following way.
Suppose the $k$ outermost trivial blocks of $M \in \NCSM_k(z)$
are $\{\pm i_1 \}, \{\pm i_2\},\dots,\{\pm i_k\}$ where $0<i_1<i_2<\dots<i_k$.
Each of the barred letters $\bar{i_1}$, $\bar{i_2}$, \dots, $\bar{i_k}$
appears in the one-line 
representation of $0_B(z,M)$; form $0_{B}^{(k)}(z,M)$
from $0_B(z,M)$ by removing these letters and then
prepending $i_1i_2\cdots i_k$.
For example, if $z = \bar1\hs\bar2\hs\bar3\hs\bar4\hs\bar5$ and $M = \{ \{-1,1\}, \{-2,2\}, \{-3,3\}, \{-4,-5\}, \{4,5\}\}$,
then 
\[0_B(z,M) =\bar3 4\bar5\hs \bar 2\hs \bar1,
\quad 0_{B}^{(1)}(z,M)= 3 4\bar5\hs \bar 2\hs \bar1,
\quad 0_{B}^{(2)}(z,M)= 23  4\bar5\hs \bar1,
\quand 0_{B}^{(3)}(z,M)= 123 4\bar5.
\]
Let $1_{B}^{(k)}(z,M) := 1_B(z,M) \cdot  g^{\BC}_{k,n}$.
Finally, define $<_A^{(k)}$ to be the transitive closure of the relation
$\cdots cab\cdots \vartriangleleft_A \cdots bca\cdots$ as in \eqref{<A-eq}, but with the extra requirement that the subwords $cab$ and $bca$ with $a<b<c$
occur in positions $i,i+1,i+2$ for $i>k$. 
Note that ${<_A^{(k)}}={<_A}$ if $k=0$.

\begin{proposition}\label{kkpre-prop}
If $z \in \I(\W_n)$ and $M \in \NCSM_k(z)$ then 
\[ \cAk(z,M)^{-1} = \left\{ w \in \W_n : 0_{B}^{(k)}(z,M) \leq^{(k)}_A w\right\}
=
 \left\{ w \in \W_n : w \leq^{(k)}_A 1_{B}^{(k)}(z,M)\right\}.\]
\end{proposition}

\begin{proof}
Since $1_B(z,M)$ is the maximum element of $\cA(z,M)^{-1}$ under $<_A$, 
it follows from Lemma~\ref{cak-cor}
that the values $w(1)<w(2)<\dots<w(k)$ are the same for all $w \in \cAk(z,M)^{-1}$.
Given these observations, the result is straightforward to derive from Theorem~\ref{0-z-thm}. 
\end{proof}

\begin{remark}
If we let $\ll_B^{(k)}$ denote the order on $\cAk(z)^{-1}$
with $v \ll_B^{(k)} w$ if $v \cdot g^{\BC}_{n,k} \ll_B w \cdot g^{\BC}_{n,k}$,
then $(\cAk(z)^{-1},\ll_B^{(k)})$ still has 
a unique minimum  $0^{(k)}_B(z) := 0^{(k)}_B(z,M_{\min})$
for $M_{\min} = \{ \{ i,-i\} : i \in \Neg(z)\}$.
 However, $\ll_B^{(k)}$ is no longer graded; see Figure~\ref{fig3}.
 It is possible to describe the covering relations in $\ll_B^{(k)}$ directly in terms  of a variant of \eqref{black2}, but we will not pursue this here.
\end{remark}

\begin{figure}[h]
\begin{center}
\begin{tikzpicture}[scale=0.9]
\node (0) at (4,7.2) {$ 1 2 \bar 3 4 \bar 5$};  
\node (1a) at (2,6) {$ 1 2 4 \bar 5\hs \bar 3$};  
\node (1b) at (4,6) {};  
\node (1c) at (8,6) {$ 3 1 \bar 2 4 \bar 5$};  
\node (1d) at (14,6) {$ 1 2 \bar 5 3 \bar 4$};  
\node (2a) at (2,4.8) {$ 1 4 \bar 5 2 \bar 3$};  
\node (2b) at (4,4.8) {};  
\node (2c) at (6,4.8) {$ 3 \bar 2\hs \bar 1 4 \bar 5$};  
\node (2d) at (8,4.8) {$ 3 1 4 \bar 5 \hs\bar 2$};  
\node (2e) at (10,4.8) {$ 5 1 \bar 2 3 \bar 4$};  
\node (2f) at (14,4.8) {};  
\node (3a) at (0,3.6) {$ 5 1 \bar 4 2 \bar 3$};  
\node (3b) at (2,3.6) {};  
\node (3c) at (4,3.6) {};  
\node (3d) at (6,3.6) {$ 3 \bar 2 4 \bar 5\hs \bar 1$};  
\node (3e) at (8,3.6) {$ 3 4\bar 51\bar 2$};  
\node (3f) at (10,3.6) {$ 5 1 3\bar 4\hs \bar 2$};  
\node (3g) at (12,3.6) {$ 5 \bar 2\hs \bar 1 3 \bar 4$};  
\node (3h) at (14,3.6) {};  
\node (4a) at (2,2.4) {$ 5 \bar 4\hs \bar 1 2 \bar 3$};  
\node (4b) at (4,2.4) {};  
\node (4c) at (6,2.4) {$ 3 4\bar 5\hs\bar 2\hs\bar 1$};  
\node (4d) at (8,2.4) {};  
\node (4e) at (10,2.4) {$ 5 3\bar 41\bar 2$};  
\node (4f) at (12,2.4) {$ 5 \bar 2 3\bar 4\hs \bar 1$};  
\node (5a) at (4,1.2) {$ 5\bar 4 2 \bar 3\hs \bar 1$};  
\node (5b) at (6,1.2) {};  
\node (5c) at (8,1.2) {$ 5 \bar 4\hs \bar 31\bar 2$};  
\node (5d) at (10,1.2) {$ 5 3\bar 4\hs \bar 2\hs \bar 1$};
\node (6) at (7,0) {$ 5 \bar 4\hs \bar 3\hs \bar 2\hs \bar 1$};  
\draw  [->,dotted] (6) -- (5a); 
\draw  [->,thick] (6) -- (4c);
\draw  [->,dotted] (6) -- (5c); 
\draw  [->,dashed] (6) -- (5d); 
\draw  [->] (5a) -- (4a);
\draw  [->,thick] (5c) -- (3e);
\draw  [->,dashed] (5c) -- (4e);
\draw  [->,dotted] (5d) -- (4e); 
\draw  [->] (5d) -- (4f);
\draw  [->,dashed] (4a) -- (3a);
\draw  [->,thick] (4a) -- (2a);
\draw  [->] (4c) -- (3d);
\draw  [->,dotted] (4c) -- (3e); 
\draw  [->] (4e) -- (3f);
\draw  [->] (4f) -- (3g);
%
\draw  [->] (3d) -- (2c);
\draw  [->] (3e) -- (2d);
\draw  [->] (3f) -- (2e);
\draw  [->,dashed] (3g) -- (2e);
\draw  [->,thick] (3g) -- (1d);
\draw  [->] (2a) -- (1a);
\draw  [->,thick] (2c) -- (0);
\draw  [->,dashed] (2c) -- (1c);
\draw  [->] (2d) -- (1c);
%
\draw  [->] (1a) -- (0); 
\end{tikzpicture}\end{center}
\caption{The Hasse diagram of  $\ll^{(k)}_B$ on $ \cAk(z)^{-1} $ for $n=5$, $k=1$, and $z = \bar{1}\hs\bar2\hs\bar3\hs\bar4\hs\bar{5} $.
Each $w$ shown has $w\cdot g^{\BC}_{n,k} = w \cdot \bar12345 \in \cA(z)^{-1}$.
The solid, dashed, or dotted arrows correspond to $\vartriangleleft_A$, 
 $\vartriangleleft_B$, or $\scov$ as in
 Figure~\ref{fig2}.
 The longer solid arrows are relations
 not obtained from covers in $(\cA(z)^{-1},\ll_B)$.}
\label{fig3}
\end{figure}

We push these constructions one step further. Assume $n-k$ is even and let
\be\label{ybc-eq-fpf} \yfpf_{n,k}:= y^{\BC}_{n,k} \cdot t_{k+1}\cdot t_{k+3} \cdots t_{n-1}
\quand \gfpf_{n,k} := 0_B(\yfpf_{n,k})  = g^{\BC}_{n,k}\cdot t_{k+1}\cdot t_{k+3} \cdots t_{n-1}
.\ee 
For example, $\yfpf_{7,3} = \bar1\hs \bar2\hs \bar 3 5476 \in \I(\W_7)$ and $\gfpf_{7,3} = \bar3\hs\bar2\hs\bar 1 5476 \in \I(\W_7)$.
Next let 
\[ \cAkfpf(z) := \cA(\yfpf_{n,k},z)
\qquand
\shfpf_k(w) := \sh(w\cdot \gfpf_{n,k})
\]
for $z \in\I(\W_n)$ and $w \in \cAkfpf(z)^{-1}$.
The set $\cAkfpf(z)$ is nonempty only if $\Fix(z)=\varnothing$,
so we think of these sets as ``fixed-point-free'' analogues of the ones given above.

\begin{lemma}
Supose $z \in \I(\W_n) $ and $M \in \NCSM(z)$.
Then the set  \[\cAkfpf(z,M) := \left\{ w \in \cAkfpf(z) : \shfpf_k(w^{-1}) = M\right\}\]
is nonempty if and only if $\Fix(z) =\varnothing$ and $M $ has exactly $k$ trivial blocks.
\end{lemma}

\begin{proof}
We have $v\in \cAkfpf(z,M)^{-1} $ if and only if $w:= v \cdot t_{k+1} \cdot t_{k+3}\cdots t_{n-1} \in \cAk(z,M)^{-1}$ and $\ell( w) = \ell(v) + \frac{n-k}{2}$.
This occurs precisely when $v(1) < v(2) < \dots< v(k) < 0$ and $v(k+2i-1) > v(k+2i)$ for all $0\leq i < \frac{n-k}{2} $, in which case it follows from
the definition of the nested descent graph and Proposition~\ref{prop-thm}
that $\Fix(z) =\varnothing$ and $M $ has exactly $k$ trivial blocks.
Conversely, if these conditions are satisfied, then it is easy to check that $1_B(z,M) \cdot \gfpf_{n,k} \in \cAkfpf(z,M)^{-1}$.
\end{proof}


\begin{corollary}
If $z \in \I(\W_n)$ then 
$\cAkfpf(z) = \bigsqcup_{M \in \NCSM_k(z) \setminus \NCSM_{k+1}(z)} \cAkfpf(z,M).$
\end{corollary}

\begin{remark}
We do not require $\Fix(z) = \varnothing$  since if this fails then 
our identity still holds, with both sides empty. 
Note that the set difference $\NCSM_k(z) \setminus \NCSM_{k+1}(z)$ 
consists of all noncrossing symmetric perfect matchings in $\{ i\in[\pm n] : z(i) =-i\}$ with exactly $k$ trivial blocks $\{ \pm i\}$.
\end{remark}

We can again efficiently generate $\cAkfpf(z,M)$.
Choose $z \in \I(\W_n)$ with no fixed points in $[\pm n]$.
Suppose $M \in \NCSM_k(z)$ with exactly $k$ trivial blocks, say $\{\pm i_1\}, \{\pm i_2\}, \dots, \{\pm i_k\}$ with $0<i_1<i_2<\dots<i_k$.
Let $X$ be the set of integer pairs in $[\pm n]\times [n]$ of the form $(-b,a)$ where $0<a<b$ and $\{a,b\} \in M$, or 
of the form $(a,b)$ where $ |a| < z(a) = b$ so that $(a,b) \in \Pair(z)$. Write
 \be\label{x-eq}X = \{(a_1,b_1),(a_2,b_2),\dots,(a_l,b_l)\}=\{(c_1,d_1),(c_2,d_2),\dots,(c_l,d_l)\}\ee
where 
$a_1<a_2<\dots<a_l $ and $ d_1<d_2<\dots<d_l$. 
Now define $0_\fpf(z,M),1_\fpf(z,M)\in \W_n$ by
\[ 0_\fpf(z,M) := i_1i_2\cdots i_k a_1b_1a_2b_2\cdots a_lb_l
\quand
 1_\fpf(z,M) := i_1i_2\cdots i_k c_1d_1c_2d_2\cdots c_ld_l.
 \]
 For example, if $z = \bar1\hs\bar2\hs\bar5\hs\bar4\hs\bar3$ and $M = \{\{-1,-2\}, \{1,2\}, \{-4,4\} \}$,
then 
\[0_\fpf(z,M) = 4\bar 3 5 \bar 2 1 
\quand
1_\fpf(z,M) =  4 \bar 2 1 \bar 3 5.
\]
 Finally, let $\vartriangleleft^{(k)}_\fpf$ denote the relation on $n$-letter words with 
$v \vartriangleleft^{(k)}_\fpf w$ if for 
some numbers $a<b<c<d$
and some index $i\geq 0$ with $k+2i+4\leq n$,
it holds that
\[
v_{k+2i+1}v_{k+2i+2}v_{k+2i+3}v_{k+2i+4} = adbc
\quand
w_{k+2i+1}w_{k+2i+2}w_{k+2i+3}w_{k+2i+4} = bcad\]
while
$v_j=w_j$ for $j\notin k+2i+\{1,2,3,4\}.$ 
Write $<_\fpf^{(k)}$ for the transitive closure of $\vartriangleleft^{(k)}_\fpf$.

\begin{proposition}\label{fpf-prop}
Suppose $z \in \I(\W_n)$ has $\Fix(z)=\varnothing$
and $M \in\NCSM_k(z)\setminus \NCSM_{k+1}(z)$. Then 
\[ \cAkfpf(z,M)^{-1} = \left\{ w \in \W_n : 0_{\fpf}(z,M) \leq^{(k)}_\fpf w\right\}
=
 \left\{ w \in \W_n : w \leq^{(k)}_\fpf 1_{\fpf}(z,M)\right\}.\]
\end{proposition}

\begin{proof}
We have $0_{\fpf}(z,M) t_{k+1} t_{k+3}\cdots t_{n-1} = 0_{B}^{(k)}(z,M)$
and
 the product on the left is length additive, so 
 $0_{\fpf}(z,M)\in \cAkfpf(z,M)^{-1}$.
 It follows likewise that $1_{\fpf}(z,M)\in \cAkfpf(z,M)^{-1}$.
 Adopt the notation in \eqref{x-eq}.
 Each element $w \in \cAkfpf(z,M)^{-1}$ must be given in one-line notation as 
 $w = i_1i_2\cdots i_k a_{j_1} b_{j_1}\cdots a_{j_l} b_{j_l}$
 by the definition of the shape map and the nested descent graph.
 Moreover, no subword $w_{k+2i+1}w_{k+2i+2}w_{k+2i+3}w_{k+2i+4}$ can 
 have the form $bdac$ for $a<b<c<d$, since then $w\cdot \gfpf_{k,n} \in \cA(z)^{-1}$
 would have a consecutive subword of the form $dbca$ and therefore would be related by $\vartriangleleft_A$ to a signed permutation containing a consecutive 321-pattern,
 contradicting Proposition~\ref{321-prop}.
 
From these observations, one argues that a sequence of moves going down 
(respectively, up) using $\vartriangleleft^{(k)}_\fpf$ transforms any  $w\in\cAkfpf(z,M)^{-1}$
to $0_{\fpf}(z,M)$ (respectively, $1_{\fpf}(z,M)$). This is a straightforward exercise,
which is similar to the arguments in \cite[\S6.2]{HMP2}; we omit the details.
\end{proof}
 
For  $z \in\I(W_n)$, define $\cAfpf(z) := \cAkfpf(z)$ where $k=0$ if $n$ is even and $k=1$ if $n$ is odd.
 These are signed analogues of the sets  of ``fixed-point-free atoms'' 
 $\cAfpf(y) := \cA(s_1s_3s_5\cdots s_{n-1}, y)$ for $y \in \I(S_{2n})$, which are
studied in \cite{CJW,HMP2}. 
Define $\cRfpf(z) := \bigsqcup_{w \in \cAfpf(z)} \cR(w)$
  and let \[z_0^{\BC_n}:=\bar 2\hs\bar 1 \hs\bar4\hs\bar3 \cdots \bar{n}\hs\bar{n-1}\text{ if $n$ is even}
  \quord
z_0^{\BC_n}:=\bar 2\hs\bar 1 \hs\bar4\hs\bar3 \cdots \bar{n-1}\hs\bar{n-2}\hs\bar{n}\text{ if $n$ is odd.}\]
Also write $w_0^{S_n} = n \cdots 321 \in \I(S_n)$.
An interesting consequence of Proposition~\ref{fpf-prop} is that there are simple bijections $\cAfpf(z_0^{\BC_{2n}}) \leftrightarrow \cAfpf(z_0^{\BC_{2n+1}}) \leftrightarrow \cAfpf( w_0^{S_{2n}})$.
Computations suggest the following related identity,
which might be proved using the methods in \cite{MP}:

\begin{conjecture}
For all $n\in \NN$ it holds that $|\cRfpf(z_0^{\BC_{n}})| = |\cR(w_0^{S_n})|$.
\end{conjecture}

E.g., 
$|\cRfpf(\bar 2\hs \bar 1 \hs \bar 4\hs \bar 3\hs\bar5)| = |\cR(54321)|=768$ and
$|\cRfpf(\bar 2\hs \bar 1 \hs \bar 4\hs \bar 3\hs\bar6\hs\bar 5)| = |\cR(654321)|=292864$.
 These numbers count the standard Young tableaux of shape $\delta_n := (n-1,\dots,3,2,1)$.
A similar identity seems to hold for the longest element
in $\W_n$:

\begin{conjecture}
If $k=\lfloor \frac{n}{2}\rfloor+1$ then $|\cRfpf\Bigl(\bar1\hs\bar2\hs\bar3\cdots\bar{n-1}\Bigr)| = |\cR\Bigl(w_0^{S_{n}}\cdot (1,2,3,\dots,k)\Bigr)|$.
\end{conjecture}
 
E.g., $|\cRfpf(\bar 1\hs \bar 2 \hs \bar 3\hs \bar 4)| = |\cR(35421)|=70$ and
 $|\cRfpf(\bar 1\hs \bar 2 \hs \bar 3\hs \bar 4\hs\bar5)| = |\cR(365421)|=5775$.
 These numbers count the standard Young tableaux whose shape 
 is formed by removing $\lfloor\frac{n}{2}\rfloor$ from $\delta_n$.
\begin{remark}
Consider the partial order on $\cAkfpf(z)^{-1}$
with $v \lll_\fpf^{(k)} w$ if $v \cdot \gfpf_{n,k} \lll_B w \cdot \gfpf_{n,k}$.
If $k\in \{0,1\}$ and $n-k$ is even,
then $(\cAkfpf(z)^{-1},\lll_B^{(k)})$  has 
a unique maximum given by the element $1^{(k)}_\fpf(z) := 1_B(z) \cdot \gfpf_{n,k}$; see Figure~\ref{fig3fpf}.
This order is less well-behaved for $k>1$.
\end{remark}

\begin{figure}[h]
\begin{center}
\begin{tikzpicture}[xscale=0.9, yscale=1.1]
\node (0) at (8,7.2) {$ 1  \bar 32  \bar 54$};  
\node (1c) at (8,6) {$ 3  \bar 21  \bar 54$};  
\node (1d) at (10,6) {$ 1  \bar 52  \bar 43$};  
\node (2a) at (6,6) {$ 1  \bar 54  \bar 32$};  
\node (2e) at (10,4.8) {$ 5 \bar 21  \bar 43$};  
\node (3a) at (6,4.8) {$ 5  \bar 41  \bar 32$};  
\node (3e) at (8,4.8) {$ 3 \bar 54\bar 21$};  
\node (4e) at (8,3.6) {$ 5 \bar 43\bar 21$};    
\draw  [->,dotted] (1c) -- (0);
\draw  [->,dotted] (2e) -- (1d);
\draw  [->,dotted] (1d) -- (0);
\draw  [->,dotted] (3a) -- (2a);
\draw  [->,dotted] (4e) -- (3e);
\draw  [->] (4e) -- (2e);
\draw  [->] (3e) -- (1c);
\draw  [->] (2a) -- (0);
\end{tikzpicture}\end{center}
\caption{The Hasse diagram of $\lll^{(k)}_\fpf$ on $\cAkfpf(z)^{-1}$  for $n=5$, $k=1$, and $z = \bar{1}\hs\bar2\hs\bar3\hs\bar4\hs\bar{5} $.
Each $w$ shown has $w\cdot \gfpf_{n,k} = w \cdot \bar13254 \in \cA(z)^{-1}$.
The solid / dotted arrows correspond to $\vartriangleleft_\fpf^{(k)}$ / $\bcov$.}
\label{fig3fpf}
\end{figure}

\section{Cohomology formulas for orbit closures}
\label{geom-sect}

In this section we explain how the sets $\cA(z)$ are related to cohomology classes of
orbit closures in flag varieties of types B and C.
Our goal is to prove  a precise form of Theorem~\ref{intro-Brion}.

\subsection{Brion's cohomology formula}\label{coho-sect}

Let $G$ be a connected, complex, simple algebraic group with a holomorphic automorphism $\theta$ that is an involution. Fix a Borel subgroup $B$ containing a torus $T$. Assume both groups are preserved by $\theta$. The subgroup of fixed points $K = G^\theta := \{ g \in G : g=\theta(g)\}$ is called a \emph{symmetric subgroup} \cite{RichSpring}. Every symmetric subgroup acts on the flag variety $G/B$ with finitely many orbits. 

There is a useful \emph{weak order} on the finite set $K\backslash G/B$, introduced in \cite{RichSpring,RichSpring2}.
Fix a simple generator $s \in S$ for the Weyl group $W=N_G(T)/T$ and write $\alpha$ for the corresponding simple root. Let $P_\alpha$ be the minimal parabolic subgroup of $G$ of type $\alpha$ containing $B$ and let $\pi_\alpha : G/B \to G/P_\alpha$ be the canonical projection. 
If $\cO$ is a $K$-orbit in $G/B$, then $\pi_\alpha^{-1}(\pi_\alpha(\cO))$ contains a unique dense $K$-orbit, denoted $s\cdot \cO$. 
The \emph{weak order} on the finite set of $K$-orbits in $G/B$ is the transitive closure of the relation that has $\cO < \cO'$ whenever $\cO \neq \cO' = s\cdot \cO$ for some $s\in S$
(see \cite{RichSpring}  or \cite[\S1.4]{Wyser}).

An algebraic subgroup $H \subseteq G$ is \emph{spherical} if $H$ acts on $G/B$ with finitely many orbits. This includes $H=K$ as a special case. If $H=B$ then the orbits are in
bijection with $W$ and
  the closures $ \overline{B wB} / B$ of these orbits are  the well-known \emph{Schubert varieties}.
 In \cite{Brion2001}, Brion derives a general formula for the cohomology classes of $H$-orbit closures in $G/B$
 as linear combinations of classes for
 Schubert varieties. When $H=K$ is a symmetric subgroup, this goes as follows.

Form a directed graph on the set of $K$-orbits in $G/B$ that has 
an edge $\cO \xrightarrow{s} \cO'$
for each covering relation $\cO < \cO ' =s\cdot \cO$ in the weak order.
The arrow  $\cO \xrightarrow{s} \cO'$ is marked as a ``doubled edge'' if $s \in S$ corresponds to a simple root $\alpha$ for which the projection $\pi_\alpha : G/B \to G/P_\alpha$ restricted to the closure of $\cO$ has degree two. 
In the terminology of Wyser's thesis, this occurs if and only if the simple root $\alpha$ associated to $s$
is \emph{non-compact imaginary of type II} \cite[Proposition 1.4.5]{Wyser}.

There is always a dense $K$-orbit $\cO_{\mathrm{dense}}$, which is the unique maximum in the weak order.
Given a directed path $P = ( \cO \xrightarrow{s_1} \cO' \xrightarrow{s_2}\cO'' \cdots \xrightarrow{s_l} \cO_{\mathrm{dense}})$,
define $w(P) := s_l \cdots s_2s_1 \in W$ and let $d(P)$ be the number of doubled edges in $P$.
For a $K$-orbit $\cO$, let $\cAB(\cO)$ be the set of elements $w(P)$ as $P$ ranges 
over all directed paths $P$ from $\cO$ to the dense orbit. 
Let $d_\cO(w)$ for $w \in \cAB(\cO)$ be the value of $d(P)$ for any path from $\cO$ to the dense orbit with $w=w(P)$;
this is the same for all choices of $P$ \cite{Brion2001}.
Brion's formula for the cohomology class of the closure $\overline{\cO}$ of $\cO$ is now
\be\label{brion} [\overline{\cO}] = \sum_{w\in \cAB(\cO)} 2^{d_\cO(w)} \left[ \overline{B w_0 wB} / B\right] \in H^*(G/B).\ee
We refer to the elements of $\cAB(\cO)$ as \emph{Brion atoms}.

There is a complete classification of the pairs $(G,K)$. Table~\ref{tbl0}
shows the possibilities with $G$ classical.
For these types, there are useful combinatorial descriptions of indexing sets for the finite set of $K$-orbits in $G/B$
and the associated weak order \cite{AY,MO,Wyser}.
The problem of giving a concrete description of the terms in Brion's formula
has been extensively studied when $G=\GL(n,\CC)$;
see
\cite{CJ,CJW,HMP2,WY} for types (AI) and (AII) and \cite{BurksPawlowski,CJW,CU2,WY0} for type (AIII).
 Here, we focus on the next three families $(G,K)$ listed in Table~\ref{tbl0}, 
 where the Weyl group is $W=\W_n$.

\def\hhline{\\ & & & & \\ [-10pt]\hline & & & & \\ [-10pt]}
\begin{table}[h]
\begin{center}
{\small
\begin{tabular}{| l | l | l | l | l | l |  l l |}
\hline
Type & $G$ & $K=G^\theta$ (assume $p,q\in \NN$ and $n=p+q$) & $W$ &   Reference
\hhline
AI & $\GL(n,\CC)$ & $\O(n,\CC)$ &  $S_n$ &  \cite[\S2.2-2.3]{Wyser}  \\
AII  & &  $\Sp(n,\CC)$ (only if $n$ is even) & &   \cite[\S2.4]{Wyser} \\
AIII  & & $\GL(p,\CC)\times \GL(q,\CC)$ & &   \cite[\S2.1]{Wyser} 
\hhline
BI & $\SO(2n+1,\CC)$ & $S(\O(2p,\CC)\times \O(2q+1,\CC))$ & $\W_n$ &   \cite[\S3.1]{Wyser} 
\hhline
CI & $\Sp(2n,\CC)$ & $\GL(n,\CC)$ & $\W_n$ &   \cite[\S4.2]{Wyser}   \\ 
CII  & & $\Sp(2p,\CC)\times \Sp(2q,\CC)$ & &   \cite[\S4.1]{Wyser} 
\hhline
DI  & $\SO(2n,\CC)$ & $ S(\O(2p,\CC) \times \O(2q,\CC))$ & $\DGroup_n$ & \cite[\S5.1]{Wyser}   \\
DII &  & $S(\O(2p+1,\CC) \times \O(2q-1,\CC))$ & &  \cite[\S5.3]{Wyser}   \\
DIII & & $\GL(n,\CC)$ & &  \cite[\S5.2]{Wyser}  \\
\hline
\end{tabular}}
\end{center}
\caption{
Symmetric subgroups $K=G^\theta$ in rank $n$ classical groups $G$,
as parametrized in \cite{Wyser}.
}\label{tbl0}
\end{table}

%
%
%

\subsection{Clans}\label{clans-sect}

Choose integers $p,q \in \NN$ with $n=p+q$.
A \emph{$(p, q)$-clan} is a sequence of $n$ symbols $(c_1,c_2,\dots,c_n)$, 
in which each $c_i$ is either an integer or one of the signs $+$ or $-$. Each integer which appears must appear exactly twice, and the difference between the number of $+$'s and the number of $-$'s appearing in the string must be precisely $p - q$. Two such sequences are considered to represent the same clan if the sets of pairs $\{i,j\}$ giving the positions of matching integers $c_i=c_j$ are the same in each. For example, $(1, -, 2, 1, +, 2)$ and  $(7, -, 0, 7, +, 0)$ are the same $(3,3)$-clan.
For any clan $\gamma = (c_1,c_2,\dots,c_m)$, let $\signs(\gamma)$ be the subsequence of signs $c_i \in \{\pm\}$.

A clan $(c_1,c_2,\dots,c_n)$ is \emph{symmetric} if it is equal to $(c_n,\dots,c_2,c_1)$,
and \emph{skew-symmetric} if it is equal to $(-c_n,\dots,-c_2,-c_1)$.
For example, $(1,+,1,2,+,2)$ is a symmetric $(4,2)$-clan and $(1,-,+,1,3,-,+,3)$ is a skew-symmetric $(4,4)$-clan. 
 Bingham, Can, and U\u{g}urlu derive generating functions 
 to count such clans in \cite{BinUgu,CU1}.
%
Given a symmetric or skew-symmetric clan $\gamma = (c_{-n},\dots,c_{-1},c_0,c_1,\dots,c_n)$ or $\gamma = (c_{-n},\dots,c_{-1,},c_1,\dots,c_n)$, define $\BCDphi(\gamma) \in \I(\W_n)$ to be the unique involution that has $(i,j)$ as a 2-cycle for distinct $i,j \in [\pm n]$ if and only if $c_i=c_j \in \ZZ$.

\subsection{Brion atoms in types CI and CII}

The symplectic group $G = \Sp(2n,\CC)$ has a symmetric subgroup $K \cong \GL(n,\CC)$
whose orbits $\cO_\gamma$ in $G/B$ are indexed by skew-symmetric $(n,n)$-clans $\gamma$
\cite[\S4.2]{Wyser};  see also \cite{Wyser2012}.
The weak order on these $K$-orbits corresponds to the following weak order on clans.
Fix a skew-symmetric clan
 \be\label{ssg-eq}\gamma = (c_{-n},\dots,c_{-2},c_{-1},c_1,c_2,\dots,c_n).\ee
 For each $s \in \{t_0,t_1,\dots,t_{n-1}\}$ we wish to define another clan $s\cdot \gamma$.
 The weak order on skew-symmetric clans 
 will be the transitive closure of the relation with $\gamma < \gamma'$ whenever $\gamma \neq \gamma' = s\cdot \gamma$ for some $s$.
 Our reference for the following definitions is \cite[\S4.2.2]{Wyser}.

If $s=t_i$ for $i \in [n-1]$ and $c_i$ and $c_{i+1}$ are not both signs, 
or if $s=t_0$ and $c_{1}$ is not a sign,
then $s \cdot \gamma$ is the unique skew-symmetric clan with 
$\BCDphi(s\cdot \gamma) = s\circ \BCDphi(\gamma) \circ s$
and $\signs(\gamma) = \signs(s\cdot\gamma)$.
Assume $s=t_i$ for $i \in [n-1]$ and $c_i$ and $c_{i+1}$ are both signs.
If $c_i=c_{i+1}$ then $s\cdot \gamma =\gamma$. Otherwise $s\cdot \gamma$ is formed from $\gamma$ by replacing $c_i$, $c_{i+1}$ by a pair of matching integers and $c_{-i-1},c_{-i}$ by another pair of matching integers (not appearing elsewhere in $\gamma$).
If $s=t_0$ and $c_1$ is a sign, then necessarily $c_{-1}\neq c_1$,
and we form $s\cdot \gamma$ from $\gamma$ by replacing $c_{-1}$, $c_1$ with a new pair of matching integers.
The edge $\cO_\gamma \xrightarrow{s} \cO_{s\cdot \gamma}$
is doubled in the directed graph for Brion's formula 
if and only if $s=t_i$ for $i \in [n-1]$ such that $c_{-i} = c_{i+1} \in \ZZ$ and $c_{-i-1}=c_i \in \ZZ$.

\begin{example}
One has $t_0 \cdot (+,-,+,-) = (+,1,1-)$ and $t_1\cdot (+,-,+,-) = (1,1,2,2)$ while $t_0 \cdot (1,1,2,2) = (1,2,1,2)$ and $t_1\cdot (+,1,1,-) = (1,+,-,1)$. Compare with \cite[Figure B.10]{Wyser}, where labeled edges $\xrightarrow{\ i\ }$ correspond in our notation to the $\cdot$ action of $t_{n-i}$.
\end{example}

For the pair $(G,K) = (\Sp(2n,\CC),\GL(n,\CC))$, we can describe all of the terms in \eqref{brion} explicitly.
Continue to fix a skew-symmetric clan $\gamma$ written as in \eqref{ssg-eq}.
Define
$ \NCSM^{\mathrm{CI}}(\gamma)$ to be the set of noncrossing symmetric perfect matchings $M$ in $\{ i \in [\pm n] : c_i = \pm\}$ such that
if $\{i,j\} \in M$ then $c_i $ and $ c_j$ are opposite signs in $\gamma$.
Recall that $\ell_0(w) = \{ i \in [n] : w(i) < 0\}$ for $w \in \W_n$.
We say that $\gamma$ is \emph{alternating} if the sequence $\signs(\gamma)$ has no equal adjacent entries. 

\begin{theorem}\label{bt1} Assume $(G,K) = (\Sp(2n,\CC), \GL(n,\CC))$.
Suppose $\cO=\cO_\gamma$ is the $K$-orbit in $G/B$
indexed by the skew-symmetric clan $\gamma$.
Let $z = -\BCDphi(\gamma) \in \W_n$
and $w \in \cAB(\cO)$.
Then \[
  d_{\cO}(w) =  
|  \{ i \in [ n] : z(i)<i\}|-\ell_0(w)
\quand 
\cAB(\cO) = \bigsqcup_{M \in  \NCSM^{\mathrm{CI}}(\gamma)} \cA(z,M).\]
Thus, $\cAB(\cO) \subseteq \cA(z)$ with equality if and only if $\gamma$ is alternating. \end{theorem}

\begin{proof}
The unique dense $K$-orbit $\cO_{\mathrm{dense}}$ is indexed by the clan
$(1,2,\dots,n,n,\dots,2,1)$
whose image under $\BCDphi$ is $w_0=\wB$.
A directed path 
$P=\Bigl(
\cO \xrightarrow{t_{a_1}} \cO' \xrightarrow{t_{a_2}} \cdots \xrightarrow{t_{a_l}} \cO_{\mathrm{dense}}\Bigr)$
 corresponds to a reduced word $\a = (a_1,a_2,\dots,a_l)$ for some 
$w \in \cA(-z,w_0) = \cA(z)^{-1}$
with $w^{-1} \in \cAB(\cO)$,
so we have $\cAB(\cO) \subseteq \cA(z)$.
From the discussion in Section~\ref{relshape-sect} and our description of the weak order,
it is clear that the reduced words $\a$ that arise from paths in this way are precisely 
those with $\sh(\a;-z) \in  \NCSM^{\mathrm{CI}}(\gamma)$.
Since $\sh(\a;-z) = \sh(w)$ by Theorem~\ref{rel-prop1},
we conclude that $w \in \cAB(\cO)$
if and only if $w \in \cA(z,M)$ for some $M \in \NCSM^{\mathrm{CI}}(\gamma)$.

In view of \eqref{prime-eq}, the number of doubled edges in the
path $P$
is
$\ell'(w_0) - \ell'(-z) -p-q$, where $p$ is half the number of nontrivial blocks in $\sh(\a;-z)=\sh(w)$ and 
$q$ is the number of trivial blocks.
It follows from \eqref{ell'} that $\ell'(w_0) - \ell'(-z)  = \ell'(z)$ and 
we have \[p+q = \neg(z) - p = \neg(z) - |\{ a \in [n] : a \leq -z(a)\}|  + \ell_0(w)
=  \ell_0(w) - |\{ a \in [n] : z(a) < -a\}|  
\] by Proposition~\ref{pair-prop}.
Thus 
$d_\cO(w^{-1}) = \ell'(z)+  |\{ a \in [n] : z(a) < -a\}|  - \ell_0(w)$.
This gives the desired formula as $\ell'(z)  = \pair(z) + \neg(z)  = |\{ b\in [n]: -b \leq z(b) < b\}|$
by \eqref{ell'} 
and $\ell_0(w^{-1}) = \ell_0(w)$.
\end{proof}

\begin{example}
Suppose $n=4$ and $\gamma = (-, -, -, +, -, +, +, +)$, which is a skew-symmetric $(4,4)$-clan.
If $z = -\BCDphi(\gamma) = \bar1\hs\bar2\hs\bar3\hs\bar4 \in \W_4$ then $\NCSM^{\mathrm{CI}}(\gamma)$ has 2 elements
\[ 
\ba   \{\{-4, 4\}, \{-3, 3\}, \{-2, -1\}, \{1, 2\}\} &= 
\arcstart{
*{\bullet}  \arc{1.2}{rrrrrrr}    & *{\bullet}   \arc{0.9}{rrrrr}    & *{\bullet}   \arc{0.4}{r}  & *{\bullet}       & *{\bullet}  \arc{0.4}{r} & *{\bullet}    & *{\bullet}   & *{\bullet}    
} 
\arcstop,
\\[-10pt]
\\
\{ \{-4, 4\}, \{-3, 3\}, \{-2, 2\}, \{-1, 1\} \} &= \arcstart
{
*{\bullet}  \arc{1.2}{rrrrrrr}    & *{\bullet}   \arc{0.9}{rrrrr}    & *{\bullet} \arc{0.6}{rrr}     & *{\bullet}   \arc{0.3}{r}     & *{\bullet} & *{\bullet}    & *{\bullet}   & *{\bullet}    
} 
\arcstop,
\ea
\]
compared to the 6 elements in $\NCSM(z)$.
Consulting Example~\ref{intro-ex}, we see that in this case $\cAB(\cO_\gamma) = \{\bar 4\hs\bar31\bar2, \bar 4\hs\bar 3\hs\bar2\hs\bar1\}^{-1}=\{3\bar4\hs\bar2\hs\bar1, \bar 4\hs\bar 3\hs\bar2\hs\bar1\}$. Here $d_\cO(3\bar4\hs\bar2\hs\bar1) = 1$, $d_\cO(\bar 4\hs\bar 3\hs\bar2\hs\bar1) = 0$.
\end{example}


Fix $p,q \in \NN$ with $n=p+q$. The group $G= \Sp(2n,\CC)$ has another symmetric subgroup of the form $K \cong \Sp(2p,\CC)\times \Sp(2q,\CC)$.
In this case the $K$-orbits $\cO_\gamma$ in $G/B$ are indexed by the symmetric $(2p,2q)$-clans $\gamma$ for which  $\Neg(\gamma) := \Neg(\BCDphi(\gamma))$ is empty
\cite[\S4.1]{Wyser}.
If we write
 \be\label{ssg-sym-eq}\gamma = (c_{-n},\dots,c_{-2},c_{-1},c_1,c_2,\dots,c_n),\ee
then
 $\Neg(\gamma)$ is the set of indices $i \in [n]$ with $c_{-i}=c_i \in \ZZ$

 Assume $\gamma$ is a symmetric clan with $\Neg(\gamma)=\varnothing$,
 written as in \eqref{ssg-sym-eq}.
 For $s \in \{t_0,t_1,\dots,t_{n-1}\}$ we define another such clan $s\cdot \gamma$ 
 by the same rules as
 above, except that if this results in  a clan with $\Neg(s\cdot \gamma)\neq\varnothing$ then we redefine
 $s\cdot \gamma = \gamma$. In particular, if $s=t_0$ and $c_{-1}=c_{1} = \pm$
 then $s\cdot \gamma = \gamma$, and if $s=t_i$ for $i>0$
 where $c_{-i-1} = c_i \in\ZZ$ and $c_{-i}=c_{i+1} \in \ZZ$ then $s\cdot \gamma =\gamma$.
 For more details concerning this action, see \cite[\S4.1.2]{Wyser} and \cite[Figure B.9]{Wyser}.
 
The weak order on $K= \Sp(2p,\CC)\times \Sp(2q,\CC)$-orbits in $G=\Sp(2n,\CC)$
is the transitive closure of the relation with $\cO_{\gamma} < \cO_{\gamma'}$ whenever $\gamma \neq \gamma' = s\cdot \gamma$ for some $s$.
In this setting,
there are no doubled edges in the directed graph for Brion's formula 
\cite[\S4.1.2]{Wyser}.

Let $X (\gamma) = \{ i \in [\pm n] : c_i = \pm\}$ and fix an integer $k\geq 0$.
Define
$ \NCSM^{\mathrm{CII}}_k(\gamma)$ to be the set of noncrossing symmetric perfect matchings $M$ in $X(\gamma)$ that have exactly $k$ trivial blocks $\{\pm i\}$, such that
if $\{i,j\} \in M$ and $i+j\neq 0$ then $c_i $ and $ c_j$ are opposite signs in $\gamma$.
The symmetric sequence $\signs(\gamma)$ will always have repeated elements in its middle two entries. We say that $\gamma$ is \emph{semi-alternating} if dividing $\signs(\gamma)$ in half gives two sequences with no adjacent repeated entries.

\begin{theorem}\label{bt11} Assume $(G,K) = (\Sp(2n,\CC), \Sp(2p,\CC)\times \Sp(2q,\CC))$
where $p,q \in \NN$ and $p+q=n$. 
Suppose $\cO=\cO_\gamma$ is the $K$-orbit in $G/B$
indexed by the symmetric $(2p,2q)$-clan $\gamma$.
Then \[
  d_{\cO}(w) = 0
\quand 
\cAB(\cO) = \bigsqcup_{M \in  \NCSM^{\mathrm{CII}}_k(\gamma)} \cAkfpf(z,M)\]
for all $w \in \cAB(\cO)$, where $k := |p-q|$ and $z := -\BCDphi(\gamma) \in \I(\W_n)$.
Thus, $\cAB(\cO) \subseteq \cAkfpf(z)$ 
and we have
$\cAB(\cO) = \cA_\fpf(z)$ if $|p-q| \leq 1$ and $\gamma$ is semi-alternating.
\end{theorem}

\begin{proof}
If $m = n-k$ then the unique dense $K$-orbit $\cO_{\mathrm{dense}}$ is indexed either by
the clan
\[
(2,1,4,3,6,5,\dots,m,m-1,+,+,\dots,+,+,m,m-1,\dots,6,5,4,3,2,1)
\]
or by this clan with every $+$ changed to $-$.
The image of this clan under $\BCDphi$ is $- \yfpf_{n,k}$ from \eqref{ybc-eq-fpf},
so a path $\cO \xrightarrow{t_{a_1}} \cO' \xrightarrow{t_{a_2}} \cdots \xrightarrow{t_{a_l}} \cO_{\mathrm{dense}}$
 corresponds to a reduced word $\a = (a_1,a_2,\dots,a_l)$ for some 
$w \in \cA(-z,-\yfpf_{n,k}) = \cAkfpf(z)^{-1}$
with $w^{-1} \in \cAB(\cO)$.
From the discussion in Section~\ref{relshape-sect} and the definition of the weak order, 
it follows that the reduced words $\a$ arising from such paths are precisely 
those with $\sh(\a;-z) \in \NCSM^{\mathrm{CII}}_k(\gamma)$.
Since $\sh(\a;-z) = \shfpf_k(w)$ by Theorem~\ref{rel-prop1},
we conclude that $w \in \cAB(\cO)$
if and only if $w \in \cAkfpf(z,M)$ for some $M \in \NCSM^{\mathrm{CII}}_k(\gamma)$.
\end{proof}

\begin{example}
Suppose $\gamma = (+, -, +, +, - ,+)$, which is a symmetric $(2p,2q)$-clan for $p=2$ and $q=1$. Let $k = |p-q|=1$
and $z= -\BCDphi(\gamma) = \bar1\hs\bar2\hs\bar3 \in \W_3$. Then $\NCSM_k^{\mathrm{CII}}(\gamma)$ has 2 elements
\[ 
\ba   \{\{1, 2\}, \{-1, -2\}, \{-3, -3\}\} &= 
\arcstart{
   *{\bullet}   \arc{0.9}{rrrrr}    & *{\bullet}   \arc{0.4}{r}  & *{\bullet}       & *{\bullet}  \arc{0.4}{r} & *{\bullet}    & *{\bullet}    
} 
\arcstop
\quand
\{ \{-1, 1\}, \{2, 3\}, \{-2, -3\} \} &= \arcstart
{
   *{\bullet}   \arc{0.4}{r}     & *{\bullet}   & *{\bullet}     \arc{0.4}{r}    & *{\bullet} & *{\bullet}   \arc{0.4}{r}   & *{\bullet}    
} 
\arcstop
\ea
\]
and we have $\cAB(\cO_\gamma) = \{3\bar 21,1\bar 3 2\}^{-1}=\{3\bar 21,13\bar2\}$.

\end{example}

\subsection{Brion atoms in type BI}
\label{sect-Brion-B}

There is one more case where Brion atoms coincide with 
atoms for elements of $\W_n$.
Suppose $G = \SO(2n+1,\CC)$ is an odd special orthogonal group. 
For each $p,q \in \NN$ with $n=p+q$, the group $G$ has a symmetric subgroup $K \cong S(\O(2p,\CC)\times \O(2q+1,\CC))$
whose orbits $\cO_\gamma$ in $G/B$ are uniquely indexed by symmetric $(2p,2q+1)$-clans $\gamma$
\cite[\S3.1]{Wyser}.
The weak order on these $K$-orbits is again defined in terms of a certain action on clans.
Fix a symmetric clan of odd length written as
 \be\label{syg-eq}\gamma = (c_{-n},\dots,c_{-2},c_{-1},c_0,c_1,c_2,\dots,c_n).\ee
Symmetry forces  $c_0$ to be $+$ or $-$.
We define $s\cdot \gamma$ for $s \in \{t_0,t_1,\dots,t_{n-1}\}$ in the same way as for skew-symmetric clans in the previous subsection, except when $s=t_0$ and $c_1 \in \{ \pm \}$.
 In this case, we let $t_0 \cdot \gamma =\gamma$ if $c_{-1}=c_0=c_1$ and otherwise
 form $t_0\cdot \gamma$ from $ \gamma$  by replacing $c_{-1},c_1$ with a pair of matching integers and reversing the sign of $c_0$.
 
 The weak order on $K$-orbits in $G/B$ 
is the transitive closure of the relation with $\cO_\gamma < \cO_{\gamma'}$ whenever $\gamma \neq \gamma' = s\cdot \gamma$ for some $s$.
The edge $\cO_\gamma \xrightarrow{s} \cO_{s\cdot \gamma}$
is doubled in the corresponding directed graph for Brion's formula 
if and only if $s=t_0$ and $c_0\neq c_1$ are opposite signs, or if $s=t_i$ for $i \in [n-1]$ where $c_{-i} = c_{i+1} \in \ZZ$ and $c_{-i-1}=c_i \in \ZZ$ \cite[\S3.1.2]{Wyser}.

\begin{example}
One has  $t_0 \cdot (-,+,+,-,+,+,-) = (-,+,1,+,1,+,-)$ and $t_2 \cdot (-,+,+,-,+,+,-) = (1,1,+,-,+,2,2) $
while
$t_1 \cdot  (1,1,+,-,+,2,2) =  (1,+,1, -,2,+,2)$.
Compare with \cite[Figure B.8]{Wyser}, where labeled edges $\xrightarrow{\ i\ }$ correspond in our notation to the $\cdot$ action of $t_{n-i}$.
\end{example}

Continue to fix a symmetric clan $\gamma$ written as in \eqref{syg-eq}.
For each integer $k \geq 0$,
define
$ \NCSM_k^{\mathrm{BI}}(\gamma)$ to be the set of noncrossing symmetric perfect matchings $M $ in the set 
$\{ i \in [\pm n]: c_i = \pm\}$ such that
(1) if $\{i,j\} \in M$ with $i+j\neq 0$ then $c_i \neq c_j$,
and (2) if the trivial blocks of $M$ are $\{\pm i_1\}, \{\pm i_2\},\dots,\{\pm i_l\}$ where $0<i_1<i_2<\dots<i_l$, then $k\leq l$ and  $c_0\neq c_{i_1}\neq c_{i_2}\neq \dots \neq c_{i_{l-k}}$
(note that this condition ignores the last $k$ trivial blocks).

\begin{theorem}\label{bt2} Assume $(G,K) = (\SO(2n+1,\CC), S(\O(2p,\CC)\times \O(2q+1,\CC))$ where  $p+q=n$.  
Suppose $\cO=\cO_\gamma$ is the $K$-orbit in $G/B$
indexed by the symmetric $(2p,2q+1)$-clan $\gamma$.
Then \[  d_\cO(w) =  |\{ i\in [n] : 0<z(i) < i\}|+\ell_0(w) \quand  \cAB(\cO) = \bigsqcup_{M \in  \NCSM^{\mathrm{BI}}_k(\gamma)} \cAk(z,M)\]
for all $w \in \cAB(\cO)$, where $k :=\lfloor |p-q-\frac{1}{2}|\rfloor$ and 
$z = -\BCDphi(\gamma)$.
Thus,  $\cAB(\cO) \subseteq \cAk(z)$ and we have
$\cAB(\cO) = \cA(z)$ if $p-q\in \{0,1\}$  and $\gamma$ is alternating.
\end{theorem}

\begin{proof}
Our argument is only slightly more involved than the proof of Theorem~\ref{bt11}.
Let $m = n-k$.
Now, the unique dense $K$-orbit $\cO_{\mathrm{dense}}$ is indexed by the clan
\[
(1,2,\dots,m,+,+,\dots,+,+,m,\dots,2,1)
\quord (1,2,\dots,m,-,-,\dots,-,-,m,\dots,2,1),\]
whose image under $\BCDphi$ is $- y^{\BC}_{n,k}$ from \eqref{ybc-eq}.
A directed path $\cO \xrightarrow{t_{a_1}} \cO' \xrightarrow{t_{a_2}} \cdots \xrightarrow{t_{a_l}} \cO_{\mathrm{dense}}$
 corresponds to a reduced word $\a = (a_1,a_2,\dots,a_l)$ for some 
$w \in \cA(-z,-y^{\BC}_{n,k}) = \cAk(z)^{-1}$
with $w^{-1} \in \cAB(\cO)$,
so we have $\cAB(\cO) \subseteq \cAk(z)$.
From the discussion in Section~\ref{relshape-sect}, 
it follows that the reduced words $\a$ arising from such paths are precisely 
those with $\sh(\a;-z) \in \NCSM^{\mathrm{BI}}_k(\gamma)$;
in particular, Lemma~\ref{tbbb-lem} implies the equivalence with the alternating condition 
on
the $k$ innermost trivial blocks.
Since $\sh(\a;-z) = \sh_k(w)$ by Theorem~\ref{rel-prop1},
we conclude that $w \in \cAB(\cO)$
if and only if $w \in \cAk(z,M)$ for some $M \in \NCSM^{\mathrm{BI}}_k(\gamma)$.
Using \eqref{prime-eq}, the number of doubled edges in our path is 
 $ \ell'(-y^{\BC}_{n,k}) - \ell'(-z) - \pair(\sh_k(w))$. By  \eqref{ell'} we have 
 \[
  \ell'(-y^{\BC}_{n,k}) - \ell'(-z) = (n-k) - (n - \ell'(z)) = \ell'(z) - k,\]
while $\pair(\sh_k(w)) =  |\{ a \in [n] : a \leq -z(a)\}| -\ell_0(w\cdot g^{\BC}_{n,k})$ by Proposition~\ref{pair-prop}.
Notice that 
\[
 |\{ a \in [n] : a \leq -z(a)\}| =  |\{ a \in [n] : z(a) \leq -a\}| =  |\{ a \in [n] : -a\leq z(a)<0\}|
 \]
 while by \eqref{ell'} we have
 \[
 \ell'(z)=\neg(z) + \pair(z) = |\{ b \in [n] : -b\leq z(b)<b\}|.
 \]
Finally, since the product $w\cdot g^{\BC}_{n,k}$ is length additive, 
we have \[\ell_0(w\cdot g^{\BC}_{n,k}) = \ell_0(w) + \ell_0(g^{\BC}_{n,k}) = \ell_0(w) +k.\]
Putting these observations together gives \[d_\cO(w^{-1}) =\ell'(z) - k -\pair(\sh_k(w)) 
= |\{ i\in [n] : 0<z(i) < i\}| +  \ell_0(w) .
 \]
Since $\ell_0(w)=\ell_0(w^{-1})$, this is 
the desired formula.
\end{proof}

%

\begin{example}\label{brion-ex1}
Suppose  $\gamma = (-, +, +, -, -, -, +, +, -)$, which is a symmetric $(2p,2q+1)$-clan with $p=q=2$ and $k =\lfloor |p-q-\frac{1}{2}|\rfloor=0$.
Let $z := -\BCDphi(\gamma) = \bar1\hs\bar2\hs\bar3\hs\bar4$.  Then $\NCSM^{\mathrm{BI}}_k(\gamma)$ has 2 elements 
\[
\ba   \{\{-4, 4\}, \{-3, 3\}, \{-2, -1\}, \{1, 2\}\} &= 
\arcstart{
*{\bullet}  \arc{1.2}{rrrrrrr}    & *{\bullet}   \arc{0.9}{rrrrr}    & *{\bullet}   \arc{0.4}{r}  & *{\bullet}       & *{\bullet}  \arc{0.4}{r} & *{\bullet}    & *{\bullet}   & *{\bullet}    
} 
\arcstop,
\\[-10pt]
\\
\{ \{-4, -3\}, \{-2, -1\}, \{1, 2\},\{3, 4\} \} &= \arcstart
{
*{\bullet}  \arc{0.4}{r}    & *{\bullet}      & *{\bullet}   \arc{0.4}{r}  & *{\bullet}     & *{\bullet}  \arc{0.4}{r}  & *{\bullet}   & *{\bullet} \arc{0.4}{r}   & *{\bullet}    
} 
\arcstop,
\ea
\]
while the larger set $\NCSM(z)$ has size 6.
Consulting Example~\ref{intro-ex}, we see that in this case 
\[\cAB(\cO_\gamma)=
\{\bar 4\hs\bar31\bar2, 3\bar41\bar2, 13\bar4\hs\bar2,1\bar23\bar4\}^{-1}
=\{3\bar4\hs\bar2\hs\bar1, 3\bar41\bar2, 1\bar42\bar3,1\bar23\bar4\}.\]
Here, we have $d_\cO(3\bar4\hs\bar2\hs\bar1) = 3$ and $d_\cO(3\bar41\bar2) = d_\cO(1\bar42\bar3) = d_\cO(1\bar23\bar4) =2 $.
\end{example}

\begin{example}\label{brion-ex2b}
If instead $\gamma = (+, +, +, -, -, -, +, +, +)$ then $p=3$, $q=2$, and $k = \lfloor |p-q-\frac{1}{2}|\rfloor = 1$.
Then $\NCSM^{\mathrm{BI}}_k(\gamma)$ 
for $z := -\BCDphi(\gamma) = \bar1\hs\bar2\hs\bar3\hs\bar4$
only contains $\{\{-4, 4\}, \{-3, 3\}, \{-2, -1\}, \{1, 2\}\} $
and we have $\cAB(\cO_\gamma)=
\{ 4\hs\bar31\bar2\}^{-1}
=\{3\bar4\hs\bar21\}$ and $d_\cO(3\bar4\hs\bar2 1) = 2$.
\end{example}

\section{Hecke atoms}\label{hecke-sect}
In this section, we include a few statements about 
the sets  $\HA(z) = \{ w \in \W_n : w^{-1}\circ w = z\}$ for $z \in \I(\W_n)$.
%
Let $\approx_A$ be the transitive closure of the symmetric relation on words with
\be\label{hecke-eq}
\cdots bca \cdots \approx_A \cdots cab \cdots \approx_A \cdots cba \cdots
\ee
whenever $a < b < c$, where the corresponding symbols $\cdots$ mask identical subwords.
This relation generates the so-called \emph{Chinese monoid} studied previously in \cite{CEHKN,DuchampKrob}.
Results in \cite[\S6.1]{HMP2} show that 
%
the $\approx_A$-equivalence classes in $S_n$ are the sets $\HA(z)^{-1}$ for $z \in \I(S_n)$.
A similar result holds for the affine symmetric group \cite[Proposition 1.10]{M}.
%
Here, we  describe a signed extension of this property.
We apply $\approx_A$ to signed permutations via their one-line representations.

\begin{lemma}\label{post-a-approx-lem}
If $v,w \in \W_n$ are such that  $v \approx_A w$,
then $v\circ v^{-1} = w\circ w^{-1}$.
\end{lemma}

\begin{proof}
Suppose $u,v,w \in \W_n$ satisfy $ u =\cdots bca \cdots \approx_A v = \cdots cab \cdots \approx_A w = \cdots cba \cdots$ for numbers $a<b<c$ appearing in positions $i$, $i+1$, and $i+2$.
Then  $u=\sigma t_it_{i+1}t_i = \sigma t_{i+1}t_it_{i+1}$, $v = \sigma t_{i+1}t_{i}$, and $w=\sigma t_{i}t_{i+1}$
for some $\sigma \in \W_n$ with $\ell(u) =\ell(\sigma)+3$ and $\ell(w) =\ell(v) = \ell(\sigma) + 2$.
It follows that $u\circ u^{-1} = v\circ v^{-1} = w \circ w^{-1}$ since 
$ t_it_{i+1}t_i \circ t_it_{i+1}t_i = t_{i+1}t_i \circ t_i t_{i+1} = t_i t_{i+1}\circ t_{i+1} t_i$.
\end{proof}

Define
 $\approx_B$ to be the transitive closure of $\approx_A$ together with the
 relation on $n$-letter words
 that
has $u\approx_B v \approx_B w$ either if
 $u_j=v_j=w_j$ for $j \notin\{1,2\}$ while
\be\label{hb-eq1}
u_1u_{2} = \bar{a}\bar{b},
\qquad
v_1v_{2} = \bar{b}\bar{a},
\qquand
w_1w_{2} =a\bar{b}\ee
for numbers $0<a<b$, or if
$u_j=v_j=w_j$ for $j \notin \{1,2,3\}$ while
\be\label{hb-eq2}
u_1u_2u_3 = \bar{c}\hs\bar{a}\bar{b},
\qquad
v_1v_2v_3 = \bar{c}\bar{b}\bar{a},
\qquand
w_1w_2w_3 =\bar{c}a\bar{b}\ee
for numbers $0<a<b<c$.
Then $\approx_B$   includes $<_A$, $\sim_A$, $\approx_A$, $<_B$, and $\sim_B$ as subrelations.
Further relations can be derived by combining~\eqref{hb-eq1} and~\eqref{hb-eq2}.
For example, $\bar 3 1 \bar 2 \approx_B \bar1\hs\bar2\hs\bar3$ via the chain $\bar 3 1 \bar 2 \approx_B \overline{3} \hs \overline{1} \hs\overline{2} \approx_B \overline{1}\hs\overline{3}\hs\overline{2} \approx_A \overline{1} \hs \overline{2} \hs \overline{3}$ where the first $\approx_B$ is from~\eqref{hb-eq2} and the second $\approx_B$ is from~\eqref{hb-eq1}.

\begin{lemma}\label{h-con-lem}
If $v,w \in \W_n$ are such that $v \approx_B w$,
then $v\circ v^{-1}=w\circ w^{-1}$.
\end{lemma}

\begin{proof}
If $u$, $v$, $w$ are as in \eqref{hb-eq1}
then there exists a common element $\sigma \in \W_n$
such that 
$u = \sigma\circ \bar{1}\hs\bar{2}$,
$v = \sigma\circ \bar{2}\hs\bar{1}$, $w = \sigma\circ 1 \bar 2$,
and $\ell(u) -1 =\ell(v) =\ell(w) = \ell(\sigma)+3$.
If $u$, $v$, $w$ are as in \eqref{hb-eq2}
then there exists $\sigma \in \W_n$
such that 
$u =\sigma \circ \bar{3}\hs\bar{1}\hs\bar{2}$,
$v = \sigma\circ \bar{3}\hs\bar{2}\hs\bar{1}$,  $w = \sigma\circ \bar3 1\bar 2$,
and $\ell(u)-1=\ell(v)=\ell(w) =\ell(\sigma) = 6$.
The lemma follows by checking 
that the inverses of
$ \bar{1}\hs\bar{2} = t_1t_0t_1t_0,$ $ \bar{2}\hs\bar{1}=t_0t_1t_0,$
and $ 1\bar 2=t_1t_0t_1$ are all Hecke atoms of $\bar{1}\hs\bar{2}$,
while the inverses of
$ \bar{3}\hs\bar{1}\hs\bar{2} = t_1t_0t_1t_0t_2t_1t_0,$ $ \bar{3}\hs\bar{2}\hs\bar{1}=t_0t_1t_0t_2t_1t_0,$ and $ \bar 3 1\bar 2 = t_1t_0t_1t_2t_1t_0$
are all Hecke atoms of $\bar{1}\hs\bar{2}\hs\bar{3}$.
\end{proof}

\begin{lemma}\label{h-pre-lem}
Let $u,v,w$ be words with $n$ letters.
Suppose, for some $i \in [n-1]$, 
that 
$v_1 < v_2 < \dots < v_i = u_{i+1} = w_{i+1} < v_{i+1}  = u_i = -w_i < 0$
and $u_j=v_j=w_j$ if $j\notin\{i,i+1\}$. Then $u\approx_B v \approx_B w$.
\end{lemma}


\begin{proof}
Define $v' $ and $v''$ from $v$ as in the proof of Lemma~\ref{jump-lem}. 
That result already shows that $v \sim_B w$, which implies $v\approx_B w$,
so we only need to check that $u\approx_B v$. 
Define $u'$ and $u''$ from $u$ analogously: in other words,
form $u'$ from $u$
by replacing  $u_{j-1}u_j$ by $\bar{u_j}u_{j-1}$ for each even index $j<i$;
then, if $i$ is odd (respectively, even), define $u''$ by removing the subword $u_iu_{i+1}$ 
(respectively, $u_{i-1}u_iu_{i+1}$)
from $u'$ and placing it at the start of the word.
Lemma~\ref{jump-lem} shows that that $u\sim_B u'$ and $v'\sim_B v$,
it is an exercise to check that $u' < _A u''$ and $v' <_A v''$,
and  by definition  $u'' \approx_B v''$.
\end{proof}


\begin{theorem}
The $\approx_B$-equivalence classes in $\W_n$ are  the sets $\HA(z)^{-1}$ for $z \in \I(\W_n)$.
\end{theorem}

\begin{proof}
Lemma~\ref{h-con-lem}
implies that each set $\HA(z)^{-1}$ for $z \in \I(\W_n)$
is preserved by $\approx_B$.
Let $w \in \W_n$.
It suffices to show that $w$ is equivalent under $\approx_B$
to an element of $\cA(z)^{-1}$ for some $z \in \I(\W_n)$.
By Theorem~\ref{0-z-thm}, we have $v\approx_A w$
for a signed permutation $v \in \W_n$ of the form $v=[[b_1a_1b_2a_2\cdots b_l a_l]] $
where $a_i\leq b_i$ for each $i \in [l]$ and $a_1<a_2<\dots<a_l$.
Consider the minimal index $j$ with $a_j<b_j <|a_j|$, if such an index exists.
We outline a procedure to create another signed permutation $u \approx_B v$ whose first one-line descent occurs strictly farther to the right:
\begin{enumerate}
	\item Use the relation $\sim_A$ to move all descents $b_ia_i$ with $i<j$ and $a_j<0<|a_i|<b_i$ to the right of $b_ja_j$.
	\item Apply Lemma~\ref{h-pre-lem} to change $b_ja_j$ to $a_j\bar{c}$
where $c =|b_j| < -a_j$.
	\item Use $\sim_A$ to move the descents $b_ia_i$ back to their original positions.
	\item Use $\approx_A$ to transform the subword $[[\hs\bar c\hs b_{j+1}a_{j+1} \cdots b_la_l]]$
 to a word of the form $[[b'_{1}a'_{1}\cdots b'_ka'_k]]$ where $a_i'\leq b_i'$ for each $i \in [k]$ and 
 $a_1<\dots<a_{j} < a_1' < \dots <a_k'$.
\end{enumerate}
The resulting $u$ is equivalent to $v$ under $\approx_B$ and of the same form, but in which the 
first occurrence of a one-line descent $ba$ with $a<b<|a|$, if one exists, is farther to the right than before.
By repeating this process and replacing $v$ with the result,
we may assume that $w\approx_B v$ where $v$, defined as above, has no descents $b_ja_j$ with $a_j<b_j<|a_j|$.
This element $v \in \W_n$ is equal to $0_B(z) \in \cA(z)^{-1}$ for the involution $z \in \I(\W_n)$
whose cycles in $[\pm n]$ consist of $\{ -a_j,a_j\}$ for each $j \in [l]$ with $a_j=b_j$
together with $\{a_j,b_j\}$ and $\{-a_j,-b_j\}$ for each $j \in [l]$ with $a_j < b_j$.
\end{proof}

\section{Atomic elements}\label{atomic-sect}

An involution $z$ in a Coxeter group is \emph{atomic} if $|\cA(z)| = 1$.
An element of a Coxeter group is \emph{fully commutative} if each of its reduced words can be transformed to any other
by a sequence commutations. 
A fully commutative involution is always atomic \cite[Proposition 7.12]{HMP2}.
The converse holds  in type A 
\cite[Corollary 6.17]{M2}, but does not extend to $\W_n$.
Theorem~\ref{minmax-thm} implies:

%
%
%


\begin{proposition}\label{mem-prop}
Let $z \in \I(\W_n)$ and define $M_{\min} = \{ \{-i,i\} : i \in \Neg(z)\}$.
Then $z$
 is atomic if and only if $\neg(z) \leq 1$ and $0_B(z) =0_B(z,M_{\min})= 1_B(z,M_{\min})$.
\end{proposition}



\begin{proposition}\label{nn-lem}
An involution in $\W_n$ is atomic if and only if
it has at most one negated point and 
it does not have two cycles $\{a,d\},\{b,c\}$ in $[\pm n]$
with $-d \neq a<b\leq c < d$.
\end{proposition}

 
\begin{proof}
Suppose $z \in \I(\W_n)$ has $\neg(z) \leq 1$ and $M_{\min} = \{ \{-i,i\} : i \in \Neg(z)\}$.
If the given condition holds then no pairs $(a,d),(b,c) \in \Cyc_B(z)$ can have $a < b\leq c < d$, so $0_B(z,M_{\min})=1_B(z,M_{\min})$.
Conversely, suppose $z$ has two cycles $\{a,d\},\{b,c\}$ in $[\pm n]$
with $-d \neq a<b\leq c < d$. Since the set of cycles of $z$ is symmetric under the map induced by
$i \mapsto -i$, we may assume that $|a| < d$. By invoking this symmetry a second time, we may further assume
that either $0<b=c$ or $|b| < c$. But then $(a,d)$ and $(b,c)$ are both in $\Cyc_B(z)$,
so  $0_B(z,M_{\min}) \neq 1_B(z,M_{\min})$.
\end{proof}

 We can describe the atomic elements of $\I(\W_n)$ more precisely.
Let $\X^0_{n}$ and $\X^1_n$ be the sets of atomic involutions in $\W_n$ with 0 and 1 negated points, respectively,
and let $\X_n = \X^0_n \sqcup \X^1_n$.
Define the \emph{radius} of $z \in \X_n$ to be the largest integer $r \in [n]$ such that $z(r) < -r$,
or 0 if no such $r$ exists. Denote the radius of $z$ by $\rho(z)$, and
let $\X_{n,r} = \{ z \in \X_n : \rho(z) = r\}$ and $\X^i_{n,r} = \X^i_n \cap \X_{n,r}$.

\begin{lemma}\label{radius-bound-lem}
If $z \in \X_n$ then $\rho(z) \leq \lfloor n/2 \rfloor$.
\end{lemma}

\begin{proof}
Suppose $n<2r$ and $z \in \I(\W_n)$ has $z(r) < - r$.
Since $[\pm n]\setminus [\pm r]$ has $2n-2r < 2r$ elements,
$z$ must have a cycle $\{i,j\}$ with $z(r)<-r<i<j<r$, so $z$ is not atomic by Proposition~\ref{nn-lem}.
\end{proof}

Let  $r = \lfloor n/2\rfloor$. When $n\geq 2$,
define 
$\eta : \{ \pm 1\}^{r-1} \to \I(\W_n)$ as the map
given as follows: for $\epsilon=(\epsilon_1, \epsilon_2,\dots,\epsilon_{r-1}) \in \{\pm1\}^{r-1}$,
let $a_1<a_2<\dots<a_{r-1}$ be the numbers $i \cdot \epsilon_i $ for $i \in [r-1]$ listed in order,
and let $\eta(\epsilon)$ 
be the unique involution $z \in \I(\W_n)$ with $z(r+1)= -r$,
with $z(r+1+i) = a_i$ for $i \in [r-1]$, and with $z(n) =n$ if $n=2r+1$ is odd.
The map $\eta$ is clearly injective.
For example, if $n=7$ so  $r=3$ and $\epsilon=(-1,+1)$, then $a_1=-1 $ and $a_2=2$ so 
$\eta(\epsilon) =  \bar 5 6 \bar 4 \hs\bar 3\hs \bar 1 2 7$.
Define 
$\Y_n= \left\{ z \in \X^0_n: z(\lfloor n/2\rfloor) < -\lfloor n/2\rfloor\right\}$ when $n\geq 2$,
and set $\Y_0 =\Y_1 = \{1\} \subset \W_n$.

\begin{lemma}\label{eta-lem1}
Let $r = \lfloor n/2\rfloor$. If $n\geq 2$ 
then $ \X^0_{n,r} = \Y_n$ and $\eta : \{\pm 1\}^{r-1} \to \Y_n$ is a bijection.
\end{lemma}

\begin{proof}
Assume $n\geq 2$
and let $z \in \Y_n$. Since $z$ is atomic with no negated points,
every number $i \in [\pm r]$ must have $z(i)<-r$ or $r<z(i)$.
It follows that there are $2r$ numbers $i \in [\pm n] \setminus [\pm r]$
with $z(i) \in [\pm r]$.
Since $[\pm n]\setminus [\pm r]$ has at most $2r+2$ elements,
we deduce that $z(n) = n$ if $n$ is odd, and that every $i \in [\pm 2r]\setminus [\pm r]$
has $z(i) \in [\pm r]$. 
Thus $\X^0_{n,r}\supset \Y_n$, so we have $ \X^0_{n,r} = \Y_n$ since the reverse containment holds by definition.
As $z(r) < -r$ and $r <z(-r)$, 
it must hold that $-r=z(r+1)$ since otherwise we would have $-r<z(r+1) < r+1 < z(r)$, contradicting Proposition~\ref{nn-lem}.
By the same lemma, it follows that $-r = z(r+1) < z(r+2) < \dots < z(2r) < r.$
We conclude that if
$\epsilon \in \{\pm 1\}^{r-1}$
is the sequence of signs of $z(r+2),z(r+3),\dots,z(2r)$ then
 $z= \eta(\epsilon)$.

Consider an arbitrary sequence $\epsilon = (\epsilon_1,\epsilon_2,\dots,\epsilon_{r-1}) \in \{\pm 1\}^{r-1}$.
The involution $\eta(\epsilon)$ has no negated points and satisfies $\eta(\epsilon)(r)<-r$, so to finish the proof of the lemma it suffices to 
check that $\eta(\epsilon)$ is atomic, whence contained in $\Y_n$.
This is easy to deduce from Proposition~\ref{nn-lem}.
\end{proof}

Let $\Z_n$ be the set of atomic involutions in $S_n$.
For each $r\in \NN$, 
define $\Z_{n,r}$ as the subset of involutions $w \in \Z_n$ with $i<w(i)$ for $i \in [r]$,
so that $\Z_{n,0}=\Z_n$ and $\Z_{n,r}=\varnothing$ if $2r >n$.

\begin{lemma}\label{eta-lem2}
If $0\leq r \leq \lfloor n/2\rfloor$
 then  $|\Z_{n,r}| = \binom{n-r}{\lceil n/2\rceil}$.
\end{lemma}

\begin{proof}
The atomic involutions in $S_n$ are the same as the involutions that 
are fully commutative (i.e., 321-avoiding), so $|\Z_{n,0}| = \binom{n}{\lceil n/2\rceil}$ \cite{SiS}.
If $z \in \Z_n$ then either $z \in \Z_{n,1}$ or $z(1) = 1$, 
so $|\Z_{n}| = |\Z_{n-1}| + |\Z_{n,1}|$ and $|\Z_{n,1}| = |\Z_n| - |\Z_{n-1}| = \binom{n}{\lceil n/2\rceil}-  \binom{n-1}{\lceil (n-1)/2\rceil}= \binom{n-1}{\lceil n/2\rceil}$.
Assume $0<r\leq \lfloor n/2\rfloor$.
If $z \in \Z_{n,r}$ and $z(r+1)<r+1$ then necessarily $z(r+1) = 1$.
In this case, removing $1$ and $r+1$ from the one-line
representation of $z$ and standardizing what remains produces an arbitrary element of $\Z_{n-2,r-1}$.
It is not possible for $z \in \Z_{n,r}$ to have $z(r+1)=r+1$ and the set of elements $z \in \Z_{n,r}$ with $r+1<z(r+1)$
is precisely $\Z_{n,r+1}$. We conclude that $|\Z_{n,r}| = |\Z_{n-2,r-1}| +|\Z_{n,r+1}|$,
so by induction $|\Z_{n,r+1}| = |\Z_{n,r}| - |\Z_{n-2,r-1}| = \binom{n-r}{\lceil n/2\rceil} -\binom{n-r-1}{\lceil n/2\rceil-1} = \binom{n-r-1}{\lceil n/2\rceil}$.
\end{proof}

Fix $0<r \leq \lfloor n/2\rfloor$ and $x \in \X^0_{n,r}$. Let $I = [\pm r] \sqcup x([\pm r]) $  and $J = x([\pm r]) \cap [n]$.
Since $x([\pm r])\cap [\pm r]  = \varnothing$, and we have $|I|=4r$ and $|J| = r$.
Define $y = \phi^{-1} \circ x \circ \phi \in \I(\W_{2r})$
where $\phi $ is the unique order-preserving bijection
$[\pm 2r] \to I$.
Now let $j_1<j_2<\dots<j_r$ be the distinct elements of $J$, set $w = (1,j_1)(2,j_2)\cdots (r,j_r) \in S_n$,
and define $z \in \I(S_n)$ to be the involution with
$z(i) = w(i)$ if $ i \in [r]\sqcup J$ and with $z(i) = x(i)$ for all other $i \in[n]$.
We finally set $\pi^0(x)= (y,z)$,
where if $r=0$  then we define $y=1 $ and $ z = x|_{[n]} \in \I(S_n)$. 

\begin{example}
This map may be understood in terms of the symmetric matchings in $[\pm n]$ that 
we draw to represent involutions in $\W_n$. For example, if $n=12$, $r=3$, and $x \in \X^0_{n,r}$  is
\[\label{x-pre-eq}
x = \Bigl( \arcstart
{
*{\bullet}  \arc{0.4}{r} & *{\bullet}    & 
*{\bullet}  \arc{0.6}{rrr} & *{\bullet} \arc{1.0}{rrrrr}   & *{\bullet} \arc{1.2}{rrrrrr}    & *{\bullet}     & *{\bullet}\arc{1.2}{rrrrrr}  & *{\bullet} \arc{1.4}{rrrrrrr}& *{\bullet}    & *{\bullet}    \arc{1.4}{rrrrrrr} & *{\bullet}  & *{\bullet} \arc{1.2}{rrrrrr}
&
*{\bullet}   & *{\bullet} \arc{1.2}{rrrrrr}   & *{\bullet}     & *{\bullet}   \arc{1.0}{rrrrr}    & *{\bullet}   & *{\bullet} & *{\bullet}  \arc{0.6}{rrr}    & *{\bullet}     & *{\bullet}  & *{\bullet}    
&
*{\bullet}  \arc{0.4}{r} & *{\bullet} \\
\bar{12} & \bar{11} &\bar{10} & \bar9 & \bar8 & \bar7 & \bar6 & \bar5 & \bar4 & \bar3 & \bar2 & \bar1 & 1 & 2 &3 & 4 &5 & 6 & 7 & 8 &9 & 10 &11 &12
} 
\arcstop\Bigr)
\]
then $y$ is obtained by first removing all edges  that do not have an endpoint in $[\pm r]$, then 
removing the isolated vertices from this picture, and finally standardizing what remains. This gives
\[
y  = \Bigl(
 \arcstart
{
 *{\bullet} \arc{0.8}{rrrr}        & *{\bullet}\arc{1.0}{rrrrr}  & *{\bullet} \arc{1.2}{rrrrrr}    & *{\bullet}    \arc{1.2}{rrrrrr} & *{\bullet}  & *{\bullet} \arc{1.0}{rrrrr}
&
*{\bullet}   & *{\bullet} \arc{0.8}{rrrr}   & *{\bullet}     & *{\bullet}     & *{\bullet}   & *{\bullet}     
 \\
 \bar6 & \bar5 & \bar4 & \bar3 & \bar2 & \bar1 & 1 & 2 &3 & 4 &5 & 6  
} 
\arcstop\Bigr)  \in \I(\W_{2r}).
\]
To construct $z$, we remove from $x$ all edges without an endpoint in $[n]\setminus[r] = \{4,5,\dots,12\}$
and all isolated vertices up to $r=3$. Then we relabel the endpoints $\bar3$, $\bar 1$, $2$ as 1, 2, 3 to get
\[
z = \Bigl(
 \arcstart
{
*{\bullet}  \arc{0.8}{rrrr}  & *{\bullet}   \arc{0.8}{rrrr} & *{\bullet}   \arc{1.0}{rrrrr}   & *{\bullet}   \arc{1.0}{rrrrr}    & *{\bullet}   & *{\bullet} & *{\bullet}  \arc{0.6}{rrr}    & *{\bullet}     & *{\bullet}  & *{\bullet}    
&
*{\bullet}  \arc{0.4}{r} & *{\bullet} \\
 1 & 2 &3 & 4 &5 & 6 & 7 & 8 &9 & 10 &11 &12
} 
\arcstop
 \Bigr) \in \I(S_{n}).
\]
\end{example}

We also have a simpler map 
$ \pi^1 : \X^1_{n+1,r} \to  \I(\W_n) \times \{r+1,r+2,\dots,n+1\} $. If $x \in \X^1_{n+1,r}$ has $\Neg(x) = \{ m\}$,
and
$\psi$ is the unique
order-preserving bijection $[\pm n] \to [\pm (n+1)] \setminus \{\pm m\}$,
then we set $y=\psi^{-1} \circ x \circ \psi$ and $\pi^1(x) = (y,m)$.
In terms of matchings, $y$ is obtained from $x$ by removing the single symmetric edge $\{-m,m\}$ and standardizing the remaining vertices.

\begin{lemma}\label{eta-lem3}
Choose an integer $0 \leq r \leq \lfloor n/2\rfloor$. Then the maps
$ \pi^0 : \X^0_{n,r} \to \Y_{2r} \times \Z_{n,r}$ and 
$\pi^1:\X^1_{n+1,r} \to  \X^0_{n,r}\times \{r+1,r+2,\dots,n+1\}$
are bijections.
%
\end{lemma}

\begin{proof}
One can verify the lemma directly
when $r=0$, so assume $0<r\leq \lfloor n/2\rfloor$.
Let $x \in \X^0_{n,r}$ and  $(y,z) = \pi^0(x) \in  \I(\W_{2r})\times \I(S_n)$.
By construction $y$ has no negated points and satisfies $y(r) < -r$. 
It follows from Proposition~\ref{nn-lem} that $y$ and $z$ are also atomic, so we have $y \in \Y_{2r}$ and $z \in \Z_{n,r}$.
To show that $\pi^0$ is a bijection, consider the inverse map defined
as follows. Given $(y,z) \in \Y_{2r} \times \Z_{n,r}$,
let $E = [\pm r] \sqcup z([r]) \sqcup -z([r])$, write $\theta$ for the unique order preserving bijection $[\pm 2r]\to E$,
and define $x \in \W_n$ 
as the permutation with $x(i) = (\theta \circ y \circ \theta^{-1})(i)$ for $i \in E$
and with $x(i) = z(i)$ and $x(-i) = -z(i)$ for $i \in [n]\setminus E$. 
Since $y$ and $z$ are both atomic and since $y$ has no negated points,   Proposition~\ref{nn-lem} implies
that $x \in \X^0_{n,r}$, and it is easy to see that $(y,z) \mapsto x$ is the inverse of $\pi^0$.

For the second map,
suppose $x \in \X^1_{n+1,r}$ and $(y,m) = \pi^1(x)$. 
Proposition~\ref{nn-lem}
implies that the 
single negated point $m \in \Neg(x)$ is greater than $r$,
so $y \in \X^0_{n,r}$.
It is straightforward to construct an inverse map $ \X^0_{n,r}\times \{r+1,r+2,\dots,n+1\}  \to \X^1_{n+1,r}$,
and we conclude that $\pi^1$ is also a bijection.
\end{proof}


\begin{theorem}\label{form-thm}
Suppose $r \in \NN$. The following identities hold:
 \ben
 
 \item[(a)] It holds that $|\X^0_{n,r}| = \left\lceil 2^{r-1} \right\rceil \binom{n-r}{\lceil n/2\rceil}
 $
 and
 $
 |\X^1_{n+1,r}| = (\lceil n/2\rceil+1) \left\lceil 2^{r-1}\right\rceil \binom{n-r+1}{\lceil n/2\rceil+1}.
 $
 
 \item[(b)] If $n$ is odd then $|\X^1_{n,r}| = \frac{1}{2}(n+1) |\X^0_{n,r}|$.
 
 \item[(c)] If $n$ is even and $r>0$ then $|\X^1_{n,0}| =  \frac{1}{2}(n+2)|\X^0_{n+1,1}|$ and $|\X^1_{n,r}| = \frac{1}{4}(n+2)|\X^0_{n+1,r+1}|$.
 
 \een
\end{theorem}

\begin{proof}
Part (a) follows from Lemmas~\ref{eta-lem1}, \ref{eta-lem2}, and \ref{eta-lem3}.
Parts (b) and (c) follow from (a).
\end{proof}

The elements of $\X_n^0$
are also naturally partitioned by their absolute lengths.
Let $\X_n^{0,k}$ be the set of atomic involutions $z \in \I(\W_n)$ with zero negated points
and  absolute length $\ell'(z) =k$.
Equivalently, $\X_n^{0,k}$ is the set of atomic involutions in $ \W_n$ with  $2k$ distinct 2-element cycles in $[\pm n]$.
To count the elements in these sets, we relate them to lattice paths of the following type.

Define $\cD_{n}$ as the set of $n$-step paths $(p_0,p_1,\dots,p_{n})$ in the nonnegative quadrant  $\NN^2$
that begin at $p_0=(0,0)$ and end at a point $p_{n} \in \{ (n,2m): m \in \NN\}$,
that have $p_i-p_{i-1} \in \{ (1,1),(1,-1),(1,0)\}$ for each $i \in [n]$,
but that have $p_i-p_{i-1} = (1,0)$ only if $p_i $ is on the $x$-axis.
Paths of this type  terminating at $(n,0)$ are sometimes called \emph{dispersed Dyck paths}.
Each path in $\cD_n$ must have an even number of steps not equal to $(1,0)$.
For each $k \in \NN$
let $\cD_{n,k}$ denote the subset of paths in $\cD_n$ that have $p_i - p_{i-1} = (1,0)$ for exactly $n-2k$ values of $i \in [n]$.

\begin{lemma}
If $0 \leq k \leq \lfloor n/2\rfloor$ then $|\cD_{n,k}| = \binom{n}{k}$.
\end{lemma}

\begin{proof}
Among the $2k$-step paths starting at the origin in $\ZZ^2$ using just the steps $(1,1)$ and $(1,-1)$,
those which do not stay in $\NN^2$ are in bijection with those
which do not terminate at $(2k,0)$ by the reflection principle; the number of paths of the latter type is
$\sum_{j \neq k} \binom{2k}{j}$, and
subtracting this  from $2^{2k}$ gives $|\cD_{2k,k}| = \binom{2k}{k}$.
It is also apparent that $|\cD_{n,0}| = 1$ for all $n \in \NN$.

Assume $0<k<\lfloor n/2\rfloor$.
The subset of paths in $\cD_{n,k}$ beginning with a horizontal step
are clearly in bijection with $\cD_{n-1,k}$,
while the subset of paths in $\cD_{n,k}$ beginning with an up step
are in bijection with $\cD_{n-1,k-1}$ via the following operation:
given a path in $\cD_{n,k}$,
remove its initial up step and replace the first down step which returns  to the $x$-axis
with a horizontal step.
Such a down step exists since a path in $\cD_{n,k}$ contains $n-2k>0$ horizontal steps.
We deduce that $|\cD_{n,k}| = |\cD_{n-1,k}| + |\cD_{n-1,k-1}|$,
so by induction $|\cD_{n,k}| = \binom{n}{k}$ for all $k \in \NN$.
\end{proof}

\begin{theorem}\label{binom-thm}
If $0 \leq k \leq \lfloor n/2\rfloor$ then  $|\X^{0,k}_n|= \binom{n}{k}$.
\end{theorem}

\begin{proof}
By the previous lemma, it suffices to construct a bijection $\X^{0,k}_n \to \cD_{n,k}$.
Given $z \in \X^{0,k}_n$,
define $p=(p_0,p_1,\dots,p_n)$ as the path starting at 
$p_0=(0,0)$  for which the step $p_i - p_{i-1}$ is
given by $(1,0)$, $(1,1)$, or $(1,-1)$
according to whether $j=-n + i-1$ has $z(j)=j$, $j<z(j)$, or $z(j)<j$, respectively.
It follows from Proposition~\ref{nn-lem}
 that $p \in \cD_{n,k}$,
so $z \mapsto p$ gives a map $\cX^{0,k}_n \to \cD_{n,k}$.

One defines an inverse map as follows.
Fix  a path $p=(p_0,p_1,\dots,p_n) \in \cD_{n,k}$ and
let $U$ and $D$ be the respective set of indices $i \in [n]$ where $p_i-p_{i-1}=(1,-1)$
and $p_i-p_{i-1} = (1,-1)$.
Write $a_0<a_1<\dots<a_{2k-1}$ for the numbers in $\{-n+i-1 : i \in U\} \sqcup \{ n-i+1 : i \in D\}$
arranged in order, and define $z\in \I(\W_n)$ to be the unique involution that has 
$z(a_i) = -a_{2k-i}$ for $i=0,1,\dots,2k-1$
and that fixes all numbers not equal to $ a_i$ or $-a_i$ for some $i$.
Since the path $p$ remains in $\NN^2$, we have $a_i < -a_{2k-i}$ for each $i$.
An index $i \in [n]$ corresponds to a horizontal step in $p$
if and only if $-n+i-1$ and $n-i+1$ 
are fixed points of $z$; since these steps are all at height zero,
$z$ has no fixed points $b$ with $a<b<z(a)$ for any $a \in [\pm n]$.
This is enough to conclude by Proposition~\ref{nn-lem} that $z$ is atomic.
Since $z$ has  $2k$ left endpoints $i \in [ \pm n]$ with $i<z(i)$,
it follows that $z$ has no negated points and belongs to $\X^{0,k}_n$.
Moreover, it holds essentially by definition that $p\mapsto z$ is the inverse of the 
map $z \mapsto p$ described in the first paragraph.
Thus $|\X^{0,k}_n| = |\cD_{n,k}| = \binom{n}{k}$.
\end{proof}


\begin{corollary}\label{atom-cor1}
The number $a_n^0 =|\X_n^0|$ of atomic involutions in $\W_n$ with no negated points satisfies $a_n^0 = 2^{n-1}$ if $n$ is odd
and $a_n^0 = 2^{n-1} + \frac{1}{2} \binom{n}{n/2}$ if $n$ is even.
\end{corollary}

\begin{proof}
Rewrite $|\X_n^0| = \sum_r |\X^0_{n,r}| = \sum_k |\X^{0,k}_n|$ using Theorems~\ref{form-thm} and \ref{binom-thm}.
\end{proof}

The sequence $\left\{a_n^0\right\}_{n=0,1,2,\dots} = (1, 1, 3, 4, 11, 16, 42, 64,\dots)$ is \cite[A027306]{OEIS}.

\begin{corollary}\label{atom-cor2}
The number $a_n^1 = |\X_n^1|$ of atomic involutions in $\W_n$ with one negated point
is either  $a_n^1 = (n+1)2^{n-2}$ when $n$ is odd  
or $a_n^1= \frac{1}{4} (n+2)( 2^n - \binom{n}{n/2})$ when $n$ is even. 
\end{corollary}

\begin{proof}
Theorem~\ref{form-thm}
 implies that if  $n$ is odd then
 $|\X^1_n| = \sum_r |\X^1_{n,r}|=\frac{1}{2}(n+1)|\X_n^0|$ and
 if $n$ is even then
 $|\X^1_n|=\frac{1}{4}(n+2) \(|\X^0_{n+1}|  - |\X^0_{n+1,0}| + |\X^0_{n+1,1}|\)
=\frac{1}{4}(n+2) \(|\X^0_{n+1}|  -\binom{n+1}{n/2+1}+ \binom{n}{n/2+1}\)$.
The corollary follows by substituting Corollary~\ref{atom-cor1}
and the identity $\binom{n+1}{n/2+1} = \binom{n}{n/2+1} + \binom{n}{n/2}$.
\end{proof}


\begin{corollary}\label{atom-cor}
The total number $a_n =a_n^0 + a_n^1$ of involutions $z \in \W_n$ with $|\cA(z)| = 1$ is 
either $a_n = (n+3) 2^{n-2}$ when $n$ is odd or $a_n = (n+4)2^{n-2} -  \frac{n}{4} \binom{n}{n/2} $ when $n$ is even.
%
\end{corollary}

The even-indexed terms of $\left\{a_n\right\}_{n=0,1,2,\dots} = (1,2,5,12,26,64,130,320,\dots)$
form the sequence \cite[A003583]{OEIS}. The odd-indexed terms are a subsequence of \cite[A045623]{OEIS}.

Recall that $\cR(w)$ is the set of reduced words for $w \in \I(\W_n)$ and  
 $\iR(z) = \bigsqcup_{w \in \cA(z)} \cR(w)$.
Suppose $\lambda $ is an integer partition and 
$\mu = (\mu_1 > \dots > \mu_l > 0)$ is a strict partition.
%
Let $f^\lambda$ be the number of standard tableaux of shape $\lambda$.
Let $g^\mu$ be the number of \emph{(unprimed) standard shifted tableaux} of shape $\mu$,
by which we mean
arrangements of $1,2,\dots,|\mu|$ in the \emph{shifted diagram}
$\{ (i,i+j-1): i \in [l],\ j\in [\mu_i]\}$
such that rows and columns are increasing.
 
  
\begin{corollary}\label{gamma-prop}
Let $p = \lfloor \frac{n+1}{2} \rfloor$, $q = \lceil \frac{n+1}{2}\rceil$,
and $\gamma_n = g^{\BC}_{n,n} := \bar{n} \cdots \bar{2}\hs\bar{1} \in \I(\W_n)$.
Then 
$|\iR(\gamma_n)| = f^{\lambda} = g^{\mu}$ 
for $\lambda = p^q = (p,p,\dots,p)$ and $\mu= (n,n-2,n-4,\dots)$.
\end{corollary}


\begin{proof}
By Lemma~\ref{0isatom-lem} and Proposition~\ref{nn-lem}, $\gamma_n$ is atomic
and its unique atom $v =0_B(z)^{-1}$ is either 
$v=\bar{n}\hs\bar{(n-2)} \cdots \bar{4}\hs\bar{2}135\cdots(n-1)$
or
$v=\bar{n}\hs\bar{(n-2)}\cdots \bar{3}\hs\bar{1}246\cdots (n-1).$
We have 
 $|\cR(v)| = g^\mu$ by \cite[Corollaries 3.3 and 4.4]{TKLam}. 
The equality $f^{\lambda} = g^{\mu}$ is a special case of \cite[Proposition 8.11]{HaimanMixed}.
\end{proof}

The sequences $\{|\iR(\gamma_n)|\}_{n=1,3,5,\dots}= (1,2, 42, 24024,701149020,\dots)$ 
and 
$\{|\iR(\gamma_n)|\}_{n=2,4,6,\dots} = (1, 5, 462, 1662804,\dots)$
are \cite[A039622]{OEIS} and \cite[A060855]{OEIS}. 

\newpage
\appendix
\section{Index of symbols}
\label{not-sect}

The table below lists our non-standard notations, with references to definitions where relevant.

\def\skip{\\[-4pt]}

\begin{center}
{\begin{tabular}{l | l r}
\text{Symbol} & \text{Meaning} & 
 \\
 \hline
  $S_X$ & Group of permutations of a finite set $X$ & \S\ref{nest-sect} \\
$\W_n$ & Group of signed permutations of $[\pm n] = \{\pm 1,\dots,\pm n\}$ & \S\ref{intro2-sect} \\
$s_i$, $t_i$ &  $s_i=(i,i+1)$ and $t_i =(-i-1,-i)(i,i+1)$, except $t_0 =(-1,1)$  \\
$\circ$ & Demazure product on a Coxeter group $W$ & \S\ref{dem-subsect} \\
$<_\I$ & Involution weak order in $\I(W)$, defined by $x \leq_\I w^{-1} \circ x \circ w$ & \S\ref{intro2-sect} \\
\skip
$\ell$, $\ell'$, $\ellhat$ & Coxeter, absolute, and involution length maps; $\ellhat = \frac{1}{2}(\ell + \ell')$ & \S\ref{dem-subsect} \\
$\ell_0$ & Map $\W_n\to \NN$ with $w \mapsto | \{ i \in [n] : w(i)<0\}|$ & \S\ref{sogmed-sect}  \\
$\ell'_0$ & Map $\W_n\to \NN$ counting cycles preserved by negation & \S\ref{sogmed-sect} \\
\skip
$\Cyc_A(z)$ & The set $\{ (a,b) \in X\times X : a \leq b = z(a)\}$ for $z \in \I(S_X)$ \\
$\Pair(z)$ & The set $ \{ (a,b) \in [\pm n]\times [n] : |a| < z(a) = b\}$ for $z \in \I(\W_n)$ \\
$\Neg(z)$ & The set $\{ i \in [n] : z(i) = -i\}$ for $z \in \I(\W_n)$ \\ 
$\Fix(z)$ & The set $\{ i \in [n] : z(i) =i\}$ for $z \in \I(\W_n)$ \\
\skip
$\Des(w)$ & The set of pairs $(w_i,w_{i+1})$ with $w_i>w_{i+1}$ for a word $w$ & \S\ref{nest-sect}  \\
$\NDes(w)$ & Nested descent set of an inverse atom $w$ for $z \in \I(\W_n)$ & \eqref{ndes-eq} \\
$\NFix(w)$ & Nested fixed points of an inverse atom $w$ for $z \in \I(\W_n)$ & Def.~\ref{recdes-thmdef} \\
$\NNeg(w)$ & Nested negated points of an inverse atom $w$ for $z \in \I(\W_n)$ & Def.~\ref{recdes-thmdef} \\
\skip
$\NCSM(z)$ & Noncrossing symmetric perfect matchings in $\{ i   : z(i) = -i\}$ 
& \S\ref{shapes-sect} \\
$\sh(w)$ & Shape of an inverse atom $w \in \cA(z)^{-1}$ for $z\in \I(\W_n)$ & \S\ref{shapes-sect}  \\
$\Pair(z,M)$  & Variant of $\Pair(z,M)$ for $z\in \I(\W_n)$ and $M \in \NCSM(z)$  & \eqref{neg-pair-cyc-b-eq} \\
$\Neg(z,M)$  & Variant of $\Neg(z,M)$ for $z\in \I(\W_n)$ and $M \in \NCSM(z)$ & \eqref{neg-pair-cyc-b-eq} \\
$\Cyc_B(z,M)$  & A certain set of pairs for $z\in \I(\W_n)$ and $M \in \NCSM(z)$  & \eqref{neg-pair-cyc-b-eq} \\
\skip
$\vartriangleleft_A$, $\vartriangleleft_B$ & Covering relations for atomic orders
& \eqref{<A-eq} \\
$\vartriangleleft^+_B$, $\scov$, $\bcov$ & Stronger forms of $\vartriangleleft_B$ & \eqref{black},  \eqref{bcov-eq} \\ 
$<_A$ & The transitive closure of $\vartriangleleft_A$ \\
$<_B$, $\ll_B$ & Transitive closures of ($\vartriangleleft_A$ and $\vartriangleleft_B$) and ($\vartriangleleft_A$ and $\scov$) \\
$\lll_B$ & The transitive closure of $\vartriangleleft_A$, $\scov$, and $\bcov$ \\
$\sim_A$, $\sim_B$ & Symmetric closures of $<_A$ and $<_B$ \\
$\approx_A$, $\approx_B$ & Equivalence relations on inverse Hecke atoms & \S\ref{hecke-sect} \\
\skip
$0_A(z)$, $1_A(z)$  & Extremal inverse atoms of $z \in \I(S_n)$ under $<_A$ & \eqref{01A-eq} \\ 
$0_B(z,M)$  & Minimal inverse atom for $z \in \I(\W_n)$ under $<_A$  & \eqref{01B-eq} \\
$1_B(z,M)$  & Maximal inverse atom for $z \in \I(\W_n)$ under $<_A$  & \eqref{01B-eq} \\
$0_B(z)$ & Minimum inverse atom for $z \in \I(\W_n)$ under $<_B$ & \eqref{0B-eq} \\
$1_B(z)$ & Maximum inverse atom for $ \in \I(\W_n)$ under $\lll_B$ & \eqref{1B-eq} \\
\skip
$y^{\BC}_{k,n}$ & The signed involution $(-1,1)(-2,2)\cdots(-k,k) \in \I(\W_n)$ & \eqref{ybc-eq} \\
$\yfpf_{k,n}$ & When $n-k$ is even, this is $y^{\BC}_{n,k} \cdot t_{k+1}\cdot t_{k+3} \cdots t_{n-1} \in \I(\W_n)$ & \eqref{ybc-eq-fpf} 
\end{tabular}}
\end{center}

\end{document}